\documentclass[a4paper, 12pt]{amsart}
\usepackage[all]{xy}
\usepackage{hyperref}
\usepackage{amssymb, color, amsmath, comment, mathrsfs}
\usepackage[T1]{fontenc}
\usepackage{tikz, epsfig}
\usepackage[normalem]{ulem}
\usepackage{cancel}
\usepackage{epsfig}
\definecolor{darkgreen}{rgb}{0.008,0.417,0.067}

\def\smallspace{\negthickspace}
\numberwithin{equation}{section}

\newcounter{theorem}
\newtheorem{thm}[theorem]{Theorem}
\newtheorem*{thmd*}{Theorem~D\textsuperscript{$\prime$}}
\newtheorem{lemma}[theorem]{Lemma}
\newtheorem{prop}[theorem]{Proposition}
\newtheorem{cor}[theorem]{Corollary}

\theoremstyle{remark}
\newtheorem*{remark*}{Remark}
\newtheorem{remark}[theorem]{Remark}
\newtheorem{remarks}[theorem]{Remarks}
\newtheorem{example}[theorem]{Example}

\theoremstyle{definition}
\newtheorem{defn}[theorem]{Definition}

\newcommand{\Aut}{\operatorname{Aut}}
\newcommand{\image}{\operatorname{im}}
\newcommand{\id}{\operatorname{id}}
\newcommand{\ospan}{\overline{\operatorname{span}}}
\newcommand{\espan}{\operatorname{span}}
\newcommand{\diag}{\operatorname{diag}}
\newcommand{\Dom}{\operatorname{Dom}_{1/2}}
\newcommand{\supp}{\operatorname{supp}}
\newcommand{\class}[1]{{[#1]_\Lambda}}
\newcommand{\classg}[1]{{[#1]_\Gamma}}

\newcommand{\CC}{\mathbb{C}}
\newcommand{\NN}{\mathbb{N}}
\newcommand{\TT}{\mathbb{T}}
\newcommand{\ZZ}{\mathbb{Z}}
\newcommand{\RR}{\mathbb{R}}

\newcommand{\FF}{\mathbb{F}}

\newcommand{\Gg}{\mathcal{G}}
\newcommand{\Kk}{\mathcal{K}}
\newcommand{\Pp}{\mathcal{P}}

\newcommand{\Uu}{\mathcal{U}}
\newcommand{\Mm}{\mathcal{M}}

\date{\today}

\author{David Pask}
\email{dpask, asierako, asims@uow.edu.au}
\author{Adam Sierakowski}
\author{Aidan Sims}
\address{School of Mathematics and Applied Statistics \\
	University of Wollongong\\
	Wollongong NSW 2522\\
	AUSTRALIA}

\thanks{This research was supported by the Australian Research Council, grant number DP150101598. Part
of this work was completed while the third author was at the Centre de Recerca Mathematica, Universitat
Aut\'onoma de Barcelona as part of the Intensive Research Program \emph{Operator algebras: dynamics and
interactions} in 2017.}

\title[Quasitraces]{Unbounded quasitraces, stable finiteness and pure infiniteness}

\date{May 3, 2017}

\begin{document}

\begin{abstract}
We prove that if $A$ is a $\sigma$-unital exact $C^*$-algebra of real rank zero, then
every state on $K_0(A)$ is induced by a 2-quasitrace on $A$. This yields a generalisation
of Rainone's work on pure infiniteness and stable finiteness of crossed products to the
non-unital case. It also applies to $k$-graph algebras associated to row-finite
$k$-graphs with no sources. We show that for any $k$-graph whose $C^*$-algebra is unital
and simple, either every twisted $C^*$-algebra associated to that $k$-graph is stably
finite, or every twisted $C^*$-algebra associated to that $k$-graph is purely infinite.
Finally we provide sufficient and necessary conditions for a unital simple $k$-graph
algebra to be purely infinite in terms of the underlying $k$-graph.
\end{abstract}

\maketitle

\tableofcontents

\renewcommand*{\thetheorem}{\Alph{theorem}}
\section*{Introduction}
We study stable finiteness and pure infiniteness of $C^*$-algebras. In particular, we
identify a large class of non-unital $C^*$-algebras which we prove are either stably
finite or purely infinite. Our first purpose is to generalise two theorems of Rainone
(see \cite[Theorem~1.1, Theorem~1.2]{Rai}) from unital to $\sigma$-unital $C^*$-algebras.
To achieve this, we first extend to $\sigma$-unital $C^*$-algebras a theorem~of Blackadar
and R{\o}rdam \cite{MR1190414} concerning extending states on $K_0$ groups to quasitraces
on $C^*$-algebras. The second purpose of this paper is to exploit our generalisation of
Rainone's results to study stable finiteness and pure infiniteness of crossed product
$C^*$-algebras and of $C^*$-algebras of row-finite $k$-graphs.

We begin with some background to motivate our results. Blackadar and R{\o}rdam proved in
\cite{MR1190414} that for a unital $C^*$-algebra $A$, every state on $K_0(A)$ is induced
by a quasitrace on $A$. More specifically, they proved that for each homomorphism
$f\colon K_0(A)\to \RR$ that is positive in the sense that $f(K_0(A)^+)\subseteq
[0,\infty)$ and that satisfies $f([1_A]) = 1$, there exists a quasitrace $\tau$ on $A$
such that $\tau(p) = f([p])$ for each projection $p$.

The requirement that $A$ is unital combined with positivity of $f$ implies that the values
of any given
quasitrace $\tau$ as in the preceding paragraph on projections are bounded above.
When $A$ is non-unital this can fail, and it becomes
necessary to consider unbounded quasitraces (see for example \cite{Kir, MR2032998}).
Nevertheless, when $A$ is otherwise sufficiently well behaved one can still construct
unbounded quasitraces on $A$. For example, Blackadar and Cuntz proved that every simple,
stably projectionless, exact $C^*$-algebra admits a nontrivial (possibly unbounded)
trace, and that every simple stably finite $C^*$-algebra admits a nontrivial (possibly
unbounded) quasitrace, see \cite[Theorem~1.1.5]{MR1878881} or \cite{MR667536}. Our first
main result shows that Blackadar and R{\o}rdam's work in \cite{MR1190414} generalises to
arbitrary $\sigma$-unital exact $C^*$-algebras of real-rank zero (and so in particular
all AF algebras) if we allow for unbounded quasitraces.

\begin{thm}\label{thm1}
Let $A$ be a $\sigma$-unital exact $C^*$-algebra of real rank zero. Then each additive
positive map $\beta \colon K_0(A)\to \RR$ extends to a lower-semicontinuous 2-quasitrace
on $A$ (see Definition \ref{def.2.quasitrace} below).
\end{thm}

Tracelessness, or more generally the absence of nontrivial quasitraces, is often closely
related to pure infiniteness. In particular, for large classes of simple $C^*$-algebras
$A$, there is a dichotomy stating that either $A$ is purely infinite or it admits a
quasitrace. For example, a recent result of Rainone \cite[Theorem~1.1]{Rai} provides such
a characterisation for a substantial class of unital simple crossed product
$C^*$-algebras. Using Theorem~\ref{thm1}, we are able to generalise Rainone's theorem to
non-unital $C^*$-algebras:
\begin{thm}\label{thm2}
Let $A$ be a separable exact $C^*$-algebra of real rank zero and with cancellation. Let
$\sigma : G \to \Aut(A)$ be a minimal and properly outer action of a discrete group.
Suppose that the semigroup $S(A, G, \sigma)$ of equivalence classes of positive elements
of $K_0(A)$ by the equivalence relation induced by $\sigma$ (Definition \ref{def.S.A.G})
is almost unperforated (Definition \ref{alm.unper}). Then the following are equivalent:
\begin{enumerate}
\item $A \rtimes_{\sigma, r} G$ is purely infinite;
\item $A \rtimes_{\sigma, r} G$ is traceless; and
\item $A \rtimes_{\sigma, r} G$ is not stably finite.
\end{enumerate}
\end{thm}

Rainone also considered stable finiteness of non-simple $C^*$-algebras in
\cite[Theorem~1.2]{Rai}. We provide a generalisation of his result to the $\sigma$-unital
case:

\begin{thm}\label{thm2b}
Let $A$ be a $\sigma$-unital exact $C^*$-algebra of real rank zero and with cancellation.
Let $\sigma : G \to \Aut(A)$ be a minimal action of a discrete group. Then the following are equivalent:
\begin{enumerate}
\item $A\rtimes_{\sigma,r} G$ admits a faithful lower-semicontinuous semifinite
    tracial weight;
\item $A\rtimes_{\sigma,r} G$ is stably finite; and
\item The action $\sigma$ is completely non-paradoxical.
\end{enumerate}
\end{thm}

Our main application of these results is to twisted higher-rank graph $C^*$-algebras  as
defined in \cite{MR3335414}. These $C^*$-algebras are always generated by projections and
partial isometries subject to relations encoded by the underlying $k$-graph $\Lambda$ and
a $2$-cocycle $c \in Z^2 ( \Lambda , \TT )$. Despite being such concrete examples it is
far from clear when these $C^*$-algebras are purely infinite. Inspired by
Theorem~\ref{thm2} we equip each $k$-graph $\Lambda$ with a semigroup $S(\Lambda)$ of
equivalence classes of finitely supported functions $f\colon \Lambda^0\to \NN$ by an
equivalence relation induced by $\Lambda$ (Definition \ref{def.S.Lambda}). We prove that
this $S(\Lambda)$ is isomorphic to the type semigroup of an associated $C^*$-dynamical
system as introduced by Rainone. When $S(\Lambda)$ is almost unperforated, we obtain a
partial answer to R{\o}rdam's \cite[Question 2.6]{MR2275660} on whether all simple
$C^*$-algebras of real rank zero are either stably finite or purely infinite.

\begin{thm}\label{thm3}
Let $\Lambda$ be a row-finite $k$-graph with no sources such that $C^*(\Lambda)$ is
simple and the semigroup $S(\Lambda)$ (Definition \ref{def.S.Lambda}) is almost
unperforated. Then $C^*(\Lambda)$ is either stably finite or purely infinite.
\end{thm}

In fact, we strengthen this result to deal with twisted $k$-graph algebras as well in
Theorem~D\textsuperscript{$\prime$} (page~\pageref{thm.d*}).

Recent work \cite{MR3507995} of Clark, an Huef and the third-named author characterises
when a $C^*$-algebra of a row-finite and cofinal $k$-graph with no sources is stably
finite. In \cite[Theorem~1.1.(1)(c)]{MR3507995} they present a technical condition
$(\sum_{i=1}^k\image(1-A^t_{e_i})) \cap \NN{\Lambda^0} = \{0\}$ that is necessary and
sufficient to ensure $C^*(\Lambda)$ is stably finite, quasidiagonal and admits a faithful
graph trace. Using our Theorem~\ref{thm2b} we extend this to the more general class of
twisted $k$-graph algebras.

\begin{thm}\label{thm.stably.finite}
Let $\Lambda$ be a row-finite and cofinal $k$-graph with no sources. Denote the
coordinate matrices of $\Lambda$ by $A_{e_1}, \dots , A_{e_k}$. Let $c$ be a 2-cocycle on
$\Lambda$. Then the following are equivalent:
\begin{enumerate}
\item\label{thm.stably.finite.(a)} $C^*(\Lambda,c)$ is quasidiagonal;
\item\label{thm.stably.finite.(b)} $C^*(\Lambda,c)$ stably finite;
\item\label{thm.stably.finite.(c)} $\big(\sum_{i=1}^k\image(1-A^t_{e_i}) \big) \cap \NN{\Lambda^0} = \{0\}$; and
\item\label{thm.stably.finite.(d)} $\Lambda$ admits a faithful graph trace.
\end{enumerate}
\end{thm}

One interesting consequence of Theorem~\ref{thm.stably.finite} is that for any two
cofinal $k$-graphs $\Lambda_1$ and $\Lambda_2$ with the same skeleton, and any $2$-cocycles $c_i$
on $\Lambda_i$, the $C^*$-algebra $C^*(\Lambda_1, c_1)$ is stably finite if and only if
$C^*(\Lambda_2, c_2)$ is stably finite; that is, stable finiteness is independent of the
factorisation rules and the twisting cocycle.

The question of when a $k$-graph algebra is purely infinite remains an intriguing open
problem, even within the class of simple $C^*$-algebras. There are many partial results, most
notably in \cite{MR2920846, BroClaSie}. In this paper we characterise when the following simple
twisted $k$-graph $C^*$-algebras are purely infinite.

\begin{thm}\label{thm4}
Let $\Lambda$ be a row-finite $k$-graph with no sources. Denote the
coordinate matrices of $\Lambda$ by $A_{e_1}, \dots , A_{e_k}$. Let $c$ be a 2-cocycle on
$\Lambda$.  Suppose that $C^*(\Lambda)$, and
hence also $C^*(\Lambda,c)$, is simple. Finally suppose that $\Lambda^0$ contains a non-empty
hereditary subset $H$ such that $H\Lambda$ is strongly connected or $H$ is finite. Then
the following are equivalent:
\begin{enumerate}
\item $C^*(\Lambda, c)$ is purely infinite;
\item $C^*(\Lambda, c)$ is traceless;
\item $C^*(\Lambda, c)$ is not stably finite; and
\item $\big(\sum_{i=1}^k\image(1-A^t_{e_i}) \big) \cap \NN{\Lambda^0} \neq \{0\}$.
\end{enumerate}
\end{thm}

We deduce in Corollary \ref{cor.unital.case} that if $\Lambda$
is a row-finite $k$-graph with no sources such that $C^*(\Lambda)$ is unital and simple,
then $C^*(\Lambda)$ is purely infinite if and only if $(\sum_{i=1}^k\image(1-A^t_{e_i})) \cap
\NN{\Lambda^0}\neq \{0\}$.

We study in more detail the semigroup $S(\Lambda)$ associated to a $k$-graph $\Lambda$.
In Propositions~\ref{prop.properties.dichotomy.merged},
\ref{prop.properties.dichotomy.three} and \ref{prop.properties.dichotomy.three.part.two}
we consider sufficient conditions ensuring that $S(\Lambda)$ is stably finite or purely
infinite. Then in  Proposition~\ref{prop.properties.dichotomy.four} we show that
$S(\Lambda)$ is stably finite if and only if $(\sum_{i=1}^k\image(1-A^t_{e_i})) \cap
\NN{\Lambda^0} = \{0\}$. This confirms the importance of the semigroup $S(\Lambda)$ in
the study of $k$-graph algebras and connects our work on semigroups with the results in
\cite{MR3507995}.

\renewcommand*{\thetheorem}{\roman{theorem}}
\numberwithin{theorem}{section}

\section{Extending additive positive maps on \texorpdfstring{$K_0(A)$}{K0(A)} and the proof of Theorem~\ref{thm1}}\label{section1}

We begin with some background material on $K_0$-groups and quasitraces which will be
needed for the proof of Theorem~\ref{thm1}.

Let $A$ be a $C^*$-algebra and let $A_+$ be the positive cone of $A$. Let $\Pp(A)$ denote
the set of projections in $A$, let $M_{m,n}(A)$ denote the space of $m\times n$ matrices
over $A$, let $\Pp_n(A)$ denote the set of projections in $M_n(A)$. Let $M_\infty(A)$
denote the union $\bigcup_n M_n(A)$, where we include $M_n(A)$ in $M_{n+1}(A)$ via
$a\mapsto \diag(a,0)$, and let $\Pp_\infty(A)$ denote the set of projections in
$M_\infty(A)$. Extending von Neumann equivalence, two projections $p\in \Pp_n(A)$ and
$q\in \Pp_m(A)$ are \emph{equivalent}, denoted $p\sim q$, if there exists $v\in
M_{m,n}(A)$ such that $p=v^*v$ and $q=vv^*$. The image in $K_0(A)$ of the equivalence class of $p$ in
$\Pp_\infty(A)/\smallspace\sim$ is denoted $[p]$ (it is often denoted $[p]_0$ in the
literature). Clearly $p\sim q$ implies $[p]=[q]$ for $p,q\in \Pp_\infty(A)$. We say
$K_0(A)$ (or $A$)  has \emph{cancellation} if $[p]=[q]$ implies $p\sim q$ for each
$p,q\in \Pp_\infty(A)$. We follow here \cite{MR3507995,Rai,MR2059808}, but refer to
\cite[V.2.4.13]{MR2188261} for other notions of cancellation in the non-unital case. We
refer the reader to \cite[V.1.1.18]{MR2188261} and \cite{MR1783408} for the definition of
$K_0(A)$ and the fact that
\begin{eqnarray}\label{eqn.K0A}
K_0(A)=\{[p]-[q]: p,q\in \Pp_\infty(A)\}
\end{eqnarray}
 if $A$ is unital, or more generally if the stabilisation of $A$ admits an approximate unit of projections. Addition on $K_0(A)$ is defined by $[p]+[q]=[p\oplus q]$, where $p\oplus q$ is the matrix $\diag(p,q)$. The \emph{positive cone} of $K_0(A)$ is given by $K_0(A)^+=\{[p]: p\in \Pp_\infty(A)\}$. A homomorphism $\beta\colon K_0(A)\to \RR$ is \emph{positive} if $\beta(K_0(A)^+)\subseteq [0,\infty)$. Following \cite{MR2032998}, a \emph{quasitrace} on a $C^*$-algebra $A$ is a function $\tau\colon A_+\to [0,\infty]=\RR_+\cup\{\infty\}$ such that $\tau(a+b) = \tau(a)+\tau(b)$ for all commuting elements of $a,b\in A_+$, and satisfying the \emph{trace property}, i.e., $\tau(d^*d) = \tau(dd^*)$ for all $d \in A$. If $a \in A_+$ and $\varepsilon>0$ then $(a-\varepsilon)_+$, the positive part of $a - \varepsilon1$ in $\Mm(A)$, is again in $A_+$ where $\Mm(A)$ is the multiplier algebra of $A$. Matrix units for $M_n(\CC)$ will be denoted  $(e_{ij})$. The compact operators acting on an infinite dimensional Hilbert space will be denoted $\Kk$. By an ideal of a $C^*$-algebra we always mean a closed and two-sided ideal.

\begin{defn}({\cite[Definition 2.22]{MR2032998}})\label{def.2.quasitrace}
Let $A$ be a $C^*$-algebra. A quasitrace on $A$ is said to be
\begin{enumerate}
\item a \emph{2-quasitrace} if it extends to a quasitrace $\tau_2$ on
    $M_2(A)\cong A\otimes M_2(\CC)$ such that $\tau_2(a\otimes e_{11})=\tau(a)$ for
    all $a\in A_+$,
\item \emph{trivial} if it only takes values $0$ or $\infty$, and
    \emph{nontrivial} otherwise,
\item \emph{locally lower semicontinuous} if $\tau(a)=
    \sup_{t>0}\tau((a-t)_+)$ for all $a\in A_+$,
\item \emph{lower semicontinuous} if $\tau(a)\leq \lim \inf_n \tau(a_n)$ for
    every sequence $a_n\to a$ in $A_+$,
\item \emph{semifinite} if the set $\Dom(\tau)=\{a\in A : \tau(a^*a)<\infty\}$
    is dense in $A$, and
\item \emph{faithful} if $I_\tau:=\{a\in A : \tau(a^*a)=0\}=\{0\}$.
\end{enumerate}
\end{defn}

\begin{lemma}\label{lem.extend.beta}
Let $A$ be an exact $C^*$-algebra of real rank zero with an increasing sequence of
projections $(p_n)$ constituting an approximate unit for $A$. Suppose $\beta \colon
K_0(A)\to \RR$ is an additive positive map. Then there exists a linear positive map
$\tau\colon \bigcup_nM_\infty(p_nAp_n)\to \CC$ with the trace property such that
$$\tau(p)=\beta([p]), \textrm{ for each }p\in \bigcup_n\Pp_\infty(p_nAp_n).$$
\end{lemma}

\begin{proof}
We can assume that $\beta$ is nontrivial (otherwise $\tau=0$ suffices). For each $n$ set
$A_n:=p_nAp_n$. Since $(p_n)$ is an increasing approximate unit for $A$ we have
$\underrightarrow{\lim} (A_n, \id) = A$, we write $\id_n\colon A_n\to A$ for the
canonical inclusion map. By continuity of $K_0$ we have
\begin{enumerate}
    \item\label{it:K0lim} $K_0(A)=\bigcup_n K_0(\id_n)(K_0(A_n))$, and
    \item\label{it:K0+lim} $K_0(A)^+=\bigcup_n K_0(\id_n)(K_0(A_n)^+)$
\end{enumerate}
(see the proof of \cite[Theorem~6.3.2]{MR1783408}).

For each $m\in \NN$ we define $\beta_m\colon K_0(A_m)\to \RR$ by $\beta_m = \beta \circ
K_0(\id_m)$. So $\beta_m$ is the restriction of $\beta$ to $K_0(A_m)$. Since $\beta$ is
additive and positive (by assumption), so is each $\beta_m$.

Since $\beta$ is nontrivial, property~(\ref{it:K0+lim}) gives $k \in \NN$ and $x\in
K_0(A_k)^+$ such that $\beta_k(x)\neq 0$. Since $\beta_k$ is additive and positive,
$\beta_k([p_k])$ is nonzero. Since $p_n \geq p_k$ for $n \ge k$ and since $\beta$ is
additive and positive, we have $\beta_n([p_n]) \ge \beta_n([p_k]) = \beta_k([p_k]) > 0$
for all $n \ge k$. So by replacing the increasing approximate unit $(p_n)^\infty_{n=1}$
with the approximate unit $(p_n)_{n=k}^\infty$ (and reindexing) we can assume without
loss of generality that $\beta_n([p_n]) > 0$ for all $n$.

For any tracial state $\tau'_n$ on $A_n$, and any $k \in \NN$ we define $\tau'_{k,n} \colon M_k(A_n)\to \CC$ by $$\tau'_{k,n}((a_{i j}))=\sum_{m=1}^k\tau'_n(a_{m m}).$$ We have $\tau'_{k,n}(a) = \tau'_{k+1,n}(\diag(a,0))$ for all $a\in M_k(A_n)$. The maps $\tau'_{k,n}$ ($k\in\NN$) are then compatible with respect to the embeddings $M_k(A_n) \to M_{k+1}(A_n)$ given by $a\mapsto \diag(a,0)$. Suppressing the embeddings we obtain an extension $\tau'_{n,\infty}\colon M_\infty(A_n)\to \CC$ of $\tau'_n\colon A_n\to \CC$ such that $\tau'_{n,\infty}|_{M_k(A_n)} = \tau'_{n,k}$ for all $k$.

Define $\beta'_n = \beta_n/\beta_n([p_n])$ for each $n\in \NN$. We get $\beta'_n([p_n]) =
1$. Since exactness passes to $C^*$-subalgebras and $A$ is exact, each $A_n$ is exact.
We can therefore apply \cite[Corollary 3.4]{MR1190414} to find a tracial state $\tau'_n
\colon A_n \to \CC$ such that the extension $\tau'_{n,\infty}$ satisfies $\tau'_{n,\infty}(p) = \beta'_n([p])$ for each $p\in
\Pp_\infty(A_n)$. For each $n$ define $\tau_n\colon A_n\to \CC$ by $\tau_n:= \beta([p_n])\cdot \tau'_{n,\infty}$.
The maps $\tau_n$ agree on projections since
\[
\tau_m(p)
    =\beta([p_m])\cdot \tau'_{m,\infty}(p)
    =\beta([p_m])\beta'_m([p])
    =\beta_m([p])
    =\beta([p])
    =\dots
    =\tau_n(p)
\]
for any projection $p\in A_n\subseteq A_m$. Since $A$ has real rank zero and since real
rank zero is inherited by corners, each $A_n$ has real rank zero (see
\cite[Theorem~2.5]{MR1120918}). In particular for any element $a\in A_n\subseteq A_m$ we
can find a sequence $a_n=\sum_i\lambda_{i,n}p_{i,n}$ in the linear span of projections of
$A_n$ such that $(a_n)$ converges to $a$. Since both $\tau_n$ and $\tau_m$ are continuous
we get $\tau_n(a)=\lim_n \sum_i\lambda_{i,n}\tau_n(p_{i,n})=\tau_m(a)$. In particular we
obtain a well-defined map $\tau\colon \bigcup_n A_n\to \CC$ given by $\tau(a)=\tau_n(a)$
for $a\in A_n$. Since each map $\tau_n$ is linear, positive and satisfies the trace
property, the same holds for $\tau$. Since each $M_k(A_n)$ has real rank zero, the above
argument applied to the $M_k(A_n)$ gives a positive linear map $\tau\colon \bigcup_n
M_\infty(A_n)\to \CC$ with the trace property such that
$\tau(a)=\tau_{k,n}(a)=\beta([p_n])\tau'_{k,n}(a)$ for $a\in M_k(A_n)$.
We have $\tau(p)=\tau_{k,n}(p)$ for each projection $p\in M_k(A_n)$, and also
$\beta([p])=\beta_n([p_n])\beta'_n([p])=\beta([p_n])\tau'_{k,n}(p)=\tau_{k,n}(p)$.
Hence $\tau(p)=\beta([p])$ on projections in $\bigcup_nM_\infty(A_n)$.
\end{proof}

Following the notation of Dixmier \cite[p.~77]{MR641217}, let $\Uu$ be a $^*$-algebra
over $\CC$ endowed with an inner product $(x,y)$ which makes it a Hausdorff pre-Hilbert
space. Let $H$ be the corresponding completion of $\Uu$. We say $\Uu$ is a \emph{Hilbert
algebra} if the following axioms are satisfied:
\begin{enumerate}
	\item\label{H.algebra.i}$(x, y)=(y^*,x^*)$ for $x,y\in \Uu$;
	\item\label{H.algebra.ii}$(xy, z)=(y,x^*z)$ for $x,y,z\in \Uu$;
	\item\label{H.algebra.iii}for each $x\in \Uu$ the map $y\to xy$ is continuous;
	\item\label{H.algebra.iv}$\espan\{xy : x,y \in \Uu\}$ is dense in $\Uu$; and
	\item\label{H.algebra.v}if $a,b \in \Uu$ and $(a,xy)=(b,xy)$ for all $x,y\in \Uu$, then $a=b$.
\end{enumerate}

\begin{lemma}\label{lem.hilbert.alg}
Let $U$ be the union of an increasing sequence of $C^*$-algebras $A_1\subseteq
A_2\dots\,$. Let $\tau\colon U\to \CC$ be a positive linear map with the trace property.
Then $I:=\{x\in U : \tau(x^*x)=0\}$ is an ideal of $U$ and $\Uu := U/I$ is a
$^*$-algebra. There is a sesquilinear form $(\cdot, \cdot)$ on $\Uu$ such that $(x+I,y+I)
= \tau(y^*x)$ for $x,y\in U$, and under this form, $\Uu$ is a Hilbert algebra.
\end{lemma}
\begin{proof}
For $x,y\in I$ we have $(x+y)^*(x+y)\leq 2(x^*x+y^*y)$ by \cite[II.3.1.9]{MR2188261}.
Hence $0\leq \tau((x+y)^*(x+y))\leq 2\tau(x^*x)+2\tau(y^*y)=0$. So $x+y\in I$. By the
trace property, $I$ is closed under adjoints. Since
$\tau(y^*x^*xy)\leq\|x\|_{A_n}^2\tau(y^*y)$ for $x,y\in A_n$, $I$ is a $^*$-algebra and a
two sided ideal of $U$ (being $^*$-closed and one sided). So $\Uu=U/I$, equipped with the
canonical operations, is a $^*$-algebra.

To see that $(\cdot, \cdot)$ is well-defined, fix $x\in U$. For $y_1-y_2\in I$ we have
$\tau(x^*y_1) - \tau(x^*y_2) = \tau(x^*(y_1 - y_2))$. The Cauchy--Schwarz inequality for
the form $(x,y) \mapsto \tau(y^*x)$ therefore gives $|\tau(x^*y_1) - \tau(x^*y_2)|^2 \le
\tau(x^*x)\tau((y_1-y_2)^*(y_1-y_2))$, which is zero since $y_1 - y_2 \in I$. Taking
conjugates gives $(x+I, y_1+I)=(x+I, y_2+I)$ as well. Now if $x_1 + I = x_2 + I$ and $y_1
+ I = y_2 + I$, then $(x_1+I, y_1+I) = (x_1+I, y_2+I) = (x_2+I, y_2+I)$.

The form $(\cdot, \cdot)$ is clearly a sesquilinear form from $\Uu$ to $\CC$. It is
positive definite by definition of $I$. To see that it is an inner product, fix $x,y
\in \Uu$ and calculate:
\begin{align*}
(x+I,y+I)
    &=\tau(\Re{e}(y^*x))-i\tau(\Im{m}(x^*y))\\
    &=\overline{\tau(\Re{e}(x^*y))+i\tau(\Im{m}(x^*y))}
    =\overline{(y+I,x+I)}.
\end{align*}

Let $H$ denote the Hilbert-space completion of $\Uu$ in $(\cdot, \cdot)$, and continue to
write $(\cdot, \cdot)$ for the extension of the existing form to $H$. We must check that
$H$ satisfies axioms \eqref{H.algebra.i}--\eqref{H.algebra.v} for a Hilbert algebra.

Axiom~\eqref{H.algebra.i} follows directly from the trace property, and~\eqref{H.algebra.ii} is immediate from the
definition of $(\cdot, \cdot)$. To verify~\eqref{H.algebra.iii} take $x,y\in A_n$. Since $A_n$ is a
$C^*$-algebra we have $y^*x^*xy\leq y^*\|x\|_{A_n}^2y$, so we get
\[
\|xy+I\|^2_{\Uu}
    =(xy+I,xy+I)
    =\tau(y^*x^*xy)\leq\|x\|_{A_n}^2\tau(y^*y)
    =\|x\|_{A_n}^2\|y+I\|_{\Uu}^2.
\]
In particular for $x\in A_m$, and $y,y_n\in U$ we see that
$$\|(xy+I)-(xy_n+I)\|_{\Uu}\leq \|x\|_{A_m}\|(y+I)-(y_n+I)\|_{\Uu}.$$
Thus the map $y+I\to xy+I$ is continuous, which is~\eqref{H.algebra.iii}. To verify~\eqref{H.algebra.iv} recall that
\cite[II.3.2.3]{MR2188261} implies that for each $x\in A_n$ there exists $y\in A_n$  such
that $x=yy^*y$. In particular $\{xy : x,y\in\Uu\}$ spans a dense subspace of $\Uu$.

Finally we verify~\eqref{H.algebra.v}. Fix $a,b\in H$ satisfying $(a,xy)=(b,xy)$ for all $x,y\in
\Uu$. Use~(iv) to find $x_n, y_n \in \Uu$ with $x_ny_n \to a-b$. Then
\[
\|a-b\|_H^2
    = (a-b, a-b)
    =\lim_n (a-b, x_ny_n)
    = \lim_n (a, x_ny_n) - (b, x_ny_n)
    = 0.\qedhere
\]
\end{proof}

Let $\Uu$ be a Hilbert algebra, and $H$ the Hilbert space completion of $\Uu$. As
discussed on \cite[p.~78]{MR641217}, formulas $L_x y = xy$ and $R_x y = yx$ for $y \in \Uu$
(uniquely) determine operators $L_x, R_x \in B(H)$. Let $L(\Uu)$ denote
the von Neumann algebra in $B(H)$ generated by $\{L_x: x\in \Uu\}$. We say that $h \in H$
is \emph{bounded} if there exists $U_h \in B(H)$ satisfying $U_h x = R_x h$ for every
$x\in \Uu$; so every $y \in \Uu$ is bounded with $U_y = L_y$.

\begin{thm}[{\cite[Theorem~1, p.~97]{MR641217}}]\label{thm.Dix}
Let $\Uu$ be a Hilbert algebra, and $H$ the completion of $\Uu$. Define
$\phi\colon L(\Uu)_+\to [0,\infty]$ on the positive elements of $L(\Uu)$ as follows.
\begin{equation*}\label{xn}
\phi(S)=\left\{
\begin{array}{ll}
(h,h)&\mbox{ if } S^{1/2}=U_h\mbox{ for some bounded }h\in H\\
+\infty&\mbox{ otherwise. }\\
\end{array}\right.
\end{equation*}
Then, $\phi$ is a faithful, semifinite normal trace on $L(\Uu)_+$.
\end{thm}

\begin{remark}[{\cite{MR2059808}}]\label{rem.rordam}
A $C^*$-algebra $A$ is said to be $\sigma$-unital if it contains a countable approximate
unit. It is  called $\sigma_p$-unital if it contains a countable approximate unit
consisting of  projections. If $A$ is $\sigma_p$-unital, then it has an increasing
approximate unit $(p_n)$ of projections, and then the $p_n$ eventually dominate any fixed
projection in $A$.
\end{remark}

\begin{proof}[Proof of Theorem~\ref{thm1}]
Fix an increasing approximate unit $(p_n)$ for $A$ consisting of projections $(p_n)$. For
each $n$ define $A_n:=M_n(p_nAp_n)$, and set $U:=\bigcup_nA_n$. Notice
$U=\bigcup_nM_\infty(p_nAp_n)$ since $M_k(p_nAp_n)\subseteq A_{\max(k,n)}$ for each
$k,n\in\NN$. Lemma~\ref{lem.extend.beta} implies that $\beta$ extends to a linear
positive map $\tau\colon U\to \CC$ with the trace property. With $I:=\{x\in U :
\tau(x^*x)=0\}$, $\Uu:=U/I$, and $(x+I,y+I):=\tau(y^*x)$ for $x,y\in U$, it follows from
Lemma~\ref{lem.hilbert.alg} that $\Uu$ is a Hilbert algebra. Let $H$ denote the
completion of $\Uu$. Theorem~\ref{thm.Dix} yields a trace $\phi : L(\Uu)_+ \to [0,\infty]$
that is normal in the
sense that for every increasing filtering family $F\subseteq L(\Uu)_+$ with supremum $y\in L(\Uu)_+$,
we have $\phi(y) = \sup_{x \in F} \phi(x)$ (see \cite[Definition 1, p.~56]{MR641217}).

Fix $a\in U$, and define $\pi(a) := L_{a + I} \colon H \to H$. Clearly $\pi(ab) =
\pi(a)\pi(b)$, and for $x,y \in U$, we have $\big(\pi(a^*) (x+I), y+I\big) = \tau(y^* a^*
x) = \tau((ay)^* x) = \big(x + I, \pi(a)(y+I)\big)$. We have $\|\pi(a) (x+I)\|^2 =
\tau(x^*a^*ax) \le\|a\|^2 \tau(x^*x) \le \|a\|^2 \|x + I\|^2$. So $\pi$ extends to a
representation $\pi\colon A\otimes \Kk \to B(H)$. Since $\pi(a) = L_{a + I} \in L(\Uu)$,
we have $\pi(A\otimes\Kk) \subseteq \pi(\Uu)'' = L(\Uu)$. So there exists a map
$\tau'\colon (A\otimes\Kk)_+\to [0,\infty]$ given by the composition of the trace $\phi$
with the representation $\pi$. This $\tau'$ is additive because $\pi$ and $\phi$ are. It
has the trace property because $\pi$ is multiplicative and $\phi$ has the trace property.
So $\tau'$ is a quasitrace in the sense of \cite{MR2032998} (see
Definition~\ref{def.2.quasitrace}). It is also locally lower semicontinuous, i.e.,
$\tau'(a)= \sup_{t>0}\tau'((a-t)_+)$, because $\phi$ is normal and $\{(a-t)_+: t>0\}$ is a increasing filtering family in $A_+$ with supremum $a$.
So \cite[Proposition~2.24(i)--(iii)]{MR2032998} shows that $\tau'$ is a
lower-semicontinuous 2-quasitrace on $A\otimes \Kk$.

It remains to check that $\tau'$ extends $\beta$. Take any positive $a\in U$. Then
$h:=a+I\in\Uu$ is bounded since for $U_h:=\pi(a)\in B(H)$ and for each $x\in \Uu$ we have
$U_hx=hx=R_xh$. Similarly $a^{1/2}+I$ is bounded. With
$S:=U_h=U_{a^{1/2}+I}U_{a^{1/2}+I}$ we get $S^{1/2}=U_{a^{1/2}+I}$, so
\begin{equation}\label{eqn.tau}
\tau'(a)=\phi(\pi(a))=\phi(U_h)=\phi(S)=(a^{1/2}+I,a^{1/2}+I)=\tau(a).
\end{equation}
Fix $p\in \Pp_\infty(A)$.  By Remark \ref{rem.rordam}, $p\in P_\infty(A_n)$ for some $n$.
Using  Lemma~\ref{lem.extend.beta} we get $\tau(p)=\beta([p])$, and from equation
\eqref{eqn.tau}, $\tau'(p)=\tau(p)$. Hence $\tau'(p)=\beta([p])$.
\end{proof}

\begin{remarks}\label{remark.thma}
Let $\tau'$ denote the 2-quasitrace obtained from Theorem~\ref{thm1}.
\begin{enumerate}
\item\label{remark.thma.i} The 2-quasitrace $\tau'$ satisfies
\begin{equation}\label{trace.add}
\tau'(\alpha a + \beta b) = \alpha \tau'(a) + \beta \tau'(b)
\end{equation}
for all $a,b\in A_+$, and all $\alpha,\beta \in \RR^+$. This is the because the map
$\phi$ in Theorem~\ref{thm.Dix} is a trace in the sense of \cite{MR641217} and this is
one of the axioms of such a trace.

\item\label{remark.thma.ii} The 2-quasitrace  $\tau'$ is semifinite because for each $a\in\bigcup_np_nAp_n$, the
trace from Lemma~\ref{lem.extend.beta} satisfies $\tau(a^*a)< \infty$. So \eqref{eqn.tau}
implies that $\tau'(a^*a)=\tau(a^*a)< \infty$.

\item\label{remark.thma.iii} The 2-quasitrace  $\tau'$ is not necessarily bounded on $\Pp(A)$ as shown by the
following example. For $n\geq 1$, let $A_n=M_{2^n-1}(\CC)$, and define $\phi_n\colon A_n\to A_{n+1}$ by $a\mapsto \diag(a,a,0)$. Let $\tau_n\colon A_n \to \CC$ be the matrix trace so $\tau_n(1)=2^n-1$. There is a 2-quasitrace $\tau$ on $A=\varinjlim(A_n,\phi_n)$ such that $\tau(p)=\tau_n(p)$ for all $p\in \Pp(A_n)$. It is clearly unbounded on $\Pp(A)$.

\item\label{remark.thma.iv} Our construction of $\tau'$ depends on von-Neumann-algebra results of
\cite{MR641217}. In the separable case this might not have been necessary if we instead applied the work of
Winter and Zacharias (see \cite[Corollary 5.2]{MR2609012}). But our approach using
Dixmier's results gives us \eqref{trace.add} and semifiniteness (see (i) and (ii)) with
no extra work and we will need both of these properties later on.
\end{enumerate}
\end{remarks}

\section{Dichotomy for crossed products and the proofs of Theorems~\ref{thm2} and \ref{thm2b}}
\label{section.two} Given a $C^*$-dynamical system $(A,G,\sigma)$ in which $G$ is
discrete, let $A\rtimes_{\sigma,r} G$ and $A\rtimes_{\sigma} G$ denote the reduced and
the full crossed product $C^*$-algebras, respectively. Applying $\sigma$ entrywise on
$M_n(A)$ gives an action of $G$ on $M_n(A)$ which is also denoted $\sigma$. Let
$\tilde\sigma$ be the action of $G$ on $K_0(A)^+$ given by
$\tilde\sigma_t([p])=[\sigma_t(p)]$. For $x,y\in K_0(A)^+$ write $x\sim_\sigma y$ if
there exist $u_1,\dots,u_n\in K_0(A)^+$ and $t_1,\dots,t_n\in G$ such that $x=\sum_iu_i$
and $y=\sum_i\tilde\sigma_{t_i}(u_i)$ for some $n\geq 1$. Recall that $K_0(A)^+$
satisfies the \emph{Riesz refinement property} if  whenever $a,b,c,d\in K_0(A)^+$ satisfy
$a+b=c+d$  then there exist $x,y,z,t\in K_0(A)^+$ such that $a=x+y$,  $b=z+t$,  $c=x+z$
and $d=y+t$. It is helpful to visualise the last four equations as a \emph{refinement
matrix}
$$\label{ref.matrix}\bordermatrix{
 &c&d\cr  a&x&y\cr  b&z&t\cr}.$$
If $K_0(A)^+$ satisfies the Riesz refinement property, then $\sim_\sigma$ is
an equivalence relation and $K_0(A)^+/\smallspace\sim_{\sigma}$ becomes an abelian semigroup
with identity $[0]_\sigma$ and addition $[x]_\sigma+[y]_\sigma:=[x+y]_\sigma$ for $x,y\in K_0(A)^+$.
Equip $K_0(A)^+/\smallspace\sim_{\sigma}$ with the algebraic order; that is
$[x]_\sigma \leq [y]_\sigma$ if $[x]_\sigma+[x']_\sigma=[y]_\sigma$ for some $x'\in
K_0(A)^+$. Then $K_0(A)^+/\smallspace\sim_{\sigma}$ is a preordered semigroup with identity. The quotient
map $\rho\colon K_0(A)^+\to K_0(A)^+/\smallspace\sim_{\sigma}$ respects $+$ and $\leq$
and the zero element.

\begin{defn}({\cite[Definition 4.8]{Rai}})\label{def.S.A.G}
Let $(A,G,\sigma)$ be a $C^*$-dynamical in which $G$ is discrete such that $K_0(A)^+$
satisfies the Riesz refinement property. We denote by $S(A,G,\sigma)$ the preordered
semigroup $K_0(A)^+/\smallspace\sim_{\sigma}$ and call it the \emph{semigroup for
$(A,G,\sigma)$}.
\end{defn}

The semigroup $S(A, G, \sigma)$ is sometimes called the \emph{type semigroup} of the $C^*$-dynamical system $(A,G,\sigma)$, see \cite{Rai}.

Let $A$ be a $C^*$-algebra. A nonzero projection in $A$ is \emph{properly infinite} if it
is von Neumann equivalent to each of two orthogonal subprojections of itself. A
projection in $A$ is \emph{infinite} if it is equivalent to a proper subprojection of
itself. A projection that is not infinite is called \emph{finite}. The $C^*$-algebra $A$
is \emph{infinite} if it contains an infinite projection, and is called \emph{finite} if
it admits an approximate unit of projections and all projections in $A$ are finite. The $C^*$-algebra $A$ is  \emph{stably finite} if $A\otimes \Kk$ is finite (\cite[p.~7]{MR1878881}). We say
that $A$ is \emph{traceless} if it admits no nontrivial lower-semicontinuous 2-quasitrace
(see \cite[p.~463]{MR2032998}). A simple $C^*$-algebra $A$ is \emph{purely infinite} if
every nonzero hereditary $C^*$-subalgebra of $A$ contains an infinite projection. For
the definition when $A$ is non-simple we refer the reader to \cite{MR1759891}. See
Section~\ref{section1} for the notion of cancellation. For positive elements $a\in
M_n(A)$ and $b\in M_m(A)$, $a$ is \emph{Cuntz below} $b$, denoted $a\precsim b$, if there
exists a sequence of elements $x_k$ in $M_{m,n}(A)$ such that $x_k^* b x_k\to a$ in norm.

Let $S$ be a preordered semigroup $S$ with identity. Then $S$ is said to \emph{purely
infinite} if $2x\leq x$ for every $x\in S$. A \emph{state} on $S$ is a map $f \colon
S\to [0,\infty]$ which respects $+$ and $\leq$ and satisfies $f(0)=0$. A state $f$ on $S$ is
\emph{nontrivial} if $f(S) \not\subseteq \{0,\infty\}$.
\begin{defn}\label{alm.unper}
Let $S$ be a preordered semigroup $S$ with identity.
We say $S$ is \emph{almost unperforated} if
whenever $x, y \in S$ and $n, m \in \NN$ satisfy $nx \leq my$ and $n > m$, we have $x\leq
y$; equivalently, $(n+1)x \le ny \implies x \le y$.
\end{defn}

We can now state the main technical result used to prove Theorem~\ref{thm2} and later
also Theorem~\ref{thm3}. Recall the definition of semifinite lower-semicontinuous
2-quasitraces from Definition~\ref{def.2.quasitrace}.

\begin{prop}\label{prop.main}
Let $(A,G,\sigma)$ be a $C^*$-dynamical system in which $G$ is discrete.
Suppose that $A$
is a $\sigma$-unital exact $C^*$-algebra of real rank zero and with cancellation.
Finally suppose that
$A\rtimes_{\sigma,r} G$ is simple. Consider the following properties:
\begin{enumerate}
\item\label{prop.main.(1)} $S(A,G,\sigma)$ is purely infinite.
\item\label{prop.main.(2)} $A\rtimes_{\sigma,r} G$ is purely infinite.
\item\label{prop.main.(3)} $A\rtimes_{\sigma,r} G$ is traceless (i.e., admits no nontrivial lower-semicontinuous 2-quasitrace).
\item\label{prop.main.(4)} $A\rtimes_{\sigma,r} G$ is not stably finite.
\item\label{prop.main.(5)} $A\rtimes_{\sigma,r} G$ admits no nontrivial \emph{semifinite} lower-semicontinuous 2-quasitrace.
\item\label{prop.main.(6)} The semigroup $S(A,G,\sigma)$ admits no nontrivial state.
\end{enumerate}
Then the following implication always hold:
\eqref{prop.main.(2)}$\Rightarrow$\eqref{prop.main.(3)}$\Rightarrow$\eqref{prop.main.(4)}$\Rightarrow$\eqref{prop.main.(5)}$\Rightarrow$\eqref{prop.main.(6)}.

If $A_+\smallspace\setminus\{0\}$ is a cofinal subset of $(A\rtimes_{\sigma,r}
G)_+\smallspace\setminus\{0\}$ with respect to $\precsim$, then \eqref{prop.main.(1)}$\Rightarrow$\eqref{prop.main.(2)}. If
$S(A,G,\sigma)$ is almost unperforated, then \eqref{prop.main.(6)}$\Rightarrow$\eqref{prop.main.(1)}.
\end{prop}

\begin{remark}\label{rem.cofinal}
For any properly outer action $\sigma$ of a discrete group $G$ on a separable
$C^*$-algebra $A$ the set $A_+\smallspace\setminus\{0\}$ is a cofinal subset of
$(A\rtimes_{\sigma,r} G)_+\smallspace\setminus\{0\}$ with respect to $\precsim$
(\cite{Sie}).
\end{remark}

Various of the implications in Proposition~\ref{prop.main} can be proved in more general
situations, so we break the bulk of the proof up into four lemmas to make this explicit.
The first result essentially establishes \eqref{prop.main.(1)}$\Rightarrow$\eqref{prop.main.(2)} in
Proposition~\ref{prop.main} when $A_+\smallspace\setminus\{0\}$ is cofinal in $(A\rtimes_{\sigma,r}
G)_+\smallspace\setminus\{0\}$.

\begin{lemma}\label{lem.pi}
Let $(A,G,\sigma)$ be a $C^*$-dynamical system in which $G$ is discrete. Suppose that $A$
is of real rank zero and with cancellation. Finally suppose that $A_+\smallspace\setminus\{0\}$ is a
cofinal subset of $(A\rtimes_{\sigma,r} G)_+\smallspace\setminus\{0\}$ with respect to
$\precsim$. If $S(A,G,\sigma)$ is purely infinite, then $A\rtimes_{\sigma,r} G$ is purely
infinite.
\end{lemma}
\begin{proof}
Since $A$ is a $C^*$-algebra of real rank  zero and with cancellation,  $K_0(A)^+$ satisfies
the Riesz refinement property (\cite{MR1301484, MR1150618}), so $S(A,G,\sigma)$ is
defined.

Take any nonzero $p\in \Pp(A)$. Since $S(A,G,\sigma)$ is purely infinite,
$x=[[p]]_\sigma$ satisfies $2x\leq x$. The proof of \cite[Proposition~4.17]{Rai} does not
require the blanket assumptions made in \cite{Rai} that $A$ is separable and unital, so
that result shows that $p$ is properly infinite in $A\rtimes_{\sigma,r} G$.

Take any nonzero $b\in (A\rtimes_{\sigma,r} G)_+$. Since
$A_+\smallspace\setminus\{0\}$ is cofinal in $(A\rtimes_{\sigma,r}
G)_+\smallspace\setminus\{0\}$ by hypothesis, there exists $a\in A_+ \setminus \{0\}$ such that
$a\precsim b$. Since $A$ has real rank zero, \cite[Lemma~5.3]{MR2873171} yields a nonzero
projection $p\in A$ and a constant $\lambda>0$ such that $\lambda p\leq a$. By
\cite[Lemma~2.3]{MR1172023}, we then have $p \precsim a$, and so transitivity of
$\precsim$ gives $p\precsim b$. By \cite[Lemma~3.8]{MR1759891} applied to $p\precsim b$
we deduce that $b$ is properly infinite. Hence $A\rtimes_{\sigma,r} G$ is purely infinite
by \cite[Theorem~4.16]{MR1759891}.
\end{proof}

The next lemma establishes \eqref{prop.main.(2)}$\Rightarrow$\eqref{prop.main.(3)}$\Rightarrow$\eqref{prop.main.(4)}$\Rightarrow$\eqref{prop.main.(5)} in
Proposition~\ref{prop.main} as a special case of a more general result.

\begin{lemma}[{\cite{MR2032998}}]\label{lem.simple}
Let $B$ be a simple $C^*$ algebra. Consider the following properties:
\begin{enumerate}
\item\label{lem.simple.i} $B$ is purely infinite.
\item\label{lem.simple.ii} $B$ is traceless (i.e., every lower-semicontinuous 2-quasitrace on $B$ is trivial).
\item\label{lem.simple.iii} $B$ is not stably finite.
\item\label{lem.simple.iv} Every \emph{semifinite} lower-semicontinuous 2-quasitrace on $B$ is trivial.
\end{enumerate}
Then \eqref{lem.simple.i}$\Rightarrow$\eqref{lem.simple.ii}$\Rightarrow$\eqref{lem.simple.iii}$\Rightarrow$\eqref{lem.simple.iv}.
\end{lemma}
\begin{proof}
For \eqref{lem.simple.i}$\Rightarrow$\eqref{lem.simple.ii}, see \cite[Proposition~4.1]{MR2032998}. For
\eqref{lem.simple.ii}$\Rightarrow$\eqref{lem.simple.iii} use \cite[Remark 2.27(viii)]{MR2032998}, which states that a
simple $C^*$-algebra $B$ is stably finite if and only if there exists a faithful
semifinite lower-semicontinuous 2-quasitrace on $B$. Finally \eqref{lem.simple.iii}$\Rightarrow$\eqref{lem.simple.iv}
follows from the fact that the set $I_\tau=\{a\in B : \tau(a^*a)=0\}$ is a closed
two-sided ideal in $B$ when $\tau$ is a 2-quasitrace, see \cite[Remark
2.27(vi)]{MR2032998}.
\end{proof}

For a $C^*$-dynamical system $(A,G,\sigma)$, an ideal $I$ of $A$ is \emph{invariant} if
$\sigma_t(I)\subseteq I$ for each $t\in G$. The action $\sigma$ is \emph{minimal} if $A$
admits no nontrivial invariant ideals.

Let $S$ be an (abelian) semigroup with identity. Following \cite{MR3369375, MR2806681},
$S$ is \emph{conical} if 0 is the only invertible element of $S$; that is, $x + y = 0$
implies $x = y = 0$ for all $x,y \in S$. Let $\leq$ be a preorder on $S$. An
\emph{order-unit} in $S$ is any nonzero element $u \in S$ such that for any $x \in S$,
there is $n \in \NN$ such that $x \leq nu$. We say that $S$ is \emph{simple} if every
nonzero element of $S$ is an order unit.

The next preliminary lemma implies in particular that \eqref{prop.main.(6)}$\Rightarrow$\eqref{prop.main.(1)} in
Proposition~\ref{prop.main} provided that $S(A, G, \sigma)$ is almost unperforated.

\begin{lemma}\label{lem.A.conical}
Let $(A,G,\sigma)$ be a $C^*$-dynamical system in which $G$ is discrete. Suppose that $A$
has cancellation and that $K_0(A)^+$ satisfies the Riesz refinement property. Then
\begin{enumerate}
\item\label{lem.A.conical.i} For any $p\in \Pp_\infty(A)$ we have $[[p]]_\sigma=0 \Leftrightarrow p=0$; in particular $S(A,G,\sigma)$ is conical.
\item\label{lem.A.conical.ii} If the action $\sigma$ is minimal, then $S(A,G,\sigma)$ is simple.
\item\label{lem.A.conical.iii} If $S(A,G,\sigma)$ is almost unperforated, then $S(A,G,\sigma)$ is
    purely infinite if and only if $S(A,G,\sigma)$ admits no nontrivial state.
\end{enumerate}
\end{lemma}
\begin{proof}
For \eqref{lem.A.conical.i}, suppose that $[p] \in K_0(A)^+$ satisfies $[[p]]_\sigma = 0 \in S(A,G,\sigma)$.
Then there exist $[p_1],\dots,[p_n]\in K_0(A)^+$ and $t_1,\dots,t_n\in G$ such that $[p]
= [\oplus_ip_i]$ and $[\oplus_i\sigma_{t_i}(p_i)] = [0]$. Since $A$ has cancellation, we
have $\oplus_i\sigma_{t_i}(p_i)\sim 0$. Thus each $p_i=0$, and hence $[p]=[0]$. Another
application of cancellation gives $p=0$. The reverse implication is evident.

We now show that $S(A,G,\sigma)$ is conical. Fix $x,y\in S(A,G,\sigma)$ such that
$x+y=0$. Choose $m\geq 1$ and $p,q\in \Pp_m(A)$ such that $x=[[p]]_\sigma$ and
$y=[[q]]_\sigma$. Since $[[p\oplus q]]_\sigma=0$ we have $p\oplus q=0$. Consequently
$p=q=0$ and $x=y=0$, so $S(A,G,\sigma)$ is conical.

To verify \eqref{lem.A.conical.ii}, assume that $\sigma$ is minimal. Take any nonzero $u\in S(A,G,\sigma)$.
By part \eqref{lem.A.conical.i}, $u=[[p]]_\sigma$ for some $p\neq 0$ in $\Pp_\infty(A)$. We show that $u$ is
an order unit. Fix $x=[[q]]_\sigma$ in $S(A,G,\sigma)$. We may assume that $p,q\in
\Pp_m(A)$ for some $m\geq 1$. The set
$$J_0:=\espan\{x\sigma_{t}(p)y^*: t\in G, \textrm{ and }x,y\in M_m(A)\}\subseteq M_m(A)$$
is $G$-invariant and closed under left and right multiplication (by elements of
$M_m(A)$). So $J:=\overline{J_0}$ is a $G$-invariant ideal of $M_m(A)$. Since the ideals
of $M_m(A)$ have the from $M_m(I)$ for some ideal $I$ of $A$, and since $\sigma$ is
minimal, it follows that $J=M_m(A)$. Fix $a=\sum_{j=1}^k x_j\sigma_{t_j}(p)y_j^*\in
J_0\subseteq M_m(A)$ such that $\|a-q\|<1/2$. Using for example
\cite[Lemma~2.2]{MR1906257}, select $d\in M_m(A)$ satisfying
$d^*ad=(q-\frac{1}{2})_+=\frac{1}{2}q$. Let $z_j:=d(x_j+y_j)\in M_m(A)$,
$z:=(z_1,\dots,z_k)\in M_{m,km}(A)$, and $y:=\sigma_{t_1}(p)\oplus \dots \oplus
\sigma_{t_k}(p)\in M_{km}(A)$. Then \[
zyz^*
    =\sum_{j=1}^k z_j\sigma_{t_j}(p)z_j^*
    \geq d\Big(\sum_{j=1}^k x_j\sigma_{t_j}(p)y_j^*\Big)d^*
        + d\Big(\sum_{j=1}^k y_j\sigma_{t_j}(p)x_j^*\Big)d^*
        = dad^* + da^*d^*
\]
So $zyz^* \geq \frac12q + \frac12q^* = q$. Since $q\leq zyz^*$ \cite[Lemma~2.3]{MR1172023} gives $q\precsim zyz^*$. Using the constant sequence $x_k=z^*$ in the definition of $\precsim$, it follows that
$zyz^*\precsim y$. Since the relation $\precsim$ is transitive, $q\precsim y$. By
\cite[p.~639]{MR1759891}, $q\oplus 0_{m(k-1)}$ is equivalent to a subprojection of $y$,
and so $[q]\leq [y]$. Since the quotient map $\rho\colon K_0(A)^+\to S(A,G,\sigma)$
respects $+$ and $\leq$, and since $u=[[p]]_\sigma$, we have $x=[[q]]_\sigma\leq
[[y]]_\sigma=[[\sigma_{t_1}(p)\oplus\cdots\oplus
\sigma_{t_k}(p)]]_\sigma=k[[p]]_\sigma=ku$. So $u$ is an order unit. Hence $S(A,G,\sigma)$ is simple.

For \eqref{lem.A.conical.iii}, we follow the argument in \cite[p.~95]{Sie}. Suppose that $S(A,G,\sigma)$ is
not purely infinite. Then $2x \not\leq x$ for some $x\in S(A,G,\sigma)$. Since
$S(A,G,\sigma)$ is almost unperforated, Tarski's Theorem gives a state $\nu\colon
S(A,G,\sigma)\to [0,\infty]$ such that $\nu(x)=1$. Since a purely infinite preordered
semigroup with identity admits no nontrivial states, this gives~\eqref{lem.A.conical.iii}.
\end{proof}

Following \cite{MR1172023} we shall refer to a map $\varphi$ from the positive cone,
$A_+$, of a $C^*$-algebra $A$ to $[0,\infty]$ as being a \emph{tracial weight} if it
satisfies $\varphi(\alpha a + \beta b) = \alpha \varphi(a) + \beta \varphi(b)$ and
$\varphi(d^*d) = \varphi(dd^*)$ for all $a,b\in A_+$ all $\alpha,\beta \in \RR^+$ and all
$d\in A$.

The (contrapositive of the) following lemma will be used to prove \eqref{prop.main.(5)}$\Rightarrow$\eqref{prop.main.(6)} of
Proposition~\ref{prop.main}.

\begin{lemma}\label{lem.S.pi}
Let $(A,G,\sigma)$ be a $C^*$-dynamical system in which $G$ is discrete. Suppose that $A$
is a $\sigma$-unital exact $C^*$-algebra of real rank zero and with cancellation.
\begin{enumerate}
\item\label{lem.S.pi.i} If $S(A,G,\sigma)$ is simple and admits a nontrivial state, then there
    exists a faithful $G$-invariant positive additive map $\beta \colon K_0(A)\to
    \RR$.
\item\label{lem.S.pi.ii} For any nontrivial (resp.\ faithful) $G$-invariant positive additive map
    $\beta \colon K_0(A)\to \RR$, there is a nontrivial (resp.\ faithful) semifinite
    lower-semicontinuous tracial weight $\tau_\beta$ on $A\rtimes_{\sigma,r} G$ such
    that $\tau_\beta(p)=\beta([p])$ for every $p \in \Pp_\infty(A)$.
\end{enumerate}
\end{lemma}
\begin{proof}
Since $A$ is a $C^*$-algebra of real rank  zero and with cancellation,  $K_0(A)^+$ satisfies
the Riesz  refinement property (\cite{MR1301484, MR1150618}), so $S(A,G,\sigma)$ is
defined.

For part \eqref{lem.S.pi.i}, let $\nu$ be a nontrivial state on $S(A,G,\sigma)$ and fix $u\in
S(A,G,\sigma)$ such that $\nu(u)=\lambda\in(0,\infty)$. Take any nonzero $y\in
S(A,G,\sigma)$. Since $S(A,G,\sigma)$ is simple there exist $n,m\in \NN$ such that $y\leq
nu$ and $u\leq my$. Hence $\nu(y)\leq n\lambda$ and $\lambda\leq m\nu(y)$. We conclude
that $\nu$ is faithful, i.e.,
\begin{equation}\label{eqn.faithful}
\nu(y)\neq 0 \ \textrm{for each nonzero}\ y\in S(A,G,\sigma).
\end{equation}
Recall $\rho\colon K_0(A)^+\to S(A,G,\sigma)$ is the quotient map $x\mapsto [x]_\sigma$, and define
$\beta \colon K_0(A)^+\to [0,\infty)$ by $\beta:=\nu\circ \rho$, where $\rho$ is
the map $x\mapsto [x]_\sigma$. Since both $\nu$ and $\rho$ respect $+$ and $\leq$ and the
zero element, so does $\beta$. By \cite[Theorem 2.6]{MR1120918}, $A$---and hence also its
stabilisation---admits an approximate unit of projections, giving~\eqref{eqn.K0A}.
If $a,a',b,b'\in K_0(A)^+$ satisfy $a-b=a'-b'$, then $a+b'=a'+b$ and hence
$\beta(a)+\beta(b')=\beta(a')+\beta(b)$ and $\beta(a)-\beta(b)=\beta(a')-\beta(b')$. So
there is a well-defined map $\beta\colon K_0(A)\to \RR$  such that
$\beta(a-b)=\beta(a)-\beta(b)$ for all $a,b\in K_0(A)^+$. This $\beta$ is positive and
additive by definition. If $[p]\in K_0(A)^+$ satisfies $\beta([p])=0$, then
$\nu([[p]]_\sigma)=0$. So~\eqref{eqn.faithful} implies that $[[p]]_\sigma=0$, and
Lemma~\ref{lem.A.conical}\eqref{lem.A.conical.i} gives $p=0$. So $\beta$ is faithful. Finally, $\beta$ is
$G$-invariant because $x\sim_{\sigma} \tilde\sigma_t(x)$ for each $x\in K_0(A)^+$ and
$t\in G$.

For part \eqref{lem.S.pi.ii}, let $\beta \colon K_0(A)\to \RR$ be a nontrivial $G$-invariant positive
additive map. Theorem~\ref{thm1} and Remarks~\ref{remark.thma}\eqref{remark.thma.i}~and~\ref{remark.thma}\eqref{remark.thma.iv}, ensure that
$\beta$ extends to a nontrivial lower-semicontinuous 2-quasitrace $\tau$ on $A$ such that
\[
    \tau(\alpha a + \beta b) = \alpha \tau(a) + \beta \tau(b) \ \textrm{for all} \ a,b\in A_+ \ \textrm{and all} \ \alpha,\beta \in \RR^+.
\]
In particular, $\tau$ is a tracial
weight. Moreover, $\tau$ is $G$-invariant because $A$ has real rank zero and $\tau$ is
$G$-invariant on projections. Using \cite[Lemma~5.3]{MR2873171} we know that if $G$ is
countable, then $\tau$ extends to a nontrivial lower-semicontinuous tracial weight
$\tilde\tau\colon (A\rtimes_{\sigma,r} G)_+\to [0,\infty]$. The proof of \cite[Lemma
5.4]{EllSie} shows that this remains valid for arbitrary $G$.

It remains to show that $\tilde\tau$ is semifinite, and faithful when $\beta$ is
faithful. By Remark~\ref{remark.thma}\eqref{remark.thma.ii}, $\tau = \tilde\tau|_A$ is semifinite. The
finitely supported functions $f\colon G\to \Dom(\tau)$ are dense in $A\rtimes_{\sigma,r}
G$, and the proof of \cite[Lemma~5.3]{MR2873171} shows that for each such $f$ we have
$\tilde\tau(f^*f) = \sum_{t\in G}\tau(f(t)^*f(t)) < \infty$. Hence $\tilde\tau$ is
semifinite.

Now suppose that $\beta$ is faithful. The trace $\tilde\tau$ is the composition of $\tau$
with the faithful conditional expectation $A\rtimes_{\alpha,r} G\to A$. So it suffices to
show that $\tau$ is faithful. Suppose that $a\in A_+$ satisfies $\tau(a)=0$. Since $A$
has real-rank zero, it is spanned by its projections. Hence we can choose
a sequence $(a_n)$ of finite linear combinations $a_n = \sum_i\lambda_{i,n}p_{i,n}$ of
projections such that $a_n \to a$. By \cite[Lemma~2.3]{MR2885587}, we may assume that
each $\lambda_{i,n}\in \RR^+$ and each $a_n\leq a$. Now $\tau(a-a_n) + \tau(a_n) =
\tau(a) = 0$ gives $\tau(a_n)=0$. Hence each $\tau(p_{i,n}) = \beta([p_{i_n}])=0$. Since
$\beta$ is faithful and $A$ has cancellation each $p_{i,n}=0$, so $a=0$. Thus $\tau$ is
faithful.
\end{proof}

\begin{proof}[Proof of Proposition~\ref{prop.main}]
Since $A$ is a $C^*$-algebra of  real rank zero and with cancellation, $K_0(A)^+$ satisfies
the Riesz  refinement property (see \cite{MR1301484, MR1150618}). So Lemma~\ref{lem.pi}
gives \eqref{prop.main.(1)}$\Rightarrow$\eqref{prop.main.(2)}. The implications
\eqref{prop.main.(2)}$\Rightarrow$\eqref{prop.main.(3)}$\Rightarrow$\eqref{prop.main.(4)}$\Rightarrow$\eqref{prop.main.(5)} follow from Lemma~\ref{lem.simple}.
To prove \eqref{prop.main.(5)}$\Rightarrow$\eqref{prop.main.(6)} we establish the contrapositive statement. Suppose that
$S(A,G,\sigma)$ admits a nontrivial state. Since the crossed product $A\rtimes_{\sigma,r}
G$ is simple, the action $\sigma$ is minimal. By Lemma~\ref{lem.A.conical},
$S(A,G,\sigma)$ is simple. We can now use Lemma~\ref{lem.S.pi} to construct a nontrivial
semifinite lower-semicontinuous 2-quasitrace on $A\rtimes_{\sigma,r} G$. Finally the
implication \eqref{prop.main.(6)}$\Rightarrow$\eqref{prop.main.(1)} is contained in Lemma~\ref{lem.A.conical}, part~\eqref{lem.A.conical.iii}.
\end{proof}

\begin{proof}[Proof of Theorem~\ref{thm2}] The  crossed product $A\rtimes_{\sigma,r} G$  is
simple because $A$ is separable, and $\sigma$ is minimal and properly outer (see
\cite{MR568973}). By Remark \ref{rem.cofinal} the set $A_+\smallspace\setminus\{0\}$ is a
cofinal subset of $(A\rtimes_{\sigma,r} G)_+\smallspace\setminus\{0\}$ with respect to
$\precsim$. So the result follows from Proposition~\ref{prop.main}.
\end{proof}

Let $(A,G,\sigma)$ be a $C^*$-dynamical system in which $G$ is discrete. Following
\cite{Rai}, given $k>l>0$, a nonzero element $x\in K_0(A)^+$ is \emph{$(G,k,l)$-paradoxical}
if there exist finitely many elements $u_1,\dots,u_n\in K_0(A)^+$ and
$t_1,\dots,t_n\in G$ such that $kx\leq \sum_iu_i$  and $lx\geq
\sum_i\tilde\sigma_{t_i}(u_i)$ with respect to the algebraic order on $K_0(A)^+$. The
action $\sigma$ is \emph{completely non-paradoxical} if no such elements exist for any
$k>l>0$. Following \cite{MR3210039} (see also \cite{MR3507995}), we let $H_\sigma$ denote the
subgroup of $K_0(A)$ generated by $\{(\id - \tilde\sigma_t)K_0(A) : t \in  G\}$.
Let $S$
be an abelian semigroup with identity. We say that $x\in S$ is \emph{infinite} if
$x+y=x$ for some nonzero $y\in S$. We call $S$ \emph{stably finite} if it contains no
infinite elements, cf.~\cite{MR1998041}. If $S$ is preordered then $x\in S$ is
\emph{properly infinite} if $x\neq 0$ and $2x\leq x$.

We now state a general result used to prove Theorem~\ref{thm2b}. The equivalence \eqref{prop.main.2.(2)}$\Leftrightarrow$\eqref{prop.main.2.(3)} of Proposition \ref{prop.main.2} have been considered previously, see for example \cite[Proposition~3.1]{MR3507995}. For results related to \eqref{prop.main.2.(3)}$\Rightarrow$\eqref{prop.main.2.(2)} when $A$ is unital we refer to \cite[Theorem 4.2]{RaiSch}.

\begin{prop}\label{prop.main.2}
Let $A$ be a $\sigma$-unital exact $C^*$-algebra of real rank zero and with cancellation.
Let $\sigma : G \to \Aut(A)$ be a minimal action of a discrete group. Then the following
are equivalent:
\begin{enumerate}
\item\label{prop.main.2.(1)} $A\rtimes_{\sigma,r} G$ admits a faithful lower-semicontinuous semifinite
    tracial weight;
\item\label{prop.main.2.(2)} $A\rtimes_{\sigma,r} G$ is stably finite;
\item\label{prop.main.2.(3)} $H_\sigma \cap K_0(A)^+ = \{0\}$;
\item\label{prop.main.2.(4)} the action $\sigma$ is completely non-paradoxical;
\item\label{prop.main.2.(5)} the semigroup $S(A,G,\sigma)$ is stably finite;
\item\label{prop.main.2.(6)} $(n+1)x\not\leq nx$ for any nonzero $x\in S(A,G,\sigma)$ and $n\in \NN$;
\item\label{prop.main.2.(7)} there exists a nontrivial state on $S(A,G,\sigma)$; and
\item\label{prop.main.2.(8)} there exists faithful $G$-invariant positive additive map $\beta \colon
    K_0(A)\to \RR$.
\end{enumerate}
\end{prop}

\begin{proof}
Since $A$ is a $C^*$-algebra of  real rank  zero and with cancellation,  $K_0(A)^+$ satisfies
the Riesz  refinement property (see \cite{MR1301484, MR1150618}), so $S(A,G,\sigma)$ is
defined.

For \eqref{prop.main.2.(1)}$\Rightarrow$\eqref{prop.main.2.(2)}, let $\tau$ be a faithful semifinite tracial weight on
$A\rtimes_{\sigma,r} G$. We suppose $A\rtimes_{\sigma,r} G$ is not
stably finite and derive a contradiction. Since $A$ is $\sigma$-unital and of real rank zero, $(A\rtimes_{\sigma,r}
G)\otimes \Kk$ admits an approximate unit of projections, and therefore it contains an
infinite projection $q$. We may assume that $q\in B:=M_n(A\rtimes_{\sigma,r} G)$ for some
nonzero $n\in \NN$ (cf.~Remark~\ref{rem.rordam}). Let $\tau_n$ be the faithful semifinite tracial
weight on $B$ given by $\tau_n(b_{ij}):=\sum_i\tau(b_{ii})$. Using that $\tau$ is
semifinite, the argument of \cite[Lemma~3.2]{MR2204676} shows that $\tau_n(q)<\infty$.
But $q$ is equivalent to a strict subprojection $r$ of $q$, so $\tau_n(q)=\tau_n(r)$ and
$\tau_n(q-r)+\tau_n(r)=\tau_n(q)$. In particular $\tau_n(q-r)=0$ contradicting
faithfulness of $\tau_n$.

The implication \eqref{prop.main.2.(2)}$\Rightarrow$\eqref{prop.main.2.(3)} is precisely the implication (3)$\Rightarrow$(4) of \cite[Proposition~3.1]{MR3507995}. We now prove the contrapositive of
\eqref{prop.main.2.(3)}$\Rightarrow$\eqref{prop.main.2.(4)}. Suppose $\sigma$ is not completely non-paradoxical. This gives a
nonzero element $x \in K_0(A)^+$, $k>l>0$, and finitely many elements $u_1,\dots,u_n\in
K_0(A)^+$ and $t_1,\dots,t_n\in G$ such that $kx\leq \sum_iu_i$  and $lx\geq
\sum_i\tilde\sigma_{t_i}(u_i)$. Select projections  $p_i\in
\Pp_\infty(A)$ such that $u_i=[p_i]$ for each $i$. We have
$$\sum_{i=1}^n (\id - \tilde\sigma_{t_i})([p_i])=\sum_{i=1}^nu_i -\sum_{i=1}^n\tilde\sigma_{t_i}(u_i)\geq (k-l)x.$$ It follows that
$\sum_i (\id - \tilde\sigma_{t_i})([p_i])=(k-l)x+y=z\in K_0(A)^+$ for some $y,z\in
K_0(A)^+$. Since $x\neq 0$, we have $z\neq 0$, so $H_\sigma \cap K_0(A)^+ \neq \{0\}$.

Next we prove the contrapositive of \eqref{prop.main.2.(4)}$\Rightarrow$\eqref{prop.main.2.(5)}. Suppose that $S(A,G,\alpha)$ is
not stably finite. Select a nonzero infinite element $u\in S(A,G,\alpha)$.
Lemma~\ref{lem.A.conical} shows that $S(A,G,\alpha)$ is simple, so \cite[p.~4510]{MR2806681} implies that $u$ is properly infinite as follows. Select nonzero $y\in
S(A,G,\alpha)$ such that $u = u + y$. Then $u = u + my$ for all $m \in  \NN$. Since $S(A,G,\alpha)$ is simple, $y$ is
an order unit, and so $2u \leq my \leq u$ for some $m\in\NN$. Select nonzero $x\in
K_0(A)^+$ such that $[x]_\sigma=u$. Then $2[x]_\sigma\leq [x]_\sigma$ and following the
proof of \cite[Lemma~4.10]{Rai}, we deduce that $x$ is $(G,2,1)$-paradoxical.
Consequently $\sigma$ is not completely non-paradoxical.

The proof of \eqref{prop.main.2.(5)}$\Rightarrow$\eqref{prop.main.2.(6)} is also
contrapositive. Suppose there exists a nonzero element $x\in S(A,G,\alpha)$ and $n\in
\NN$ such that $(n+1)x\leq nx$. Then $n>0$ because $S(A,G,\alpha)$ is conical by
Lemma~\ref{lem.A.conical} and $nx$ is infinite. Consequently $S(A,G,\alpha)$ fails to be
stably finite. To verify \eqref{prop.main.2.(6)}$\Rightarrow$\eqref{prop.main.2.(7)} fix
a nonzero $x\in S(A,G,\alpha)$. Then Tarski's Theorem yields a nontrivial state
$\nu\colon S(A,G,\sigma)\to [0,\infty]$ with $\nu(x)=1$. The implication
\eqref{prop.main.2.(7)}$\Rightarrow$\eqref{prop.main.2.(8)} follows from
Lemma~\ref{lem.A.conical}\eqref{lem.A.conical.ii} and
Lemma~\ref{lem.S.pi}\eqref{lem.S.pi.i}, and
\eqref{prop.main.2.(8)}$\Rightarrow$\eqref{prop.main.2.(1)} follows from
Lemma~\ref{lem.S.pi}\eqref{lem.S.pi.ii}.
\end{proof}

\begin{proof}[Proof of Theorem~\ref{thm2b}]
This is \eqref{prop.main.2.(1)}$\Leftrightarrow$\eqref{prop.main.2.(2)}$\Leftrightarrow$\eqref{prop.main.2.(4)} of Proposition \ref{prop.main.2}.
\end{proof}

\section{Dichotomy for \texorpdfstring{$k$}{k}-graph \texorpdfstring{$C^*$}{C*}-algebras and the proof of Theorem~\ref{thm3}}
\label{app.two}

Following \cite{MR1745529, MR2139097, MR1961175} we briefly recall the notion of a
$k$-graph and the associated notation. For $k \ge 1$, a \emph{$k$-graph} is a non-empty
countable small category equipped with a functor $d : \Lambda \to \NN^k$ that satisfies
the following \emph{factorisation property}. For all $\lambda \in \Lambda$ and $m,n\in
\NN^k$ such that $d(\lambda) = m + n$ there exist unique $\mu, \nu\in \Lambda$ such that
$d(\mu) = m$, $d(\nu) = n$, and $\lambda = \mu\nu$. When $d(\lambda) = n$ we say
$\lambda$ has \emph{degree} $n$, and we write $\Lambda^n = d^{-1}(n)$. The standard
generators of $\NN^k$ are denoted $e_1, \dots ,e_k$, and we write $n_i$ for the
$i^{\textrm{th}}$ coordinate of $n \in \NN^k$. We define a partial
order on $\NN^k$ by $m\leq n$ if $m_i\leq n_i$ for all $i\leq k$.

If $\Lambda$ is a $k$-graph, its \emph{vertices} are the elements of $\Lambda^0$. The
factorisation property implies that these are precisely the identity morphisms, and so
can be identified with the objects. For $\lambda \in \Lambda$ the \emph{source}
$s(\lambda)$ is the domain of $\lambda$, and the \emph{range} $r(\lambda)$ is the
codomain of $\lambda$ (strictly speaking, $s(\lambda)$ and $r(\lambda)$ are the identity
morphisms associated to the domain and codomain of $\lambda$). Given $\lambda,\mu\in
\Lambda$ and $E\subseteq \Lambda$, we define
\begin{gather*}
\lambda E=\{\lambda\nu : \nu\in E, r(\nu)=s(\lambda)\},\quad E\mu=\{\nu\mu : \nu\in E, s(\nu)=r(\mu)\},\qquad\text{ and}\\
\lambda E\mu=(\lambda E)\mu=\lambda(E\mu).
\end{gather*}
For $X,E,Y\subseteq\Lambda$, we write $XEY$ for $\bigcup_{\lambda\in X, \mu\in Y}\lambda
E\mu$. We say that the $k$-graph $\Lambda$ is \textit{row-finite} if $|v{\Lambda}^{n}|<\infty$
is finite for each $n\in\mathbb{N}^k$ and $v\in{\Lambda}^0$, and has \textit{no sources}
if $0<|v{\Lambda}^{m}|$ for all $v\in{\Lambda}^0$ and $m\in\mathbb{N}^k$.

If $X$ is a countable set, we write $\NN X$ for the abelian semigroup with identity
$x\mapsto 0$ of finitely supported functions $f\colon X\to \NN$, with respect to
point-wise addition.  We denote the generator corresponding to $x\in X$ by $\delta_x$,
and for $a\in  \NN X$ we write $a_x$ for $a(x)$. If $\Lambda$ is a $k$-graph and $n\in
\NN^k$, we write $A_n$ and $A_n^t$ for the elements of $M_{\Lambda^0}(\NN)$ given by
$$A_n(v, w) = |v\Lambda^n w| \ \ \textrm{and} \ \ A_n^t(v, w) = |w\Lambda^n v|.$$
By the factorisation property $A_n$ is equal to the matrix product $A_{e_1}^{n_1}\cdots A_{e_k}^{n_k}$, where each $A_{e_i}$ is
the \emph{coordinate matrix} defined by $A_{e_i}(u, v) =  |u\Lambda^{e_i} v|$. We often regarded each $A_{e_i}$ as
a semigroup endomorphism of $\NN\Lambda^0$: $(A_{e_i}f)(u) =\sum_{v\in \Lambda^0} A_{e_i}(u,
v)f(v)$. By the
factorisation property we have (for $m,n\in \NN^k$),
\begin{eqnarray}\label{factorisation}
A_mA_n = A_{m+n}.
\end{eqnarray}

\begin{lemma}\label{lem.relation.one}
The relation $\sim_{\Lambda}$ on $\NN\Lambda^0$ given by $x\sim_{\Lambda} y$ if and only
if there exists $p,q\in \NN^k$ such that $A^t_px=A^t_qy$ is an equivalence relation.
\end{lemma}
\begin{proof}
It is clear that $\sim_{\Lambda}$ is reflexive and symmetric. If $x\sim_{\Lambda}y$ and
$y\sim_{\Lambda}z$, then $A^t_px=A^t_qy$ and $A^t_ry=A^t_sz$ for some $p,q,r,s\in \NN^k$,
so $A^t_{p+r}x=A^t_{q+s}z$.
\end{proof}

Let $S$ be an (abelian) semigroup $S$ with an identity and satisfying the Riesz
refinement property. A simple induction shows that for any finite collection of elements $u_1, \dots, u_n$, $v_1, \dots,
v_l$ in $S$ such that $\sum_{i=1}^nu_i=\sum_{j=1}^lv_j$, there exist $\{w_{ij}: i\leq n, j\leq l\}\subseteq S$ such
that $u_i=\sum_jw_{ij}$ and $v_j=\sum_iw_{ij}$ for each $i$ and $j$.

\begin{lemma}\label{lem.relation.two}
The relation $\approx_{\Lambda}$ on $\NN\Lambda^0$ given by $f\approx_{\Lambda}g$ if and
only if there exist finitely many $x_i,y_i\in \NN \Lambda^0$, such that
$f\sim_{\Lambda}\sum_i x_i$, $g\sim_{\Lambda}\sum_i y_i$ and $x_i\sim_{\Lambda} y_i$ for
all $i$ is an equivalence relation. Moreover $f\approx_{\Lambda}g$ if and only if there
exist $m,m_0\in \NN^k$ and finitely many $f_i\in  \NN \Lambda^0$ and $t_i\in \ZZ^k$ such
that $0\leq m_0,m_0+t_i\leq m$ for all $i$, and
$$A^t_mf=\sum_iA^t_{m_0}f_i, \quad\text{and}\quad A^t_mg=\sum_i A^t_{m_0+t_i}f_i.$$
\end{lemma}
\begin{proof}
We start by proving the second statement. Suppose $f\approx_{\Lambda}g$, and choose
$\{x_i\}$ and $\{y_i\}$ with $f\sim_{\Lambda}\sum_i x_i$, $g\sim_{\Lambda}\sum_i y_i$ and
$x_i\sim_{\Lambda} y_i$ for all $i$. Take $p,q,r,s$ in $\NN^k$ such that
$A^t_pf=A^t_q\sum_i x_i$ and $A^t_rg=A^t_s\sum_i y_i$. Define $n=p+q+r+s$,
$f_i=A^t_{s+q}x_i$ and $h_i=A^t_{s+q}y_i$. We have
$A^t_{n}f=A^t_{n-p}A^t_pf=A^t_{n-p+q}\sum_i x_i=\sum_i A^t_{n-p-s}f_i$, and similarly
$A^t_{n}g=\sum_i A^t_{n-r-q}h_i$. Since $f_i\sim_{\Lambda}h_i$, there exist $n_i$ and
${n_i+s_i}$ in $\NN^k$ ($s_i\in \ZZ^k$) such that $A^t_{n_i}h_i=A^t_{n_i+s_i}f_i$. Now
for $m=\max(|s_i|,n_i)$ we have $A^t_{n+m}f=\sum_i A^t_{m+n-p-s}f_i$ and
$A^t_{m+n}g=\sum_i A^t_{m+n-r-q-n_i}A^t_{n_i}h_i=\sum_i A^t_{m+n-r-q+s_i}f_i$. Set
$m_0:=m+n-p-s$ and define $t_i\in \ZZ^k$ such that $m+n-r-q+s_i=m_0+t_i$. Then
$A^t_{n+m}f=\sum_i A^t_{m_0}f_i$ and $A^t_{m+n}g=\sum_i A^t_{m_0+t_i}f_i$. The converse
implication follows directly from the definition of $\approx_{\Lambda}$ since each
$A^t_{m_0}f_i\sim_{\Lambda}f_i\sim_{\Lambda}A^t_{m_0+t_i}f_i$.

We now show that $\approx_{\Lambda}$ is  an equivalence relation. It is clear that
$\approx_{\Lambda}$ is reflexive and symmetric. Now, suppose that $f\approx_{\Lambda}g$ and
$g\approx_{\Lambda}h$. By the preceding paragraph there exist $f_i,g_i\in \NN\Lambda^0$,
$m,m_i,l_i\in \NN$, and $t_i,s_i\in\ZZ^k$ such that $0\leq m_i,m_i+t_i\leq m$, $0\leq
l_i,l_i+s_i\leq m$, and
$$A^t_mf=\sum_{i=1}^nA^t_{m_i}f_i, \ \ \ A^t_mg=\sum_{i=1}^nA^t_{m_i+t_i}f_i, \ \ \ A^t_mg=\sum_{j=1}^lA^t_{l_j}g_j, \ \ \text{and} \  \ A^t_mh=\sum_{j=1}^lA^t_{l_j+s_j}g_j,$$
for some $l,k\geq 1$. Since $\NN$ satisfies the Riesz refinement property, so does $\NN
\Lambda^0$. Thus there exist $\{w_{ij}: i\leq n, j\leq l\}\subseteq \NN\Lambda^0$ such that
$$\sum_{j=1}^lw_{ij}=A^t_{m_i+t_i}f_i, \quad\text{and}\quad\sum_{i=1}^nw_{ij}=A^t_{l_j}g_j.$$
Now
\begin{align*}
\sum_{ij}A^t_{m+s_j}w_{ij}&=\sum_{j}A^t_{m+s_j}\sum_iw_{ij}=\sum_{j}A^t_{m+s_j+l_j}g_j=A^t_{2m}h,\ \ \ \text{and}\\
\sum_{ij}A^t_{m-t_i}w_{ij}&=\sum_{i}A^t_{m-t_i}\sum_jw_{ij}=\sum_{i}A^t_{m-t_i}A^t_{m_i+t_i}f_i=A^t_m\sum_{i}A^t_{m_i}f_i=A^t_{2m}f,
\end{align*}
so $f\approx_{\Lambda}h$.
\end{proof}

\begin{defn}\label{def.relation}
Let $\Lambda$ be a row-finite $k$-graph with no sources. We say that $x,y\in \NN \Lambda^0$
are $\sim_{\Lambda}$-equivalent, denoted $x\sim_{\Lambda}y$, if there exist $p,q\in
\NN^k$ such that $A^t_px=A^t_qy$. We say that $f,g\in \NN \Lambda^0$ are
$\approx_{\Lambda}$-equivalent, denoted $f\approx_{\Lambda}g$, if there exist finitely
many $x_i,y_i\in \NN \Lambda^0$, such that $f\sim_{\Lambda}\sum_i x_i$,
$g\sim_{\Lambda}\sum_i y_i$ and $x_i\sim_{\Lambda} y_i$ for each $i$.
\end{defn}

The advantage of the equivalence relation $\approx_{\Lambda}$ over $\sim_{\Lambda}$ is
that addition in  $\NN \Lambda^0$ induces an addition on $\NN
\Lambda^0/\smallspace\approx_{\Lambda}$.

\begin{lemma}\label{lem.addition}
Let $\Lambda$ be a row-finite $k$-graph with no sources. For $f\in \NN \Lambda^0$, let
$[f]_\Lambda$ denote the $\approx_{\Lambda}$-class of $f$. Then there is a well-defined
commutative binary operation $+$ on $\NN \Lambda^0/\smallspace\approx_{\Lambda}$ such
that $\class{f}+\class{g}=\class{f+g}$ for all $f,g\in \NN \Lambda^0$.
\end{lemma}
\begin{proof}
Take $f,f',g\in \NN \Lambda^0$ with $f\approx_{\Lambda}f'$. Use Lemma~\ref{lem.relation.two} to find finitely many $f_i\in
\NN \Lambda^0$, $m,m_0\in \NN^k$ and $t_i\in \ZZ^k$ such that $0\leq m_0,m_0+t_i\leq m$,
$A^t_mf=\sum_iA^t_{m_0}f_i$, and $A^t_mf'=\sum_i A^t_{m_0+t_i}f_i$. Then
$$A^t_m(f+g)=\sum_iA^t_{m_0}f_i+A^t_mg, \ \ \ A^t_m(f'+g)=\sum_i A^t_{m_0+t_i}f_i+A^t_mg,$$
giving $f+g\approx_{\Lambda}f'+g$. This operation is commutative because addition in $\NN
\Lambda^0$ is commutative.
\end{proof}

\begin{defn}\label{def.S.Lambda}
Let $\Lambda$ be a row-finite $k$-graph with no sources. We define $$S(\Lambda):=\NN
\Lambda^0/\smallspace\approx_{\Lambda}$$ under the addition defined in Lemma
\ref{lem.addition}. We call $S(\Lambda)$ the \emph{semigroup associated to $\Lambda$}; note that $\class{0}$ is an identity for $S(\Lambda)$.
\end{defn}

Our goal is to relate $S(\Lambda)$ to Rainone's type semigroup $S(A,\ZZ^k,\sigma)$ for an
appropriate $C^*$-dynamical system associated to the $k$-graph $\Lambda$. We must first
recall the definition of a $C^*$-algebra of the $k$-graph and describe the associated
dynamical system. The $C^*$-algebra $C^*(\Lambda)$ of a row-finite $k$-graph $\Lambda$
with no sources is the universal $C^*$-algebra generated by elements $\{s_\lambda :
\lambda\in\Lambda \}$ satisfying the Cuntz--Krieger relations:
\begin{itemize}
\item[(CK1)] $\{s_v : v\in \Lambda^0\}$ is a collection of mutually orthogonal
    projections;
\item[(CK2)] $s_\mu s_\nu = s_{\mu\nu}$ whenever $s(\mu) = r(\nu)$;
\item[(CK3)] $s^*_\lambda s_\lambda = s_{s(\lambda)}$ for all $\lambda\in \Lambda$;
    and
\item[(CK4)] $s_v = \sum_{\lambda \in v\Lambda^n} s_\lambda s_\lambda^*$ for all
    $v\in \Lambda^0$ and $n\in \NN^k$.
\end{itemize}

Let $\Lambda$ be a row-finite $k$-graph with no sources. Set $A:=C^*(\Lambda \times_d
\ZZ^k)$, where $\Lambda \times_d \ZZ^k$ is the \emph{skew-product} $k$-graph of $\Lambda$
and $\ZZ^k$, see \cite[Definition 5.1]{MR1745529}. Recall (see
\cite[Lemma~5.2]{MR3011251} or \cite[Theorem~5.7]{MR1745529}) that there is an action
$\sigma$ of $\ZZ^k$ on $A$ such that $\sigma_m(s_{(\lambda,n)})=s_{(\lambda,n-m)}$ for
all $\lambda \in\Lambda$ and $m, n\in \ZZ^k$. Let $P:=\sum_{v\in \Lambda^0}s_{(v,0)}\in
\Mm(A)\subseteq \Mm(A\rtimes_{\sigma,r}\ZZ^k)$. Then $P$ is \emph{full} (in the sense that $\overline{APA}=A$), and there is an isomorphism $\varphi \colon
P(A\rtimes_{\sigma,r}\ZZ^k)P \to C^*(\Lambda)$ satisfying $\varphi(s_{(v,0)})=s_v$ for
all $v\in \Lambda^0$. So by \cite{MR0463928}, $A\rtimes_{\sigma,r}\ZZ^k$ is stably
isomorphic to $C^*(\Lambda)$.

\begin{defn}\label{defn.dynamical}
We call $(A,\ZZ^k,\sigma)$ \emph{the $C^*$-dynamical system associated to $\Lambda$}.
\end{defn}

\begin{lemma}\label{lem.S.isom}
Let $\Lambda$ be a row-finite $k$-graph with no sources. In the terminology of the
preceding paragraph, the semigroup $S(\Lambda)$ is isomorphic to the semigroup
$S(A,\ZZ^k,\sigma)$ for the $C^*$-dynamical system associated to $\Lambda$ as defined in
Definition \ref{def.S.A.G}, via an isomorphism that carries $\class{\delta_v}$ to $[[s_{(v,0)}]]_\sigma$.
\end{lemma}
\begin{proof}
Following the notation of \cite{MR3507995}, write $G_n = \ZZ \Lambda^0$ for each $n\in
\ZZ^k$, and define ${A^t_{m,n} \colon G_m \to G_n}$ for $m \le n \in \ZZ^k$ by
$A^t_{m,n}= A^t_{n-m}$. Considering each $G_n$ as a free abelian group on generators
indexed by $\Lambda^0$ we form the direct limit $\varinjlim(\ZZ \Lambda^0, A^t_{m,n})$.
For each $n\in \ZZ^k$, let $A^t_{n,\infty} \colon G_n = \ZZ \Lambda^0 \to
\varinjlim(\ZZ\Lambda^0, A^t_{m',n'})$ be the canonical homomorphism into the limit such
that $A^t_{n,\infty} \circ A^t_{m,n} = A^t_{m,\infty}$ for all $m \le n\in \ZZ^k$.

Let $(A,\ZZ^k,\sigma)$ be the $C^*$-dynamical system associated to $\Lambda$. By
\cite[Lemma~3.4]{MR3507995} there is an order-isomorphism
$$\rho\colon K_0(A)\to \varinjlim(\ZZ\Lambda^0, A^t_{m,n})=\bigcup_{m\in \ZZ^k} A^t_{m,\infty} (\ZZ \Lambda^0),$$
satisfying $\rho([s_{(v,m)}])=A^t_{m,\infty}\delta_v$ and $\rho(K_0(A)^+)=\bigcup_{m\in
\NN^k} A^t_{m,\infty} (\NN \Lambda^0)$. Define an action $\beta$ of $\ZZ^k$ on $\bigcup_{m\in \NN^k} A^t_{m,\infty} (\NN \Lambda^0)$ as follows: For $m\in \NN^k$, $n\in \ZZ^k$, $l\in \NN^k$ such that $n+l\in \NN^k$ and $v\in\
\Lambda^0$ set
\begin{eqnarray}\label{action.beta}
\beta_n(A^t_{m,\infty}\delta_v)=\beta_n(A^t_{m+l,\infty}\circ A^t_{l}\delta_v):=A^t_{m+l,\infty} \circ A^t_{n+l}\delta_v.
\end{eqnarray}
Then $\rho$ intertwines the action
$\tilde\sigma$ of $\ZZ^k$ on $K_0(A)^+$ induced by $\sigma$ with the action $\beta$ of
$\ZZ^k$ on $\bigcup_{m\in \NN^k} A^t_{m,\infty} (\NN \Lambda^0)$.
Identifying $K_0(A)^+$ with $\bigcup_{m\in \NN^k} A^t_{m,\infty} (\NN \Lambda^0)$ and
$\tilde\sigma$ with $\beta$ as above, we claim that there is an additive map $\Phi\colon
S(\Lambda) \to S(A,\ZZ^k,\sigma)$ such that
 \begin{equation}\label{eqn.the.map}
 \Phi([f]_\Lambda)= [A^t_{0,\infty}f]_\beta
 \end{equation}
 for all $f\in \NN\Lambda^0$. To see this suppose that $x,y\in \NN\Lambda^0$ satisfy $x\sim_\Lambda y$. Then $A^t_px=A^t_qy$ for some $p,q\in \NN^k$. Using \eqref{action.beta}, $\beta_p(A^t_{0,\infty}x)=A^t_{0,\infty} \circ A^t_px=A^t_{0,\infty} \circ A^t_qy=\beta_q(A^t_{0,\infty}y)$, and so $[A^t_{0,\infty}x]_\beta=[A^t_{0,\infty}y]_\beta$. Now suppose that $f\approx_{\Lambda} g$. Choose $\{x_i\}$ and $\{y_i\}$ with $f\sim_{\Lambda}\sum_i x_i$, $g\sim_{\Lambda}\sum_i y_i$ and $x_i\sim_{\Lambda} y_i$ for all $i$. Then $[A^t_{0,\infty}f]_\beta=[A^t_{0,\infty}\sum_ix_i]_\beta$, $[A^t_{0,\infty}g]_\beta=[A^t_{0,\infty}\sum_iy_i]_\beta$ and $[A^t_{0,\infty}x_i]_\beta=[A^t_{0,\infty}y_i]_\beta$ for all $i$. It follows that $[A^t_{0,\infty}f]_\beta=[A^t_{0,\infty}g]_\beta$. So the formula \eqref{eqn.the.map} is well-defined. It is additive because $\rho$ and $A^t_{0,\infty}$ are additive.

We show that $\Phi$ is surjective. Take any $x=A^t_{m,\infty}f$ with $m\in \NN^k$ and
$f\in \NN\Lambda^0$. Using \eqref{action.beta} we obtain
$\Phi(\class{f})=[A^t_{0,\infty}f]_\beta=[A^t_{m,\infty} \circ
A^t_mf]_\beta=[\beta_m(A^t_{m,\infty}f)]_\beta=[\beta_m(x)]_\beta=[x]_\beta$.

We show that $\Phi$ is injective. Take any $f,g\in \NN\Lambda^0$. Suppose that
$\Phi(\class{f})=\Phi(\class{g})$; that is, that
$[A^t_{0,\infty}f]_\beta=[A^t_{0,\infty}g]_\beta$. Select finitely many elements
$u_1,\dots,u_n\in \bigcup_{m\in \NN^k} A^t_{m,\infty} (\NN \Lambda^0)$ with
$u_i=A^t_{m_i,\infty}f_i$, and $t_1,\dots,t_n\in \ZZ^k$ such that
$A^t_{0,\infty}f=\sum_iu_i$ and $A^t_{0,\infty}g=\sum_i\beta_{t_i}(u_i)$ for some $n\geq
1$. Define $m=\max_i(|t_i|)+\max_i(m_i)$ and $n_i=m-m_i$. We get
$$\textstyle{0=A^t_{0,\infty}f-\sum_iu_i=A^t_{m,\infty} \circ A^t_mf-\sum_iA^t_{m_i+n_i,\infty} \circ A^t_{n_i}f_i=A^t_{m,\infty} \circ (A^t_mf-\sum_i A^t_{n_i}f_i).}$$
Since $\ker A^t_{m,\infty}=\bigcup_{p\geq m}\ker A^t_p$, we get $A^t_pA^t_mf=A^t_p\sum_i
A^t_{n_i}f_i$ for some $p\geq m$. In particular $f \sim_\Lambda \sum_i A^t_{n_i}f_i$.
Similarly $0=A^t_{m,\infty} \circ (A^t_mg-\sum_i A^t_{n_i+t_i}f_i)$, so $g \sim_\Lambda
\sum_i A^t_{n_i+t_i}f_i$. So $\class{f}=\class{g}$. Therefore $\Phi$ is a semigroup
isomorphism.

The map $\Phi$ carries $\class{\delta_v}$ to $[A^t_{0,\infty}\delta_v]_\beta=[\rho([s_{(v,0)}])]_\beta$ which is exactly the image of $[[s_{(v,0)}]]_\sigma$ under the identification of $S(A,\ZZ^k,\sigma)$ with $\bigcup_{m\in \NN^k} A^t_{m,\infty} (\NN \Lambda^0)/\smallspace\sim_\beta$ induced by $\rho$. So the isomorphism $S(\Lambda)\to S(A,\ZZ^k,\sigma)$ carries $\class{\delta_v}$ to $[[s_{(v,0)}]]_\sigma$.
\end{proof}

Lemma~\ref{lem.S.isom} gives a very useful concrete description of $S(A,\ZZ^k,\sigma)$.
We illustrate this by providing a few examples of the computation of $S(A,\ZZ^k,\sigma)$.
First we recall the notion of $k$-coloured graphs allowing us to describe $k$-graphs
diagrammatically.

Following \cite{MR3056660}, a $k$-\emph{coloured} graph is a directed graph $E$ endowed
with a map $c \colon E^1 \to \{c_1, \dots, c_k\}$. We let $E^*$ denote the set of finite
paths in $E$. We think of $c$ as determining a colour map from $E^*$ to the free
semigroup $\FF_n^+$ generated by $\{c_1, \dots, c_k\}$.  For $w\in \FF_n^+$ we say
$\lambda\in E^*$ is \emph{$w$-coloured} if $c(\lambda) = w$. A collection of
\emph{factorisation rules} for $E$ is a range- and source-preserving bijection
$\theta_{ij}$ from the $c_ic_j$-coloured paths in $E^*$ to the $c_jc_i$-coloured paths in
$E^*$, for each $i\neq j$. For $k = 2$ the associativity condition of \cite{MR3056660} is
trivial, and so \cite[Theorems 4.4 and 4.5]{MR3056660} say that for every 2-coloured
graph $(E, c)$ with a collection of factorisation rules $\theta_{12}$, there is a unique
2-graph $\Lambda$ with  $\Lambda^{e_i}=c^{-1}(c_i)$, $\Lambda^0=E^0$ , and $ef=f'e'$ in
$\Lambda$ whenever $\theta_{12}(ef)=f'e'$ in $E^*$.

\begin{example}
Consider the 1-graph $\Lambda$ illustrated on the left on Figure~\ref{pic1}.
\begin{figure}
\begin{center}
\begin{tikzpicture}
        \node[inner sep=2.8pt, circle] (27) at (0,0) {$v$};	
        	\path[->,every loop/.style={looseness=14}, >=latex] (27)
			 edge  [in=150,out=210,loop, blue, >=latex] ();
	\path[->,every loop/.style={looseness=24}, >=latex] (27)
			 edge  [in=140,out=220,loop, blue, >=latex] ();		
\end{tikzpicture}
\ \ \ \  \ \ \ \ \ \ \
\begin{tikzpicture}
        \node[inner sep=2.8pt, circle] (27) at (2,8) {$\bullet$};
	\node at (2,6.4) {$\vdots$};
	\node at (2,5.55) {$n_k$};
	\node at (2,9.8) {$\vdots$};
	\node at (2,10.4) {$n_2$};
	\node at (0.3,8) {$\cdots$};
	\node at (-0.5,8) {$n_1$};	
	\node at (3,8.45) {\tiny{.}};
	\node at (3.15,8.225) {\tiny{.}};
	\node at (3.2,8) {\tiny{.}};
	\node at (3.15,7.775) {\tiny{.}};
	\node at (3,7.55) {\tiny{.}};
	\path[->,every loop/.style={looseness=10}, >=latex] (27)
	         edge  [in=70,out=110,loop, red, dashed, >=latex] ();
	\path[->,every loop/.style={looseness=14}, >=latex] (27)
			 edge  [in=60,out=120,loop, red, dashed, >=latex] ();
	\path[->,every loop/.style={looseness=24}, >=latex] (27)
			 edge  [in=50,out=130,loop, red, dashed, >=latex] ();			
	\path[->,every loop/.style={looseness=10}, >=latex, dashdotted] (27)
	         edge  [in=250,out=290,loop, darkgreen] ();
	\path[->,every loop/.style={looseness=14}, >=latex, dashdotted] (27)
			 edge  [in=240,out=300,loop, darkgreen, >=latex] ();
	\path[->,every loop/.style={looseness=24}, >=latex, dashdotted] (27)
			 edge  [in=230,out=310,loop, darkgreen] ();		
	\path[->,every loop/.style={looseness=10}, >=latex] (27)
	         edge  [in=160,out=200,loop, blue, >=latex] ();
	\path[->,every loop/.style={looseness=14}, >=latex] (27)
			 edge  [in=150,out=210,loop, blue, >=latex] ();
	\path[->,every loop/.style={looseness=24}, >=latex] (27)
			 edge  [in=140,out=220,loop, blue, >=latex] ();			
\end{tikzpicture}
\ \ \
\begin{tikzpicture}

\def \n {5}
\def \radiuss {2.0cm}
\def \radiusm {2.15cm}
\def \radiusb {2.15cm}
\def \margin {6} 

\foreach \s in {2,...,0} {
  \node[circle, inner sep=0pt] at ({360/\n * (\s - 1)}:\radiusm) {$v_{\s}$};
  \draw[blue, <-, >=latex] ({360/\n * (\s - 1)+\margin}:\radiuss)
    arc ({360/\n * (\s - 1)+\margin}:{360/\n * (\s)-\margin}:\radiuss);

} \foreach \s in {2,...,0} {
  \draw[red, dashed, <-, >=latex] ({360/\n * (\s - 1)+\margin}:\radiusb)
    arc ({360/\n * (\s - 1)+\margin}:{360/\n * (\s)-\margin}:\radiusb);
} \foreach \s in {4} {
  \node[circle, inner sep=0pt] at ({360/\n * (\s - 1)}:\radiusm) {$v_{n-1}$};
  \draw[blue, <-, >=latex] ({360/\n * (\s - 1)+\margin}:\radiuss)
    arc ({360/\n * (\s - 1)+\margin}:{360/\n * (\s)-\margin}:\radiuss);
} \foreach \s in {4} {
  \draw[red, dashed, <-, >=latex] ({360/\n * (\s - 1)+\margin}:\radiusb)
    arc ({360/\n * (\s - 1)+\margin}:{360/\n * (\s)-\margin}:\radiusb);
} \foreach \s in {3} {
  \node[circle, inner sep=0pt] at ({360/\n * (\s - 1)}:\radiusm) {$v_{\s}$};
  \draw[black, loosely dotted, >=latex] ({360/\n * (\s - 1)+\margin}:\radiuss)
    arc ({360/\n * (\s - 1)+\margin}:{360/\n * (\s)-\margin}:\radiuss);
} \foreach \s in {3} {
  \draw[black, loosely dotted, >=latex] ({360/\n * (\s - 1)+\margin}:\radiusb)
    arc ({360/\n * (\s - 1)+\margin}:{360/\n * (\s)-\margin}:\radiusb);
}
\end{tikzpicture}
\caption{} \label{pic1}
\end{center}
\end{figure}
Let us compute $S(A,\ZZ,\sigma)$ for the dynamical system associated to $\Lambda$.
Revisiting Lemma~\ref{lem.S.isom}, we have the ordered isomorphism $K_0(A)\cong
\varinjlim(\ZZ\Lambda^0, A^t_{m,n})$, hence $K_0(A)\cong\varinjlim (\ZZ, \times 2)\cong
\ZZ[\frac{1}{2}]$ and $K_0(A)^+\cong \ZZ[\frac{1}{2}]\cap \RR^+=\NN[\frac{1}{2}]$. The
coordinate matrix is just multiplication by $2$. Consequently, the action $\beta$ of
$\ZZ$ on $\NN[\frac{1}{2}]$ is given by $\beta_1(x)=2x$. We conclude that for each
element $x\in S(A,\ZZ,\sigma)\cong \NN[\frac{1}{2}]/\smallspace\sim_{\beta}$,
$[x]_\beta=[y]_\beta$ for some $y\in \NN$. Moreover, if $y\geq 2$, then
$[y]_\beta=[y-2]_\beta+[2]_\beta=[y-2]_\beta+[1]_\beta=[y-1]_\beta$. It follows that $y$
equals $[0]_\beta$ or $[1]_\beta$. Thus $S(A,\ZZ,\sigma)\cong \{0, 1\}$, which is totally
ordered and purely infinite with $1+1=1$.
\end{example}

\begin{example}
\label{three.colour} Consider the $k$-graph $\Lambda$ illustrated in the middle on
Figure~\ref{pic1} for some collection of factorisation rules. Denote the number of
$c_i$-coloured edges by $n_i$ ($i\in\{1,\dots, k\}$). One can show that $S(\Lambda)\cong
\NN[\frac{1}{n_1},\dots, \frac{1}{n_k}]/\smallspace\approx_{\Lambda}$, with $f
\approx_{\Lambda} g$ if and only if $f=\sum_i x_i$,  $g=\sum_i \beta(t_i) x_i$ for some
$x_i\in \NN[\frac{1}{n_1},\dots, \frac{1}{n_k}]$ and $t_i\in \ZZ^{k}$ where
$\beta(t_i)=n_1^{t_{i1}}\cdots n_k^{t_{ik}}$. If each $n_i=1$ then $\beta=1$, so $f
\approx_{\Lambda} g$ if and only if $f=g$. So in this case $S(\Lambda)\cong\NN$ is totally
ordered and stably finite. If at least one $n_i\neq 1$, we claim that $x\leq y$ for any
nonzero $x,y\in S(\Lambda)$. To see this fix nonzero $l,m\in \NN\setminus \{0\}$ such
that $x=\class{l}$ and $y=\class{m}$, and then find $N>0$ such that $l\leq n_i^N m$. Then
$x=\class{l}\leq \class{n_i^N m}=\class{m}=y$. In particular $S(\Lambda)$ is purely
infinite.
\end{example}

We now briefly recall the groupoid approach to $k$-graph $C^*$-algebras. We do this in
order to prove Lemma~\ref{lem.top.principal}, which characterises which $k$-graphs
$\Lambda$ have the property that for every positive $b \in C^*(\Lambda)$, there is a
vertex projection $s_v \in C^*(\Lambda)$ that is Cuntz-below $b$.

Following \cite{BroClaSie}, a \emph{topological groupoid} $\Gg$ is a small category in
which every morphism is invertible, equipped with a topology in which composition and
inversion are continuous. The set of objects in $\Gg$ is denoted by $\Gg^{(0)}$. Each
morphism $\gamma$ in the category has a range and source denoted $r(\gamma)$ and
$s(\gamma)$ respectively. A topological groupoid is \emph{\'etale} if $s$ is a local
homeomorphism. To each locally compact, Hausdorff, \'etale groupoid $\Gg$ one can
associate the reduced and the full $C^*$-algebras denoted $C^*_r(\Gg)$ and $C^*(\Gg)$
respectively. Let $\pi\colon  C^*(\Gg)\to  C^*_r(\Gg)$ denote the standard surjection. A
topological groupoid is \emph{topologically principal} if the set $\{u\in \Gg^{(0)} :u\Gg
u =\{u\}\}$ of units with trivial isotropy is dense in $\Gg^{(0)}$. This can be
characterised $C^*$-algebraically as follows.

\begin{lemma}[{\cite[Proposition~5.5]{MR3189105},\cite[Theorem~2]{MR1258035}}]
\label{lem.top.free} Let $\Gg$ be second-countable, locally compact, Hausdorff, \'etale
groupoid.  Then $\Gg$ is topologically principal  if and  only if every  ideal of
$C^*(\Gg)$, which  has  zero intersection with  $C_0(\Gg^{(0)})$, is contained in
$\ker\pi$.
\end{lemma}

\begin{proof}
The ``only if'' implication follows from \cite[Proposition~5.5(1)]{MR3189105}. For ``if'' suppose that $\Gg$ is not topologically principal. Then the proof of
\cite[Proposition~5.5(2)]{MR3189105} constructs an ideal $I$ of $C^*(\Gg)$ which has zero
intersection with $C_0(\Gg^{(0)})$, and $f,f_0\in C_c(\Gg)\subseteq C^*(\Gg)$ such that $0\neq f-f_0\in I$.
Since $\pi|_{C_c(\Gg)}$ is the identity map, $\pi(f-f_0)\neq 0$, so $I$ has zero intersection with
$C_0(\Gg^{(0)})$, but is not contained in $\ker\pi$.
\end{proof}

Following \cite{MR1745529}, let $\Lambda$ be a row-finite $k$-graph with no sources. Let
$\Gg_\Lambda$ denote the graph groupoid of $\Lambda$ \cite[Definition 2.7]{MR1745529}.
Then $\Gg_\Lambda$ is a second-countable, locally compact, Hausdorff, \'etale groupoid such that
$C^*(\Gg_\Lambda)\cong C^*_r(\Gg_\Lambda)$ and $C^*(\Lambda)\cong C^*(\Gg_\Lambda)$ via
an isomorphism sending $s_\lambda s_\mu^*$ to $1_{Z(\lambda,\mu)}$ (see \cite[Corollary
3.5]{MR1745529}). The $k$-graph $\Lambda$ is said to satisfy the \emph{aperiodicity
condition} if for every vertex $v \in \Lambda^0$ there is an infinite path $x$ with range
$v$ such that $\sigma^p(x)\neq \sigma^q(x)$ for all distinct $p,q\in \NN^k$. For more
details see \cite{MR1745529}.

\begin{lemma}\label{lem.top.principal}
Let $\Lambda$ be a row-finite $k$-graph with no sources. Then the following are equivalent:
\begin{enumerate}
\item\label{lem.top.principal.i} The groupoid $\Gg_\Lambda$ is topologically principal;
\item\label{lem.top.principal.ii} for every nonzero $b\in C^*(\Gg_\Lambda)_+$  there exist nonzero $x\in
    C_0(\Gg^{(0)}_\Lambda)_+$ with $x\precsim  b$;
\item\label{lem.top.principal.iii} for every nonzero $b\in C^*(\Lambda)_+$ there exist $v\in \Lambda^0$
    satisfying $s_v\precsim  b$;
\item\label{lem.top.principal.iv} every nonzero ideal of  $C^*(\Gg_\Lambda)$ has nonzero intersection with
    $C_0(\Gg^{(0)}_\Lambda)$; and
\item\label{lem.top.principal.v} the $k$-graph $\Lambda$ satisfies the aperiodicity condition.
\end{enumerate}
\end{lemma}
\begin{proof}
For \eqref{lem.top.principal.i}$\Rightarrow$\eqref{lem.top.principal.ii} we refer to \cite[Lemmas~3.1 and 3.2]{BroClaSie}.

For \eqref{lem.top.principal.ii}$\Rightarrow$\eqref{lem.top.principal.iii}, recall that the compact sets $Z(\mu):=\{x\in \Lambda^\infty:
x(0,d(\mu))=\mu\}$ are a basis for the topology on $\Lambda^{\infty}\cong
\Gg^{(0)}_\Lambda$ and there is an isomorphism $\ospan\{s_\mu s_\mu^*\} \cong
C_0(\Gg^{(0)}_\Lambda)$ sending each $s_\mu s_\mu^*$ to $1_{Z(\mu)}$, see \cite{MR1745529}. So
for each nonzero positive $x\in \ospan\{s_\mu s_\mu^*\} \cong C_0(\Gg^{(0)}_\Lambda)$
there exists $\mu\in\Lambda$ such that $s_\mu s_\mu^*\precsim x$ (because $\supp f
\subseteq \supp g\Leftrightarrow f \precsim g$ for functions $f,g\geq 0$).
Hence $s_{s(\mu)}=s_\mu^* s_\mu\precsim s_\mu s_\mu^*\precsim x$.

For \eqref{lem.top.principal.iii}$\Rightarrow$\eqref{lem.top.principal.iv}, fix a nonzero ideal $J$ of $C^*(\Gg_\Lambda)$. Identifying
$C^*(\Gg_\Lambda)$ with $C^*(\Lambda)$, select nonzero $b\in J_+$ and $v\in \Lambda^0$
such that $s_v\precsim  b$. Then $x_kbx_k^*\to s_v$ for some $x_k\in C^*(\Gg_\Lambda)$.
In particular $J\cap C_0(\Gg^{(0)}_\Lambda)$ is nonzero since it contains $s_v$.

Since $C^*(\Gg_\Lambda)\cong C^*_r(\Gg_\Lambda)$ the equivalence \eqref{lem.top.principal.i}$\Leftrightarrow$\eqref{lem.top.principal.iv}
follows from Lemma~\ref{lem.top.free}, and the equivalence \eqref{lem.top.principal.i}$\Leftrightarrow$\eqref{lem.top.principal.v} is
contained in \cite[Proposition~4.5]{MR1745529}.
\end{proof}

Let $A$ be a $C^*$-algebra. Following \cite{MR1759891}, a nonzero element $a\in A_+$ is
\emph{properly infinite} in $A$ if $a\oplus a \precsim a$. By
\cite[Theorem~4.16]{MR1759891}, $A$ is purely infinite if and only if every nonzero
positive element in $A$ is properly infinite.

\begin{proof}[Proof of Theorem~\ref{thm3}] Let $\Lambda$ be a row-finite $k$-graph with no
sources such that $C^*(\Lambda)$ is simple and the semigroup $S(\Lambda)$ of Definition
\ref{def.S.Lambda} is almost unperforated. Let $(A,\ZZ^k,\sigma)$ be the $C^*$-dynamical
system associated to $\Lambda$ as in Definition \ref{defn.dynamical}.

We verify the hypotheses of Proposition~\ref{prop.main}. By
\cite[Theorem~5.5]{MR1745529}, $A$ is an AF-algebra. Since $\Lambda$ is a countable
category, $A$ is separable. Hence $A$ is $\sigma$-unital, exact, of real rank zero
(\cite[p.~140]{MR1120918}) and with cancellation (\cite[p.~131]{MR1783408}).  Since
ideal structure is preserved under Morita equivalence,  $C^*(\Lambda)$ is simple if and
only if $A\rtimes_{\sigma,r} \ZZ^k$ is simple. Finally, $S(\Lambda)\cong
S(A,\ZZ^k,\sigma)$ is almost unperforated by Lemma~\ref{lem.S.isom}.

Now, rather than verifying that $A_+\smallspace\setminus\{0\}$ is a cofinal subset of
$(A\rtimes_{\sigma,r} \ZZ^k)_+\smallspace\setminus\{0\}$ used to prove
\eqref{prop.main.(1)}$\Rightarrow$\eqref{prop.main.(2)} of Proposition~\ref{prop.main}, we verify \eqref{prop.main.(1)}$\Rightarrow$\eqref{prop.main.(2)}
directly. Suppose that $S(A,\ZZ^k,\sigma)$ is purely infinite. As noted previously in the
proof of Lemma~\ref{lem.pi}, this ensures each nonzero $p\in \Pp(A)$ is properly infinite
in $A\rtimes_{\sigma,r} \ZZ^k$. Fix a nonzero positive element $b\in
P(A\rtimes_{\sigma,r} \ZZ^k)P$. Let $\varphi$ be the isomorphism   from
$P(A\rtimes_{\sigma,r} \ZZ^k)P$ into $C^*(\Lambda)$ introduced before
Definition~\ref{defn.dynamical}. Since $C^*(\Lambda)$ is simple it follows from
Lemma~\ref{lem.top.principal} that there exists $v\in \Lambda^0$ such that
$\varphi(s_{(v,0)})=s_v\precsim \varphi(b)$. Therefore $s_{(v,0)} \precsim b$ in
$P(A\rtimes_{\sigma,r} \ZZ^k)P$. Since $s_{(v,0)}\in \Pp(A)$, we see that $s_{(v,0)}$ is properly
infinite in $A\rtimes_{\sigma,r} \ZZ^k$. So  \cite[Lemma~2.2]{MR1759891} shows that $s_{(v,0)}$ is
properly infinite in $P(A\rtimes_{\sigma,r} \ZZ^k)P$. Applying
\cite[Lemma~3.8]{MR1759891} to $s_{(v,0)}\precsim b$ in  $P(A\rtimes_{\sigma,r} \ZZ^k)P$
we see that $b$ is properly infinite in $P(A\rtimes_{\sigma,r} \ZZ^k)P$. Using
\cite[Theorem~4.16]{MR1759891} it follows that $P(A\rtimes_{\sigma,r} \ZZ^k)P$, and
therefore also $A\rtimes_{\sigma,r} \ZZ^k$, is purely infinite.

So Proposition~\ref{prop.main} implies that $A\rtimes_{\sigma,r} \ZZ^k$ is either purely
infinite or stably finite. Since pure infiniteness (resp.\ stable finiteness for
$C^*$-algebras with an approximate identity of projections) is preserved under stable
isomorphism, we conclude that $C^*(\Lambda)$ is either purely infinite or stably finite.
\end{proof}

\section{Stable finiteness of twisted \texorpdfstring{$k$}{k}-graph algebras and the proof of Theorem~\ref{thm.stably.finite}}
\label{section.ThmE} Recent work \cite{MR3507995} of Clark, an Huef and the third-named
author characterises when a $C^*$-algebra of a cofinal $k$-graph is stably finite. In
\cite[Theorem 1.1.(1)]{MR3507995} they show the following.
\begin{thm}[Clark, an Huef, Sims]\label{thm.stably.finite.untwisted}
Let $\Lambda$ be a row-finite and cofinal $k$-graph with no sources, and coordinate
matrices $A_{e_1}, \dots , A_{e_k}$. Then the following are equivalent:
\begin{enumerate}
\item $C^*(\Lambda)$ is quasidiagonal;
\item $C^*(\Lambda)$ stably finite;
\item $\big(\sum_{i=1}^k\image(1-A^t_{e_i}) \big) \cap \NN{\Lambda^0} = \{0\}$; and
\item $\Lambda$ admits a faithful graph trace.
\end{enumerate}
\end{thm}
In this section we show that Theorem \ref{thm.stably.finite.untwisted} remains valid for
twisted $k$-graph $C^*$-algebras $C^*(\Lambda,c)$ in place of $C^*(\Lambda)$  (see
Theorem \ref{thm.stably.finite}). As in \cite{MR3507995}, we restrict our attention to
$k$-graphs $\Lambda$ which are \emph{cofinal}, meaning that for all pairs $v,w\in
\Lambda^0$ there exists $n\in \NN^k$ such that $s(w\Lambda^n)\subseteq s(v\Lambda)$, cf.~\cite{MR3444442}.

Let us now introduce the terminology of Theorem \ref{thm.stably.finite}. Following
\cite[p.~961]{MR3507995}, a $C^*$-algebra $A$ is \emph{quasidiagonal} if it admits a
faithful quasidiagonal representation. Recall that $A$ is stably finite if $A \otimes
\Kk$ is finite. Let $\Lambda$ be a row-finite $k$-graph with no sources. Following
\cite{MR2434188}, a \emph{graph trace} on $\Lambda$ is a function $\tau \colon
\Lambda^0\to \RR^+$ satisfying $\tau(v)=\sum_{\lambda\in v\Lambda^n} \tau(s(\lambda))$
for all $v\in \Lambda^0$ and all $n\in \NN^k$. A graph trace is \emph{faithful} if
$\tau(v)\neq 0$ for each $v\in \Lambda^0$. Following \cite{MR3311883}, a \emph{2-cocycle}
on $\Lambda$ is a map from the set of composable 2-tuples $\Lambda^{*2}=\{(\lambda,\mu) :
s(\lambda) = r(\mu)\}$ into $\TT$ such that $c(\lambda, \mu) = 1$ whenever $\lambda$ or
$\mu$ is a vertex, and $c(\lambda, \mu)c(\lambda\mu, \nu) = c(\mu, \nu)c(\lambda,
\mu\nu)$ for each composable 3-tuple $(\lambda, \mu, \nu)$. The set of 2-cocycles on
$\Lambda$ is denoted $Z^2(\Lambda,\TT)$. The $C^*$-algebra $C^*(\Lambda,c)$ associated to
a row-finite $k$-graph $\Lambda$ with no sources and a cocycle $c\in Z^2(\Lambda,\TT)$ is
the universal $C^*$-algebra generated by elements $\{s_\lambda : \lambda \in \Lambda \}$
satisfying (CK1), (CK3), (CK4), and the twisted version, $s_\mu s_\nu = c(\mu,
\nu)s_{\mu\nu}$ whenever $s(\mu) = r(\nu)$, of (CK2).

\begin{remark}\label{easy.graphtrace}
By \cite[Lemma~A.5]{MR2434188}, a function $\tau \colon \Lambda^0\to \RR^+$ on a
row-finite $k$-graph $\Lambda$ with no sources is a graph trace if and only if
$\tau(v)=\sum_{\lambda\in v\Lambda^{e_i}} \tau(s(\lambda))$ for all $v\in \Lambda^0$ and
all $1\leq i\leq k$.
\end{remark}

We will need four technical lemmas to prove Theorem~\ref{thm.stably.finite}. For the
proof of the first lemma, recall that for a $k$-graph $\Lambda$, a subset $H$ of
$\Lambda^0$ is called \emph{hereditary} if $s(H\Lambda)\subseteq H$. For the notion of a
\emph{saturated} subset of $\Lambda^0$ we refer to \cite{MR2323468}. Recall that an element of
a $C^*$-algebra $A$ is full if it is not contained in any proper (closed two-sided) ideal
of $A$.

\begin{lemma}\label{lem.3.5.fix}
Let $\Lambda$ be a row-finite $k$-graph with no sources. Let $c$ be a 2-cocycle on
$\Lambda$. Then $\Lambda$ is cofinal if and only if every vertex projection $s_v$ is full
in $C^*(\Lambda,c)$.
\end{lemma}
\begin{proof}
The proof of \cite[Proposition~3.4]{MR2323468} for $c=1$ generalises; we give a short
outline. By the proof of \cite[Proposition~5.1]{MR1777234} and
\cite[Proposition~3.4]{MR2323468}, $\Lambda$ is cofinal if and only if $\emptyset$ and
$\Lambda^0$ are the only saturated hereditary subsets of $\Lambda^0$. The result now
follows from \cite[Theorem~4.6]{MR3262073}. The definition of $C^*(\Lambda,c)$ in
\cite{MR3262073} agrees with our definition by \cite[Lemma~2.2]{MR3311883}.
\end{proof}

Let $\Lambda$ be a row-finite $k$-graph with no sources. Let $c$ be a 2-cocycle on
$\Lambda$. Let $\Gamma$ be the skew-product $k$-graph $\Lambda \times_d \ZZ^k$ of
$\Lambda$ and $\ZZ^k$, $\phi\colon \Gamma \to \Lambda$ the functor  $\phi(\lambda, n) =
\lambda$, and $\tilde{c}\in Z^2(\Gamma,\TT)$ the composition $c\circ \phi$. Set
$A:=C^*(\Gamma, \tilde{c})$. Similarly to the case $c=1$ (see Definition
\ref{defn.dynamical}), there is an action $\sigma$ of $\ZZ^k$ on $A$ such that
$\sigma_m(s_{(\lambda,n)})=s_{(\lambda,n-m)}$ for all $\lambda \in\Lambda$ and $m, n\in
\ZZ^k$. The projection $P=\sum_{v\in \Lambda^0}s_{(v,0)}\in \Mm(A)\subseteq
\Mm(A\rtimes_{\sigma,r}\ZZ^k)$ is such that $A\rtimes_{\sigma,r} \ZZ^k$ stably isomorphic
to $C^*(\Lambda,c)$ via an isomorphism $\varphi\colon P(A\rtimes_{\sigma,r} \ZZ^k)P\to C^*(\Lambda,c)$ satisfying $\varphi(s_{(v,0)})=s_v$ for all $v\in \Lambda^0$.
\begin{defn}\label{system.twisted}
We call $(A,\ZZ^k,\sigma)$ \emph{the $C^*$-dynamical system associated to $(\Lambda, c)$}.
\end{defn}

Let $X$ be a non-empty set $X$. Following \cite{MR3335414}, we write $\Kk_X$ for the
unique nonzero $C^*$-algebra generated by elements $\{\theta_{x,y} : x, y \in X\}$
satisfying $\theta_{x,y}^*=\theta_{y,x}$ and
$\theta_{x,y}\theta_{w,z}=\delta_{y,w}\theta_{x,z}$. We call the $\theta_{x,y}$ the
\emph{matrix units} for $\Kk_X$. Recall that for a $C^*$-dynamical system
$(A,\ZZ^k,\sigma)$  the action $\sigma$ is minimal if $A$ admits no nontrivial
invariant ideals.

\begin{lemma}\label{lem.minimal}
Let $\Lambda$ be a row-finite $k$-graph with no sources. Let $c$ be a 2-cocycle on
$\Lambda$. Let $(A,\ZZ^k,\sigma)$ be the $C^*$-dynamical system associated to
$(\Lambda, c)$. Suppose that $\Lambda$ is cofinal. Then $\sigma$ is minimal.
\end{lemma}
\begin{proof}
Let $I$ be an nonzero invariant ideal of $A$. Define $b \colon \Gamma^0 \to \ZZ^k$ by
$b(v, n) = n$. Using \cite[Lemma~8.4]{MR3335414} and its proof we have an isomorphism
$A\cong \overline{\bigcup_n \bigoplus_{b(w)=n} \Kk_{\Gamma w}}$ via an isomorphism
satisfying $s_\mu s_\nu^*\mapsto \theta_{\mu,\nu} \in \Kk_{\Gamma w}$ for each $w\in
\Lambda^0\times \{n\}\subseteq \Gamma^0$. Since $I$ is nonzero, $I\cap \Kk_{\Gamma w}\neq
\{0\}$ for some $n,w$ with $b(w)=n$. In particular $s_\mu s_\mu^*\in I$ for some $n,w$
with $b(w)=n$ and $\mu\in \Gamma w$ because $\Kk_{\Gamma w}$ is simple. This implies that
$s_{s(\mu)}\in I$. We have that $\mu=(\lambda,n)$ for some $\lambda$ and then $s(\mu)=
(s(\lambda), n + d(\lambda))$ (see \cite[Definition 5.1]{MR1745529}). Since $I$ is
invariant, $s_{(s(\lambda),0)}\in I$. We deduce that the ideal
$J:=\varphi(P(I\rtimes_{\sigma,r} \ZZ^k)P)$ in $C^*(\Lambda,c)$ contains
$s_{s(\lambda)}=\varphi(s_{(s(\lambda),0)})$. Using Lemma~\ref{lem.3.5.fix}, we conclude
that $J=C^*(\Lambda,c)$, so $I=A$.
\end{proof}

Let $A$ be a $C^*$ algebra. Recall that a tracial weight on a $A$ is a map $\varphi\colon
A_+\to [0,\infty]$ such that $\varphi(\alpha a + \beta b) = \alpha \varphi(a) + \beta
\varphi(b)$ and $\varphi(d^*d) = \varphi(dd^*)$ for all $a,b\in A_+$ all $\alpha,\beta
\in \RR^+$ and all $d\in A$. Recall, given an action $\sigma$ of $\ZZ^k$ on $A$, that
$H_\sigma$ denotes subgroup of $K_0(A)$ generated by $\{(\id - \tilde\sigma_t)K_0(A) : t
\in  \ZZ^k\}$.

\begin{lemma}\label{lem.apply.prop}
Let $\Lambda$ be a row-finite and cofinal $k$-graph with no sources. Let $c$ be a
2-cocycle on $\Lambda$. Then the following are equivalent:
\begin{enumerate}
\item\label{lem.apply.prop.i} $C^*(\Lambda,c)$ admits a faithful lower-semicontinuous semifinite tracial
    weight;
\item\label{lem.apply.prop.ii}$C^*(\Lambda,c)$ is stably finite;
\item\label{lem.apply.prop.iii} $H_\sigma \cap K_0(A)^+ = \{0\}$; and
\item\label{lem.apply.prop.iv} $\Lambda$ admits a faithful graph trace.
\end{enumerate}
\end{lemma}
\begin{proof}
By \cite[Lemma~8.4]{MR3335414}, $A$ is AF, and by Lemma~\ref{lem.minimal}, $\sigma$ is
minimal. Hence we can employ Proposition~\ref{prop.main.2}.

To verify \eqref{lem.apply.prop.i}$\Rightarrow$\eqref{lem.apply.prop.ii} we use the argument in Proposition~\ref{prop.main.2},
\eqref{prop.main.2.(1)}$\Rightarrow$\eqref{prop.main.2.(2)}, replacing $A\rtimes_{\sigma,r} \ZZ^k$ by $C^*(\Lambda,c)$. For
\eqref{lem.apply.prop.ii}$\Rightarrow$\eqref{lem.apply.prop.iii}, assume that $C^*(\Lambda,c)$ is stably finite. Since $C^*(\Lambda,c)$
and ${A\rtimes_{\sigma,r} \ZZ^k}$ are stably isomorphic and have approximate identities consisting
of projections, $A\rtimes_{\sigma,r} \ZZ^k$ is also stably finite. By
Proposition~\ref{prop.main.2}, \eqref{prop.main.2.(2)}$\Rightarrow$\eqref{prop.main.2.(3)} we get $H_\sigma \cap K_0(A)^+ =
\{0\}$. We now show \eqref{lem.apply.prop.iii}$\Rightarrow$\eqref{lem.apply.prop.i}. Since $H_\sigma \cap K_0(A)^+ = \{0\}$ we
conclude from Proposition~\ref{prop.main.2}, \eqref{prop.main.2.(3)}$\Rightarrow$\eqref{prop.main.2.(1)} that $A\rtimes_{\sigma,r} \ZZ^k$ admits a faithful lower-semicontinuous semifinite
tracial weight. Consequently,
the corner $P(A\rtimes_{\sigma,r} \ZZ^k)P$, and hence also $C^*(\Lambda,c)$, admits a
faithful lower-semicontinuous semifinite tracial weight. Finally, \eqref{lem.apply.prop.i} and \eqref{lem.apply.prop.iv} are equivalent by \cite[Theorem~7.4]{MR3311883}.
\end{proof}

The following lemma collects some of the results in \cite{MR3335414} regarding the core
of a twisted $k$-graph $C^*$-algebra. For the definition of a skew-product, see the paragraph
preceding Definition \ref{defn.dynamical}.

\begin{lemma}\label{lem.3.4.fix}
Let $\Lambda$ be a row-finite $k$-graph with no sources. Let $c$ be a 2-cocycle on
$\Lambda$. Define $\Gamma$ to be the skew-product $k$-graph $\Lambda \times_d \ZZ^k$ of
$\Lambda$ by $\ZZ^k$. Define $\phi\colon \Gamma \to \Lambda$ by $\phi(\lambda, n) =
\lambda$. Define $b \colon \Gamma^0 \to \ZZ^k$ by $b(v, n) = n$. Let $\tilde{c}\in
Z^2(\Gamma,\TT)$ be the composition $c\circ \phi$.
\begin{enumerate}
\item\label{lem.3.4.fix.i} For each $n\in \ZZ^k$ let
    $B_n:=\ospan\{s_{(\lambda,n-d(\lambda))}s_{(\gamma,n-d(\gamma))}^*:
    s(\lambda)=s(\gamma)\}$. Then $B_n$ is a $C^*$-subalgebra of $C^*(\Gamma,
    \tilde{c})$ and there exists an isomorphism $\pi_n\colon B_n\to \bigoplus_{v\in
    \Lambda^0} \Kk_{\Lambda v}$ such that
    $\pi_n(s_{(\lambda,n-d(\lambda))}s_{(\gamma,n-d(\gamma))}^*)=\theta_{\lambda,\gamma}\in
    \Kk_{\Lambda s(\lambda)}\subseteq \bigoplus_{v\in \Lambda^0} \Kk_{\Lambda v}$.
\item\label{lem.3.4.fix.ii} For each $m\leq n\in \ZZ^k$ the endomorphism $j_{m,n}\colon\bigoplus_{v\in
    \Lambda^0} \Kk_{\Lambda v}\to \bigoplus_{v\in \Lambda^0} \Kk_{\Lambda v}$,
    defined by $j_{m,n}(\theta_{\lambda,\gamma})=\sum_{\alpha\in
    s(\lambda)\Lambda^{n-m}}c(\lambda,\alpha)\overline{c(\gamma,\alpha)}\theta_{\lambda\alpha,\gamma\alpha}$
    satisfies $j_{m,n}\circ \pi_m=\pi_n$.
\item\label{lem.3.4.fix.iii} If $\gamma$ denotes the gauge action in the sense of \cite{MR3335414},
    then $C^*(\Lambda,c)\rtimes_\gamma \TT^k \cong C^*(\Gamma,\tilde{c})$.
\end{enumerate}
\end{lemma}

\begin{proof}
Property \eqref{lem.3.4.fix.i} simply rephrases parts of the proof of \cite[Lemma~8.4]{MR3335414}. There
the $n$th $C^*$-algebra in the inductive limit for $C^*(\Gamma, \tilde{c})$ is shown to
be isomorphic to $\bigoplus_{b(w)=n} \Kk_{\Gamma w}$. Since the elements $\{w\in
\Gamma^0: b(w)=n\}=\Lambda^0\times \{n\}$ are in bijective correspondence with $\Lambda^0$ it follows that the
$n$th $C^*$-algebra $B_n$ is isomorphic to $\bigoplus_{v\in \Lambda^0} \Kk_{\Lambda v}$.

For part \eqref{lem.3.4.fix.ii}, take any  $m\leq n\in \ZZ^k$. It is easy to see that $j_{m,n}$ is an
endomorphism: Since
$j_{m,n}(\theta_{\lambda,\gamma})j_{m,n}(\theta_{\eta,\zeta})=\sum_{\alpha\in
s(\lambda)\Lambda^{n-m}}\sum_{\beta\in
s(\eta)\Lambda^{n-m}}c(\lambda,\alpha)\overline{c(\gamma,\alpha)}c(\eta,\beta)\overline{c(\zeta,\beta)}\theta_{\lambda\alpha,\gamma\alpha}\theta_{\eta\beta,\zeta\beta}$
only has nonzero terms if $\gamma=\eta$ and $\alpha=\beta$ the product
$j_{m,n}(\theta_{\lambda,\gamma})j_{m,n}(\theta_{\eta,\zeta})$ collapses to the
expression $\delta_{\gamma,\eta}\sum_{\alpha\in
s(\lambda)\Lambda^{n-m}}c(\lambda,\alpha)\overline{c(\zeta,\alpha)}\theta_{\lambda\alpha,\zeta\alpha}=\delta_{\gamma,\eta}j_{m,n}(\theta_{\lambda,\zeta})=j_{m,n}(\theta_{\lambda,\gamma}\theta_{\eta,\zeta})$.
To see that $j_{m,n}\circ \pi_m=\pi_n$ recall that $B_m \subseteq B_n$, which follows from the
observation that
$$s_{(\lambda,m-d(\lambda))}s_{(\gamma,m-d(\gamma))}^*=\sum_{\alpha\in s(\lambda)\Lambda^{n-m}}c(\lambda,\alpha)\overline{c(\gamma,\alpha)}s_{(\lambda\alpha,n-d(\lambda\alpha))}s_{(\gamma\alpha,n-d(\gamma\alpha))}^*.$$
From this equality we obtain that $j_{m,n}\circ \pi_m=\pi_n$. Part \eqref{lem.3.4.fix.iii}
follows from \cite[Lemma~8.5]{MR3335414}, as pointed out in the proof of \cite[Corollary
8.7]{MR3335414}.
\end{proof}

\begin{proof}[Proof of Theorem~\ref{thm.stably.finite}]
The implication \eqref{thm.stably.finite.(a)}$\Rightarrow$\eqref{thm.stably.finite.(b)} is valid for any
$C^*$-algebra. The implications \eqref{thm.stably.finite.(b)}$\Leftrightarrow$\eqref{thm.stably.finite.(c)}$\Leftrightarrow$\eqref{thm.stably.finite.(d)} follow from
Lemma~\ref{lem.apply.prop}, and \cite[Lemma~3.4]{MR3507995} with the following small modification.
Using \cite[Lemma~3.4]{MR3507995} we have $H_\sigma \cap K_0(A)^+ = \{0\}$ if and only if
$(\sum_{i=1}^k\image(1-A^t_{e_i})) \cap \NN{\Lambda^0} = \{0\}$ when $c=1$. When $c\neq 1$ we
can use the same ideas, but there are three places in the proof of
\cite[Lemma~3.4]{MR3507995} where care is required. Those are the references
to \cite[Lemma~4.1]{MR3011251} used to construct isomorphisms $\pi_n$ ($n\in \ZZ^k$),
\cite[Theorem~4.2]{MR3011251} used to construct endomorphisms $j_{m,n}$ ($m\leq n\in
\ZZ^k$)(\footnote{There is a typo in the statement of \cite[Theorem~4.2]{MR3011251},
replace $m\leq n\in \NN^k$ by $m\leq n\in \ZZ^k$.}), and \cite[Corollary 5.3]{MR3507995}
used to identify $C^*(\Lambda\times_d \ZZ^k)$ with a crossed product. When $c\neq 1$ we
can instead use the three parts of Lemma~\ref{lem.3.4.fix} respectively. The crucial
consequence of Lemma~\ref{lem.3.4.fix} is that the action of $\ZZ^k$ on
$K_0(C^*(\Lambda\times_d \ZZ^k,\tilde{c}))^+$ remains unchanged when we change the
2-cocycle on $\Lambda$.

Finally we prove \eqref{thm.stably.finite.(d)}$\Rightarrow$\eqref{thm.stably.finite.(a)}. The result for the case $c=1$ is
\cite[Theorem~3.7]{MR3507995}.  It relies on \cite[Lemma~2.1]{MR3507995} and refers to
four results from elsewhere. Clearly \cite[Lemma~2.1]{MR3507995} applies unchanged to the
twisted case. As to the four referred results only the fourth one is a general
$C^*$-algebra result. The other three are the facts that a faithful graph trace on
$\Lambda$ induces a faithful semifinite trace on $C^*(\Lambda)$, that every vertex
projection $s_v$ ($v\in \Lambda^0$) in $C^*(\Lambda)$ is full when $\Lambda$ is cofinal,
and that $C^*(\Lambda)$ is nuclear and in the UCT class. All these results generalise
to the twisted case, see \cite[Theorem~7.4]{MR3311883}, Lemma~\ref{lem.3.5.fix}, and
\cite[Corollary 8.7]{MR3335414} respectively.
\end{proof}

\begin{remark}\label{rem.action}
As mentioned in the first paragraph of the proof of Theorem~\ref{thm.stably.finite}, the
action of $\ZZ^k$ on $K_0(C^*(\Lambda\times_d \ZZ^k,\tilde{c}))^+$ remains unchanged when
we change the 2-cocycle $c$ on $\Lambda$.
\end{remark}

\begin{remark} An alternative way to prove Theorem~\ref{thm.stably.finite} would be to more closely follow the proof of \cite[Theorem~1.1]{MR3507995}. But it seems worthwhile to record an alternative proof using our Proposition~\ref{prop.main.2} even in the untwisted setting.
\end{remark}

\section{Purely infinite twisted \texorpdfstring{$k$}{k}-graph algebras  and the proof of Theorem~\ref{thm4}}

In this section we prove Theorem~\ref{thm4}. The proof essentially consists of seven
lemmas and the proof of Theorem~D\textsuperscript{$\prime$}. Lemma~\ref{lem.cond.exp} summarises
basic facts about twisted groupoid $C^*$-algebras. Lemma~\ref{lem.top.principal.cycle} connects aperiodicity
of a $k$-graph to the notion of Cuntz-below generalising parts of
Lemma~\ref{lem.top.principal}. Theorem~D\textsuperscript{$\prime$} is a generalisation of
Theorem~\ref{thm3} to the twisted setting. Lemmas \ref{lem.connected}--\ref{lem.loop} are algebraic in nature
and focus on properties of the semigroup $S(\Lambda)$. This includes considering $k$-graphs $\Lambda$ that are strongly connected, or $k$-graphs that are cofinal and either have a nontrivial hereditary subset of $\Lambda^0$, or contain a cycle whose degree is coordinatewise nonzero.

Let $\Gg$ be a locally compact Hausdorff \'etale groupoid. Following the terminology in
\cite{MR3519045}, we let $Z^2(\Gg,\TT)$ denote the set of all continuous maps $\sigma$
from the set of composable 2-tuples $\Gg^{(2)} = \{(\lambda,\mu) : s(\lambda) = r(\mu)\}$
into $\TT$ such that $\sigma(\lambda,\mu) = 1$ whenever $\lambda$ or $\mu$ is in
$\Gg^{(0)}$ and $\sigma(\lambda, \mu)\sigma(\lambda \mu, \nu) = \sigma(\mu,
\nu)\sigma(\lambda, \mu\nu)$ for each composable 3-tuple $(\lambda, \mu, \nu)$. We will
omit the word continuous henceforth and call such maps \emph{$2$-cocycles}. Let $C_c(\Gg,
\sigma)$ denote the $^*$-algebra of compactly supported continuous functions from $\Gg$
to $\CC$ with convolution and involution defined by
$$(fg)(\gamma) = \sum_{\eta\zeta = \gamma}\sigma(\eta,\zeta)f(\eta)g(\zeta),  \ \ \ \text{and} \ \ \ f^*(\gamma) = \overline{\sigma(\gamma, \gamma^{-1})} \overline{f(\gamma^{-1})}.$$
The completing of $C_c(\Gg,\sigma)$ with respect to the reduced and the full $C^*$-norm are called the reduced and the full twisted groupoid $C^*$-algebra
denoted $C^*_r(\Gg,\sigma)$ and $C^*(\Gg,\sigma)$, respectively, see \cite{MR584266}.

\begin{lemma}[cf.~\cite{MR584266}]\label{lem.cond.exp}
Let $\Gg$ be a locally compact, Hausdorff and \'etale groupoid, and let $\sigma$ be a
$2$-cocycle on $\Gg$. Then
\begin{enumerate}
\item\label{lem.cond.exp.i} The extension map from $C_c(\Gg^{(0)})$ into $C_c(\Gg,\sigma)$ (where a
    function is defined to be zero on $\Gg \setminus  \Gg^{(0)}$) extends to an
    embedding of $C_0(\Gg^{(0)})$ into $C^*_r(\Gg,\sigma)$.
\item\label{lem.cond.exp.ii}  The restriction map $E_0: C_c(\Gg,\sigma)\to C_c(\Gg^{(0)})$ extends to a
    conditional expectation $E\colon  C^*_r(\Gg,\sigma)\to  C_0(\Gg^{(0)})$.
\item\label{lem.cond.exp.iii}  The contraction $E$ in \emph{(2)} is faithful. That is,
    $E(b^*b)=0\Rightarrow b=0$ for $b\in C^*_r(\Gg,\sigma)$.
\item\label{lem.cond.exp.iv}  For every $\varepsilon > 0$ and $c \in C^*_r(\Gg,\sigma)_+$, there exists
    $f\in C_0(\Gg^{(0)})_+$ such that $\|f\|=1$, $\|fcf-fE(c)f\|<\varepsilon$, and
    $\|fE(c)f\|>\|E(c)\|-\varepsilon$.
\end{enumerate}
\end{lemma}
\begin{proof}
The proofs of \eqref{lem.cond.exp.i}--\eqref{lem.cond.exp.iii} are contained in \cite{MR584266} (see also
\cite[note~\dag]{arxiv.1402.7126}). For the reader not familiar with regular
representations or quasi-invariant probability measures, see also
\cite[Lemma~2.1]{BroClaSie}; its proof generalises to the untwisted case.

For part \eqref{lem.cond.exp.iv}, we refer to \cite[Lemma~3.1]{BroClaSie}. Its proof also generalises to
the untwisted case (because for $b \in C_c(\Gg, \sigma)\cap C^*(\Gg,\sigma)_+$, ${f\in
C_c(\Gg^{(0)})}$, $K=\supp(b-E(b))$ and $V\subseteq \Gg^{(0)}$ such that
$\supp(f)\subseteq V$ and  $VKV=\emptyset$, we still have $fbf = fE(b)f$).
\end{proof}

\begin{lemma}\label{lem.top.principal.cycle}
Let $\Lambda$ be a row-finite $k$-graph with no sources and let $c$ be a 2-cocycle on
$\Lambda$. Suppose that $\Lambda$ satisfies the aperiodicity condition. Then for every
nonzero $b\in C^*(\Lambda, c)_+$ there exists $v\in \Lambda^0$ satisfying $s_v\precsim
b$.
\end{lemma}
\begin{proof}
We must check that the implications
\eqref{lem.top.principal.v}$\Rightarrow$\eqref{lem.top.principal.i}$\Rightarrow$\eqref{lem.top.principal.ii}$\Rightarrow$\eqref{lem.top.principal.iii}
of Lemma~\ref{lem.top.principal} all generalise to the twisted $C^*$-algebras.

Clearly \eqref{lem.top.principal.v}$\Rightarrow$\eqref{lem.top.principal.i} needs no
adjustment as it only deals with properties of $\Lambda$.  The implication
\eqref{lem.top.principal.i}$\Rightarrow$\eqref{lem.top.principal.ii} relies on
\cite[Lemmas~3.1 and 3.2]{BroClaSie}. The first of these lemmas generalises to the
twisted case by Lemma~\ref{lem.cond.exp}, and the second is of general nature and its
proof needs only an introduction of the 2-cocycle into its wording. Finally
\cite[Corollary 7.9]{MR3335414} shows that there exists a 2-cocycle $\sigma_c$ on
$\Gg_\Lambda$ for which there is an isomorphism $C^*(\Lambda, c)\cong
C^*(\Gg_\Lambda,\sigma_c)$ that carries each $s_\mu s_\mu^*$ to $1_{Z(\mu)}\in
C_0(\Gg^{(0)}_\Lambda)$. Using this isomorphism, it is easy to check that
\eqref{lem.top.principal.ii}$\Rightarrow$\eqref{lem.top.principal.iii} generalises to the
twisted case.
\end{proof}

We can now strengthen Theorem \ref{thm3} twofold. Firstly we incorporate a 2-cocycle into
our result giving a dichotomy theorem for twisted $k$-graph algebras. Secondly, relying
on Theorem \ref{thm.stably.finite} and $K$-theory computations of section
\ref{section.ThmE}, the dichotomy of Theorem \ref{thm3} is linked to a purely algebraic
property of $\Lambda$, see \eqref{thm.d*.(4)} below.

Recall that a $k$-graph $\Lambda$ is cofinal if for every pair $v,w\in \Lambda^0$ there
exists $n\in \NN^k$ such that $s(w\Lambda^n)\subseteq s(v\Lambda)$, and $\Lambda$
satisfies the aperiodicity condition if for every vertex $v \in \Lambda^0$ there is an
infinite path $x$ with range $v$ such that $\sigma^p(x)\neq \sigma^q(x)$ for all distinct
$p,q\in \NN^k$ (see \cite{MR1745529}). By \cite[Theorem 3.1]{MR2323468} $\Lambda$ is
cofinal and satisfies the aperiodicity condition if and only if $C^*(\Lambda)$ is simple; and then
\cite[Corollary 8.2]{MR3335414} implies that $C^*(\Lambda, c)$ is simple for every 2-cocycle
$c$ on $\Lambda$.

\begin{thmd*}\label{thm.d*}
Let $\Lambda$ be a row-finite $k$-graph with no sources. Let $c$ be a 2-cocycle on
$\Lambda$. Suppose that $\Lambda$ is cofinal and satisfies the aperiodicity condition,
and that $S(\Lambda)$ is almost unperforated. Then the following are equivalent:
\begin{enumerate}
\item\label{thm.d*.(1)} $C^*(\Lambda,c)$ is purely infinite;
\item\label{thm.d*.(2)} $C^*(\Lambda,c)$ is traceless;
\item\label{thm.d*.(3)} $C^*(\Lambda,c)$ is not stably finite; and
\item\label{thm.d*.(4)} $\big(\sum_{i=1}^k\image(1-A^t_{e_i}) \big) \cap \NN\Lambda^0\neq \{0\}$.
\end{enumerate}
\end{thmd*}
\begin{proof}
We have the implications \eqref{thm.d*.(1)}$\Rightarrow$\eqref{thm.d*.(2)}$\Rightarrow$\eqref{thm.d*.(3)} by Lemma~\ref{lem.simple}.
Since $\Lambda$ is cofinal, we can apply Theorem~\ref{thm.stably.finite} to prove \eqref{thm.d*.(3)}$\Leftrightarrow$\eqref{thm.d*.(4)}.
Finally \eqref{thm.d*.(3)}$\Rightarrow$\eqref{thm.d*.(1)} is exactly the proof of Theorem~\ref{thm3} except we add a
2-cocycle into the calculation. Firstly, when $c\neq 1$ we need to use Definition
\ref{system.twisted} to obtain the appropriate $C^*$-dynamical system $(A,\ZZ^k,\sigma)$
associated to $(\Lambda, c)$. There $A=C^*(\Lambda\times_d \ZZ^k,c\circ \phi)$ with
$\phi\colon \Lambda\times_d \ZZ^k \to \Lambda$ defined by $\phi(\lambda, n) = \lambda$.
To argue that $A$ is an AF-algebra we rely on \cite[Lemma~8.4]{MR3335414} (generalising
\cite{MR1745529}). To deduce that $S(A,\ZZ^k,\sigma)$ is almost unperforated we still
rely on Lemma~\ref{lem.S.isom}, but now in combination with Remark \ref{rem.action}
ensuring that $S(A,\ZZ^k,\sigma)$ is independent of $c$. Finally, replacing the reference to Lemma~\ref{lem.top.principal} by one to Lemma~\ref{lem.top.principal.cycle} in the proof that
pure infiniteness of $S(A,\ZZ^k,\sigma)$ implies pure infiniteness of $A\rtimes_{\sigma,r} \ZZ^k$
proves \eqref{thm.d*.(3)}$\Rightarrow$\eqref{thm.d*.(1)} for $c\neq
1$.
\end{proof}

\begin{example}\label{eq.rotation.algebra}
Fix any $n\geq 1$. Let $T_k$ denote the $k$-graph with one vertex and one edge of each
colour, and $1\colon T_k\to \ZZ / n\ZZ$ the functor mapping each edge of $T_k$ to 1, the
generator of $ \ZZ / n\ZZ$, cf.~\cite[Definition 5.2]{MR3311883}. Define $\Lambda$ be
the 2-graph $T_2 \times_1  \ZZ / n\ZZ$ as illustrated on the right on Figure~\ref{pic1}.
Let us compute the semigroup $S(\Lambda)$ associated to $\Lambda$. We have
$[\delta_{v_i}]=[\delta_{v_j}]$ for any $i,j$ because $A_{e_1}\delta_{v_j}=\delta_{v_{j+1 \,
\text{mod}\,  n}}$ for each $j\in \{0,\dots,n-1\}$. It follows $S(\Lambda) \cong \NN$ is
totally ordered and stably finite (as defined in Section \ref{section.two}).

Let $A_\theta$ denote the rotation $C^*$-algebra corresponding to the angle $\theta\in
\RR$, and let $c$ be the $2$-cocycle on $\Lambda$ defined by $c(\mu,\nu)=e^{2\pi \theta
d(\mu)_2 d(\nu)_1}$. By \cite[Lemma~5.8]{MR3311883}, $$A_{n\theta}\otimes M_{n}(\CC)\cong
C^*(\Lambda,c).$$ We have that $C^*(\Lambda,c)$ is stably finite independently of the
choice of $\theta\in \RR$ and $n\geq 1$. The semigroup $S(\Lambda)$ is independent of
both the 2-cocycle and the number of vertices in the cycle on the right in
Figure~\ref{pic1}.
\end{example}

Put together, the next three lemmas show that each strongly connected row-finite $k$-graph $\Lambda$
with no sources has an unperforated semigroup $S(\Lambda)$. We recall the terminology
involved. Following \cite{MR3444442}, a row-finite $k$-graph $\Lambda$ is \emph{strongly
connected} if for all $v,w\in \Lambda^0$, there exists $n\in \NN^k$ such that
$v\Lambda^nw\neq \emptyset$. A preordered semigroup $S$ is called \emph{unperforated} if
$mx\leq my$ always implies $x\leq y$, for any $m \in \NN$ and $x, y\in S$.

\begin{lemma}\label{lem.connected}
Let $\Lambda$ be a row-finite $k$-graph with no sources. If $v\Lambda w\neq \emptyset$
for some $v,w\in \Lambda^0$ then $\class{\delta_w}\leq \class{\delta_v}$. Consequently,
if $\Lambda$ is strongly connected then $\class{\delta_w}\leq \class{\delta_v}$ for any
$v,w$.
\end{lemma}

\begin{proof} Suppose  $v\Lambda w\neq \emptyset$. Then $A^t_p\delta_v=\delta_w + a$ for some $p\in \NN^k$ and $a\in \NN\Lambda^0$. Hence $\class{\delta_v}=\class{A^t_p\delta_v}=\class{\delta_w}+\class{a}$, so $\class{\delta_w}\leq \class{\delta_v}$.
\end{proof}

We prove that $S(\Lambda)$ is unperforated when $\Lambda$ is strongly connected by considering two complementary cases.

\begin{lemma}\label{lem.connected.one}
Let $\Lambda$ be a strongly connected row-finite $k$-graph with no sources. Suppose that
$|v\Lambda^p|=1$ for all $v\in \Lambda^0$ and $p\in \NN^k$. Then $S(\Lambda)$ is
unperforated.
\end{lemma}

\begin{proof}
For each $f=\sum_{i=1}^m\delta_{v_i}\in \NN\Lambda^0$ (possibly with $v_i=v_j$ for $i\neq
j$) write $\|f\|_1=m$. Take $p\in \NN^k$. For each $i=1,\dots,m$, select $w_i$ such that
$A^t_p\delta_{v_i}=\delta_{w_i}$. Then
$$\textstyle{\|A^t_pf\|_1=\Big\|A^t_p\sum_{i=1}^m\delta_{v_i}\Big\|_1=\Big\|\sum_{i=1}^m\delta_{w_i}\Big\|_1=m=\|f\|_1.}$$
Consequently, if $f,g\in \NN\Lambda^0$ satisfy $f\approx_{\Lambda} g$, then
$\|f\|_1=\|g\|_1$. Conversely assume that $\|f\|_1=\|g\|_1$. Then, with $m=\|f\|_1$, we can express $f=\sum_{i=1}^m\delta_{v_i}$ and $g=\sum_{i=1}^m\delta_{w_i}$. Since $\Lambda$ is
strongly connected, for each $i$ there exists $p_i$ such that $v_i\Lambda^{p_i}w_i\neq
\emptyset$. So $A^t_{p_i} \delta_{v_i}\geq \delta_{w_i}$. By hypothesis
$\|A^t_{p_i}\delta_{v_i}\|_1=1$ and we conclude that $A^t_{p_i}\delta_{v_i}=\delta_{w_i}$. Hence
$f=\sum_{i=1}^m\delta_{v_i}\approx_{\Lambda} \sum_{i=1}^m\delta_{w_i}=g$. It
follows that $\class{f}\leq \class{g}$ if and only if $\|f\|_1\leq \|g\|_1$. Since
$\|mf\|_1=m\|f\|_1$, we deduce that $m\class{f}\leq m\class{g}\Rightarrow m\|f\|_1\leq
m\|g\|_1\Rightarrow \|f\|_1\leq \|g\|_1\Rightarrow \class{f}\leq \class{g}$.
\end{proof}

\begin{lemma}\label{lem.connected.two}
Let $\Lambda$ be a strongly connected row-finite $k$-graph with no sources. Suppose that there
exist $v\in \Lambda^0$ and $p\in \NN^k$ such that $|v\Lambda^p|>1$. Then $S(\Lambda)$ is
purely infinite and hence unperforated.
\end{lemma}
\begin{proof}
Select $p\in \NN^k$ and $v\in \Lambda^0$ such that $A^t_p\delta_v \neq \delta_w$ for all
$w\in \Lambda^0$. Since $\Lambda$ has no sources, $A^t_p \delta_v\neq 0$. It follows that
$A^t_p\delta_v = \delta_w+\delta_{w'}+a$ for some $w,w'\in\Lambda^0$ and $a\in
\NN\Lambda^0$. Consequently $\class{\delta_w}+\class{\delta_{w'}}\leq \class{\delta_v}$.
By Lemma~\ref{lem.connected}, for each $u\in \Lambda^0$ we have $2\class{\delta_u}\leq
\class{\delta_w}+\class{\delta_{w'}}\leq \class{\delta_v}\leq \class{\delta_u}$. Hence
$2x\leq x$ for every $x\in S(\Lambda)$, so $S(\Lambda)$ is purely infinite.

Fix $m\in \{2,3, \dots\}$ and $x,y\in S(\Lambda)$, such that $mx\leq my$. Since
$S(\Lambda)$ is purely infinite, we get $2y+a=y$ for some $a\in S(\Lambda)$, so
$y=y+(y+a)=y+(y+a)+(y+a)=\dots=my + (m-1)a\geq my$. Thus $x\leq mx\leq my\leq y$.
\end{proof}

We now consider $k$-graphs that are not strongly connected. Let $\Lambda$ be a row-finite
$k$-graph with no sources and let $H$ be a subset of $\Lambda^0$. We write $H^c$ for the
complement of $H$ in $\Lambda^0$. Recall that $H$ is called hereditary if
$s(H\Lambda)\subseteq H$. When $H$ is non-empty and hereditary it follows that
$\Gamma=H\Lambda$ is a row-finite $k$-graph with no sources (\cite[p.~114]{MR1961175}).
Moreover, for each $n\in \NN^k$, $A_n(u,v)=0$ whenever $u\in H$ and $v\in H^c$. This
means that each $A^t_n$ has a block-upper-triangular decomposition with respect to the
decomposition $\Lambda^0=\Gamma^0 \sqcup H^c$. As in \cite{PasSieSim}, we will use the
following notation for this decomposition.
\begin{equation}\label{eq.matrix}
A^t_{n,\Lambda} = \left(
\begin{array}{cc}
A^t_{n,\Gamma} & A^t_{n,H,H^c} \\
0 & A^t_{n,H^c}
\end{array}
 \right).
\end{equation}
We also identify $\NN\Lambda^0$ with $\NN\Gamma^0\oplus \NN H^c$.

\begin{lemma}\label{lem.hereditary}
Let $\Lambda$ be a cofinal row-finite $k$-graph with no sources and $H$ a non-empty
hereditary subset of the vertex set $\Lambda^0$. Then $S(\Lambda)\cong S(H\Lambda)$.
\end{lemma}

\begin{proof}
Set $\Gamma=H\Lambda$. We claim there exists a homomorphism $\Phi\colon S(\Gamma)\to
S(\Lambda)$ such that
$$\Phi(\classg{f})=\class{f\oplus 0}.$$

We must show $\Phi$ is well-defined and additive. Fix $x,y\in \NN\Gamma^0$ such that
$x\sim_\Gamma y$. Then $A^t_{p,\Gamma}x=A^t_{q,\Gamma}y$ for some $p,q\in \NN^k$. By
\eqref{eq.matrix} we obtain $A^t_{p,\Lambda}(x\oplus 0)=A^t_{p,\Lambda}(y\oplus 0)$, so
$x\oplus 0 \sim_\Lambda y\oplus 0$. Now suppose that $f\approx_\Gamma g$. Select
$x_i,y_i\in\NN\Gamma^0$ witnessing this. Then $f\oplus 0 \sim_\Lambda \sum_i(x_i\oplus
0)$, $g\oplus 0 \sim_\Lambda \sum_i(y_i\oplus 0)$ and $x_i\oplus 0 \sim_\Lambda y_i
\oplus 0$. Hence  $f\oplus 0 \approx_{\Lambda} g\oplus 0$ making $\Phi$ well-defined; and $\Phi$
 is additive since $(f+g)\oplus 0=f\oplus 0+g\oplus 0$.

We show that $\Phi$ is surjective. Since $\Lambda$ is cofinal, for each $v\in \Lambda^0$
there exists $n_v\in \NN^k$ such that $s(v\Lambda^{n_v})\subseteq H$. Then
$$A^t_{n_v,\Lambda}\delta_v=g\oplus 0,$$
for some $g\in\NN H$. For each $f\in \NN\Lambda^0$ define $N_f=\vee\{n_v: v\in \supp
(f)\}$, where $\vee$ denotes the coordinatewise maximum. Since $H$ is hereditary $A^t_{N_f,
\Lambda} f=h\oplus 0$ for some $h\in \NN H$. It follows that
$\Phi(\classg{h})=\class{h\oplus 0}=\class{A^t_{N_f,\Lambda}f}=\class{f}$.

We show that $\Phi$ is injective. Fix $f,g\in \NN\Gamma^0$. Suppose that $f\oplus 0
\approx_{\Lambda} g\oplus 0$. By Lemma~\ref{lem.relation.two} there exist finitely many
$f_i\in  \NN \Lambda^0$, $m,m_0\in \NN^k$ and $t_i\in \ZZ^k$ such that $0\leq
m_0,m_0+t_i\leq m$, and
$$A^t_{m,\Lambda}(f\oplus 0)=\sum_iA^t_{m_0,\Lambda}f_i, \quad\text{and}\quad A^t_{m,\Lambda}(g\oplus 0)=\sum_i A^t_{m_0+t_i,\Lambda}f_i.$$
For $N=\vee \{N_{f_i}\}$ we have $A^t_{N,\Lambda}f_i=h_i\oplus 0$ for some $h_i\in \NN
\Gamma^0$. Using \eqref{eq.matrix} we get $(A^t_{m+N,\Gamma}f)\oplus 0 =
A^t_{m+N,\Lambda}(f\oplus 0)=\sum_iA^t_{m_0,\Lambda}(h_i\oplus
0)=(\sum_iA^t_{m_0,\Gamma}h_i)\oplus 0$, and similarly  $A^t_{m+N,\Gamma}g =
\sum_iA^t_{m_0+t_i,\Gamma}h_i$. Hence $f\approx_\Gamma g$.
\end{proof}

\begin{lemma}\label{lem.loop}
Let $\Lambda$ be a cofinal row-finite $k$-graph with no sources. Suppose there exist
$\lambda\in \Lambda$ such that $d(\lambda)\geq (1,\dots,1)$ and
$s(\lambda)=r(\lambda)=w$. Then $\{v\in \Lambda^0\colon w\Lambda v\neq
\emptyset\}\Lambda$ is strongly connected and $S(\Lambda)$ is unperforated.
\end{lemma}
\begin{proof}
Define $H:=\{v\in \Lambda^0: w\Lambda v\neq \emptyset\}$, where $w=r(\lambda)$. It is
clear that $H$ is non-empty hereditary subset of $\Lambda^0$.

We show $\Gamma=H\Lambda$ is strongly connected. Since $\Lambda$ is cofinal so is
$\Gamma$ (because $u\Gamma^n=u\Lambda^n$ for each $u\in \Gamma^0$ and $n\in \NN^k$). Take
any $v\in \Gamma^0$. Select $n\in \NN^k$ such that $s(w\Gamma^n)\subseteq s(v\Gamma)$.
Since $d(\lambda)\geq (1,\dots,1)$, there exists $t\in \NN$ such that $m:=td(\lambda)\geq
n$. We claim that $s(w\Gamma^m)\subseteq s(v\Gamma)$. To see this fix $\nu\in w\Gamma^m$. By the
factorisation property $\nu=\nu_1\nu_2$ for some $\nu_1,\nu_2\in \Gamma$ with
$d(\nu_1)=n$. Consequently $\nu_1\in w\Gamma^n$ so $r(\nu_2)=s(\nu_1)\in
s(w\Gamma^n)\subseteq s(v\Gamma)$. This implies that $r(\nu_2)=s(\mu)$ for some $\mu\in
v\Gamma$. Now $s(\nu)=s(\nu_2)=s(\mu\nu_2)\in s(v\Gamma)$, so $s(w\Gamma^m)\subseteq
s(v\Gamma)$ as claimed.

Using the $t$ selected above, let $\lambda^t:=\overbrace{\lambda\lambda\dots\lambda}^{t \ \textrm{terms}}$. It follows that $\lambda^t\in w\Gamma^m$. Hence
$w=s(\lambda^t)\in s(w\Gamma^m)\subseteq s(v\Gamma)$ so $v\Gamma w\neq \emptyset$. Now
take any $u\in \Gamma^0$. By the definition of $H=\Gamma^0$, $w\Lambda u\neq \emptyset$.
Also, since $w\in H$, we have $w\Gamma u=wH\Lambda u =w\Lambda u$. So $v\Gamma u \neq
\emptyset$.

Since $\Gamma$ is a strongly connected row-finite $k$-graph with no sources,
Lemma~\ref{lem.connected.one} and Lemma~\ref{lem.connected.two} ensure $S(\Gamma)$ is
unperforated. By Lemma~\ref{lem.hereditary}, $S(\Lambda)$ is unperforated.
\end{proof}

\begin{proof}[Proof of Theorem~\ref{thm4}]
Let $\Lambda$ be a row-finite $k$-graph with no sources and coordinate matrices $A_{e_1},
\dots, A_{e_k}$ such that $C^*(\Lambda)$ is simple. It follows from Lemma~\ref{lem.3.5.fix}
or \cite[Theorem~3.1]{MR2323468} that $\Lambda$ is cofinal.

Suppose $\Lambda^0$ contains a non-empty hereditary subset $H$ which is finite. It
follows that $\Gamma:= H\Lambda$ is a cofinal row-finite $k$-graph with no sources (as
before). Take any $v_1\in \Gamma^0$ and set $n=(1,\dots,1)\in \NN^k$. Using that $\Gamma$ has
no sources select $\lambda_1\in v_1\Gamma^n$, set $v_2=s(\lambda_1)\in \Gamma^0$,  select
$\lambda_2\in v_2\Gamma^n$,  set $v_3=s(\lambda_2)\in \Gamma^0$ and so on. Since
$\Gamma^0$ is finite there exists $i\neq j$ with $v_i=v_j$. We obtain $\lambda\in \Gamma$
such that $d(\lambda)\geq n$ and $s(\lambda)=r(\lambda)$. By Lemma~\ref{lem.loop},
$S(\Gamma)$ is unperforated.

Now suppose $\Lambda^0$ contains a non-empty hereditary subset $H$ such that
$\Gamma:=H\Lambda$ is strongly connected. By Lemma~\ref{lem.connected.one} and
Lemma~\ref{lem.connected.two}, $S(\Gamma)$  is unperforated.

In both cases Lemma~\ref{lem.hereditary}
implies that $S(\Lambda)$ is unperforated. Hence the result follows from
Theorem~D\textsuperscript{$\prime$}.
\end{proof}

The following result could be recovered from \cite{MR2920846} (see also
\cite[Proposition~8.8]{MR2270926}), but is also a consequence of our more general
Theorem~\ref{thm4}.

\begin{cor}\label{cor.unital.case}
Let $\Lambda$ be a row-finite $k$-graph with no sources such that $C^*(\Lambda)$ is
unital and simple. Then $C^*(\Lambda)$ is purely infinite if and only if $\big(\sum_{i=1}^k\image(1-A^t_{e_i}) \big) \cap \NN\Lambda^0\neq \{0\}$.
\end{cor}

Using our results we can now establish the following corollary, which in turns motivates
the next section on the semigroup $S(\Lambda)$.
\begin{cor}\label{stably.fin.purely.inf}
Let $\Lambda$ be a row-finite $k$-graph with no sources such that $C^*(\Lambda)$ is
simple. If $S(\Lambda)$ is stably finite then so is $C^*(\Lambda)$. If $S(\Lambda)$ is
purely infinite then so is $C^*(\Lambda)$.
\end{cor}
\begin{proof}
If $S(\Lambda)$ is stably finite, then $S(A,\ZZ^k,\sigma)$ is stably finite by Lemma
\ref{lem.S.isom}. Since $C^*(\Lambda)$ is simple, Lemma~\ref{lem.3.5.fix} in combination
with Lemma~\ref{lem.minimal} implies that $\sigma$ is minimal. Proposition
\ref{prop.main.2} now gives that $A\rtimes_{\sigma,r} \ZZ^k$ is stably finite. Hence also
$C^*(\Lambda)$ is stably finite (see the last paragraph of the proof of
Theorem~\ref{thm3}).

If $S(\Lambda)$ purely infinite, then it is also almost unperforated, see Lemma
\ref{lem.connected.two}. By Theorem~\ref{thm3} we get that $C^*(\Lambda)$ is purely
infinite or stably finite. We suppose that $C^*(\Lambda)$ is stably finite and derive a
contradiction. Since $C^*(\Lambda)$ is stably finite so is $A\rtimes_{\sigma,r} \ZZ^k$, and by
simplicity of $C^*(\Lambda)$ the action $\sigma$ is minimal. Proposition~\ref{prop.main.2}
now gives $S(\Lambda)\cong S(A,\ZZ^k,\sigma)$ is stably finite, so
$S(\Lambda)$ is not purely infinite, a contradiction. It follows that $C^*(\Lambda)$ is
purely infinite.
\end{proof}

\section{The semigroup \texorpdfstring{$S(\Lambda)$}{S(Lambda)}}

Motivated by Corollary \ref{stably.fin.purely.inf} we now present a number of sufficient
conditions ensuring that the semigroup $S(\Lambda)$ associated to a $k$-graph $\Lambda$
is stably finite or purely infinite. First we recall some terminology.

Let $S$ be an (abelian) semigroup with identity $0$. Recall that $S$ is conical if $x + y
= 0$ implies $x = y = 0$ for any $x,y\in S$. Also, recall from Section~\ref{section.two} that $S$ satisfies the Riesz
refinement property if whenever $a,b,c,d \in S$ satisfy $a + b = c + d$ then there exist
$x, y, z,t\in S$ such that $a = x + y$, $b = z + t$, $c = x + z$ and $d = y + t$. If $S$ is preordered, it is said to be simple if for any nonzero $x,y$ there
exists $n\in\NN$ such that $x\leq ny$. A nonzero element $x\in S$ is infinite if $x+y=x$
for some nonzero $y\in S$ and it is properly infinite if $2x\leq x$. The semigroup $S$ is
stably finite if no nonzero element of $S$ is infinite, and purely infinite if every
nonzero element of $S$ is properly infinite.

\begin{lemma}\label{lemma.properties.conical}
Let $\Lambda$ be a row-finite $k$-graph with no sources.  Then
\begin{enumerate}
\item\label{lemma.properties.conical.i} $S(\Lambda)$ with identity $\class{0}$ is conical.
\item\label{lemma.properties.conical.ii} $S(\Lambda)$ satisfies the Riesz refinement property.
\item\label{lemma.properties.conical.iii} If $\Lambda$ is cofinal, then $S(\Lambda)$ is simple and every infinite
    element of $S(\Lambda)$ is properly infinite.
\end{enumerate}
\end{lemma}
\begin{proof}
For \eqref{lemma.properties.conical.i}, suppose that $f,g\in \NN\Lambda^0$ satisfy $\class{f}+\class{g}=0$. Then
$f+g\sim_\Lambda \sum_ix_i$ and $0\sim_\Lambda \sum_iy_i$ for some $x_i, y_i\in
\NN\Lambda^0$ such that $x_i\sim_\Lambda y_i$ for each $i$. Hence $0=A^t_p0=A^t_q\sum_iy_i$
for some $p,q\in \NN^k$. Since $\Lambda$ has no sources, each $A^t_n$ is a non-negative
integer matrix with no zero columns, giving $\ker A^t_n\cap \NN \Lambda^0 = \{0\}$. In
particular $\sum_iy_i=0$, and similarly $\sum_ix_i=0$, giving $f+g=0$. Consequently, $f=g=0$ and
$\class{f}=\class{g}=0$.

For \eqref{lemma.properties.conical.ii}, take any $a, b, c, d\in \NN\Lambda^0$ such that
$\class{a}+\class{b}=\class{c}+\class{d}$. By Lemma~\ref{lem.relation.two} there exist
finitely many $f_i\in  \NN \Lambda^0$, $m,m_0\in \NN^k$ and $s_i\in \ZZ^k$ such that
$0\leq m_0,m_0+s_i\leq m$, and
$$A^t_m(a+b)=\sum_{i=1}^nA^t_{m_0}f_i, \quad\textrm{ and } \quad A^t_m(c+d)=\sum_{i=1}^n A^t_{m_0+s_i}f_i.$$
Recall that $\NN\Lambda^0$ satisfies the Riesz refinement property (see proof of
Lemma~\ref{lem.relation.two}). So there is a refinement matrix (see page \pageref{ref.matrix}) implementing the first
equality, i.e.,
$$\bordermatrix{
 &A^t_m a&A^t_m b\cr  A^t_{m_0} f_1&a_1&b_1\cr  A^t_{m_0} f_2&a_2&b_2\cr \ \ \ \ \vdots & & \cr  A^t_{m_0} f_n&a_n&b_n\cr}.$$
 We have $\sum_{i=1}^n A^t_{m+s_i}(a_i+b_i)=\sum_{i=1}^n A^t_{m+s_i}A^t_{m_0}f_i=A^t_m\sum_{i=1}^n A^t_{m_0+s_i}f_i=A^t_{2m}(c+d)$. Choose a refinement matrix implementing this equality, i.e.,
 $$\bordermatrix{
 &A^t_{2m} c&A^t_{2m} d\cr  A^t_{m+s_1} a_1&x_1&y_1\cr  A^t_{m+s_2} a_2&x_2&y_2\cr \ \ \ \ \ \vdots & &\cr  A^t_{m+s_n} a_n&x_n&y_n\cr  A^t_{m+s_1} b_1&z_1&t_1\cr  A^t_{m+s_2} b_2&z_2&t_2\cr  \ \ \ \ \ \vdots & & \cr  A^t_{m+s_n} b_n&z_n&t_n\cr}.$$
 Define $x=\sum_j x_j$, $y=\sum_j y_j$, $z=\sum_j z_j$ and $t=\sum_j t_j$. From the equalities
 \begin{align*}
 x+y & =\sum_j(x_j+y_j)=\sum_j (A^t_{m+s_j} a_j) \approx_{\Lambda} \sum_j a_j = A^t_m a,\\
 z+t & =\sum_j(z_j+t_j)=\sum_j (A^t_{m+s_j} b_j) \approx_{\Lambda} \sum_j b_j = A^t_m b,\\
 x+z & =\sum_j(x_j+z_j) = A^t_{2m} c, \\
 y+t & =\sum_j(y_j+t_j) = A^t_{2m} d,
 \end{align*}
 we obtain the refinement matrix
$$\bordermatrix{
 &\class{c}&\class{d}\cr  \class{a}&\class{x}&\class{y}\cr  \class{b}&\class{z}&\class{t}\cr},$$
so $S(\Lambda)$ satisfies the Riesz refinement property.

Finally, for \eqref{lemma.properties.conical.iii}, fix any nonzero $x,y\in S(\Lambda)$. We show that $x\leq ly$ for some
$l\in \NN$ (the second part of \eqref{lemma.properties.conical.iii} then follows from the proof of
Proposition~\ref{prop.main.2} \eqref{prop.main.2.(4)}$\Rightarrow$\eqref{prop.main.2.(5)}). Take nonzero $f,g\in \NN\Lambda^0$
such that $x=\class{f}$ and $y=\class{g}$. Fix $w\in \supp(f)$ and $v\in \supp(g)$.
Since $\Lambda$ is cofinal there exists $n\in \NN^k$ such that $s(w\Lambda^n)\subseteq s(v\Lambda)$. Write $A^t_n\delta_w
= \sum_{i=1}^m \delta_{w_i}$ (possibly with $w_i=w_j$ for $i\neq j$ if several
$\lambda\in w\Lambda^n$ have the same source). Since $s(w\Lambda^n)=\{w_i\}\subseteq
s(v\Lambda)$, there exists $\lambda_i\in v\Lambda$ such that
$s(\lambda_i)=w_i$ ($i=1,\dots,m$). Let $p_i:=d(\lambda_i)$ for each $i$. For each $i$ we have
$A^t_{p_i}\delta_v=\delta_{w_i} +a_i$ for some $a_i\in \NN\Lambda^0$. With
$a=\sum_{i=1}^m a_i$ we have $\sum_{i=1}^mA^t_{p_i}\delta_v=\sum_{i=1}^m\delta_{w_i} + a
= A^t_n\delta_w + a$, giving that $\class{\delta_w}+\class{a}=\class{A^t_n\delta_w
+a}=\sum_{i=1}^m \class{A^t_{p_i}\delta_v}=m\class{\delta_v}$. Applying this procedure to
each $w\in \supp(f)$ and adding up we get $x=\class{f}\leq l\class{\delta_v}\leq ly$ for
sufficiently large $l\in \NN$.
\end{proof}

 \begin{remark} Properties \eqref{lemma.properties.conical.i} and \eqref{lemma.properties.conical.iii} of Lemma~\ref{lemma.properties.conical} can both be recovered immediately from our previous $C^*$-algebraic results. Property \eqref{lemma.properties.conical.i} follows from Lemma~\ref{lem.A.conical}\eqref{lem.A.conical.i} and Lemma~\ref{lem.S.isom}. Property \eqref{lemma.properties.conical.iii} has two parts. The first part follows from applying Lemma~\ref{lem.minimal}, Lemma~\ref{lem.A.conical}\eqref{lem.A.conical.ii} and Lemma~\ref{lem.S.isom}, and the second part is algebraic. However, it seemed worthwhile to provide a direct algebraic proof rather than relying on $C^*$-algebraic results.

 We do not know for which $C^*$-dynamical systems (beyond the ones coming from a $k$-graph) the associated semigroup automatically has the Riesz refinement property.
 \end{remark}

\begin{prop}\label{prop.properties.dichotomy.merged}
Let $\Lambda$ be a row-finite $k$-graph with no sources. Suppose that $\Lambda$ is
cofinal and contains a cycle $\lambda$ satisfying $d(\lambda)\geq (1,\dots,1)$. Then
either $S(\Lambda)$ is stably finite and totally ordered, or $S(\Lambda)$ is  purely
infinite and $x\leq y$ for any nonzero $x,y\in S(\Lambda)$.
\end{prop}
\begin{proof}
Choose any $\lambda\in \Lambda$ such that $d(\lambda)\geq (1,\dots,1)$ and
$s(\lambda)=r(\lambda)$ and define $H:=\{v\in \Lambda^0: r(\lambda)\Lambda v\neq
\emptyset\}$. Lemma~\ref{lem.hereditary} gives $S(\Lambda)\cong S(H\Lambda)$.
Lemma~\ref{lem.loop} implies that  $\Gamma:=H\Lambda$ is strongly connected.

We consider two cases. First suppose that $|v\Gamma^p|=1$ for all $v\in \Gamma^0$ and
$p\in \NN^k$. This implies that $\class{f}=\class{g}$ if and only if $\|f\|_1=\|g\|_1$ as in
Lemma~\ref{lem.connected.one}. So $S(\Gamma)\cong \NN$ is stably finite and totally
ordered. Now suppose that $|v\Gamma^p|>1$ for some $v\in \Gamma^0$ and $p\in
\NN^k$. Then $S(\Gamma)$ is purely infinite by Lemma~\ref{lem.connected.two}.
Lemma~\ref{lemma.properties.conical} shows that $S(\Lambda)$, and hence also $S(\Gamma)$,
is simple. Consequently, for any nonzero $x,y\in S(\Gamma)$ there exists $n\in \NN$ such
that $x\leq ny$. Since $S(\Gamma)$ is purely infinite, $x\leq  ny\leq y$. As
$S(\Gamma)\cong S(\Lambda)$ this completes the proof.
\end{proof}

\begin{remark}
If $\Lambda$ is a strongly connected row-finite $k$-graph (with no sources), then
$\Lambda$ is cofinal (\cite[Proposition~4.5]{MR3444442}). It also contains a cycle
$\lambda$ satisfying $d(\lambda)\geq (1,\dots,1)$: since $\Lambda$ has no sources, we can
choose $\mu \in \Lambda^{(1, \dots, 1)}$, and then since $\Lambda$ is strongly connected,
there exists $\nu \in s(\mu)\Lambda r(\mu)$; now $\lambda = \mu\nu$ is a cycle with
$d(\lambda) \ge (1, \dots, 1)$. Therefore, Proposition
\ref{prop.properties.dichotomy.merged} applies to such $k$-graphs $\Lambda$.
\end{remark}

In the next two results we describe another class of examples where our results provide a
complete dichotomy. To see the connection between Propositions
\ref{prop.properties.dichotomy.three}~and~\ref{prop.properties.dichotomy.three.part.two},
observe that the $k$-graphs of Proposition~\ref{prop.properties.dichotomy.three} have the
same property as those in Proposition~\ref{prop.properties.dichotomy.three.part.two} with
the exception that the $n_{w,i,j}$ are all $1$. Figure~\ref{fig1} indicates what the
skeleton of a $2$-graph satisfying
Proposition~\ref{prop.properties.dichotomy.three.part.two} might look like (in that
example, $n_{w,1,2} = m_{w,1,2} = 2$ for all $v,w$).

\begin{prop}\label{prop.properties.dichotomy.three}
Let $\Lambda$ be a row-finite and cofinal $k$-graph with no sources, and suppose that
$\ell \colon \Lambda^0\to \ZZ$ satisfies $\ell(s(\lambda))=1+\ell(r(\lambda))$ for all
$\lambda\in\Lambda^{e_i}$. Suppose that for each $w,v\in \Lambda^0$, and $i\neq j$ we
have $|w\Lambda^{e_i}v|=|w\Lambda^{e_j}v|$. Then $S(\Lambda)$ is stably finite.
\end{prop}
\begin{proof}
We suppose that $S(\Lambda)$ contains a nonzero infinite element and derive a
contradiction. Take any vertex $w\in \Lambda^0$. Upon replacing $\ell$ by $\ell-\ell(w)$,
we may assume $\ell(w)=0$. Define $H:=\{v\in \Lambda^0: w\Lambda v\neq \emptyset\}$,
$\Gamma:=H\Lambda$, and $x:=\classg{\delta_w}$. For each $i$, let $H_i := H \cap
\ell^{-1}(i)$, so that $H_0 = \{w\}$. We also write $H_{\ge i} := \bigcup_{j \ge i} H_j$.

Since $S(\Lambda)\cong S(\Gamma)$ we can select a nonzero infinite element $y\in
S(\Gamma)$. Lemma~\ref{lemma.properties.conical}\eqref{lemma.properties.conical.iii} says that $S(\Lambda)$ is simple,
and so there exists $n > 0$ such that $y \le nx$. Another application of simplicity
yields $m$ such that $2nx \le my$. The second part of
Lemma~\ref{lemma.properties.conical}\eqref{lemma.properties.conical.iii} says that $y$ is properly infinite, and so $y
\ge my$. Hence $2nx \le my \le y \le nx$, and we deduce that $nx$ is properly infinite
(see also \cite[Lemma~2.9]{MR2806681}). Hence there exist $z_0, f_i, \dots, f_l \in \NN
\Gamma^0$, a pair $m,m_0\in \NN^k$, and elements $t_1, \dots, t_l\in \ZZ^k$ such that
$0\leq m_0, m_0+t_i\leq m$, and
\begin{equation}\label{eq:matchups}
    A^t_m (n \delta_w)=\sum_iA^t_{m_0}f_i \quad\text{ and }\quad A^t_m(2n \delta_w + z_0) = \sum_i A^t_{m_0+t_i}f_i.
\end{equation}
Let
\begin{equation}\label{eq:Ndef}
    N := 2n + (z_0)_w > n
\end{equation}
and let $z := z_0 - (z_0)_w \delta_w$. Then
\begin{equation}\label{eq:matchup}
    A^t_m(N \delta_w + z) = \sum_i A^t_{m_0+t_i}f_i,
\end{equation}
and $w \not\in \supp(z)$. Since $H_0 = \{w\}$, it follows that $\supp(z) \subseteq H_{\ge
1}$.

For each $i\in \{1,\dots, k\}$ and $v\in H_j$, we have $\supp(A^t_{e_i} \delta_v) \subseteq
H_{j+1}$. Writing $|a| := \sum_{i=1}^k a_i$ for $a \in \NN^k$, we see by induction that
\[
\supp(A^t_m \delta_w) \subseteq H_{|m|}\quad\text{ and }\quad
    \supp(A^t_m z) \subseteq H_{\ge (|m|+1)}.
\]
Now~\eqref{eq:matchup} shows that each $\supp(A^t_{m_0+t_i}f_i) \subseteq H_{\ge |m|}$.

Since $A^t_m (n \delta_w)=\sum_iA^t_{m_0}f_i$, we also have $\supp(A^t_{m_0} f_i)
\subseteq \supp(A^t_m(n\delta_w)) \subseteq H_{|m|}$. Since $m_0$, $m - m_0$ and $m$ all
belong to $\NN^k$ we have
$$|m-m_0| + |m_0| = |m|.$$
So each $\supp(f_i) \subseteq H_{|m - m_0|}$. As $\supp(A^t_{m_0+t_i}f_i) \subseteq
H_{\ge |m|}$, we get $|m-m_0|+|m_0+t_i|\geq |m|$.

Putting all of this together, we deduce that $|m_0 + t_i| \ge |m_0|$ for all $i$, and
therefore each $|t_i| \ge 0$. We have
\[
\supp(A^t_m (N\delta_{w})) \subseteq H_{|m|}\quad\text{ and }\quad
    \supp(A^t_m z) \subseteq H_{\ge (|m|+1)},
\]
and likewise
\[
\supp\Big(\sum_{\{i : |t_i| = 0\}} A^t_{m_0+t_i}f_i\Big) \subseteq H_{|m|}
    \quad\text{ and }\quad \supp\Big(\sum_{\{i: |t_i| > 0\}} A^t_{m_0+t_i}f_i\Big) \subseteq H_{\ge |m|+1}.
\]
Using~\eqref{eq:matchup} we must have
\[
A^t_m (N\delta_{w}) = \sum_{\{i: |t_i| = 0\}} A^t_{m_0+t_i}f_i, \quad\text{ and }\quad
    A^t_m z = \sum_{\{i: |t_i| > 0\}} A^t_{m_0+t_i}f_i.
\]
By hypothesis, we have $|u \Gamma^{e_i} v| = |u \Gamma^{e_j} v|$ for all $i,j \le k$ and
$u,v \in \Lambda^0$. So an induction argument gives $|u \Gamma^p v| = |u \Gamma^q v|$ for all $p,q
\in \NN^k$ with $|p| = |q|$. In particular, if $|t_i| = 0$, then $|m_0 + t_i| = |m_0|$,
and so $A_{m_0 + t_i} = A_{m_0}$. Hence
\[
A^t_m (N\delta_{w}) = \sum_{\{i: |t_i| = 0\}} A^t_{m_0+t_i}f_i \,\leq\, \sum_i A^t_{m_0}f_i.
\]
But the first part of~\eqref{eq:matchups} then gives
\[
N A^t_m (\delta_{w}) = A^t_m(N \delta_w) \leq  A^t_m(n\delta_{w}) = n A^t_m(\delta_w),
\]
and since $A^t_m \delta_w \not= 0$, this forces $N \leq n$,
contradicting~\eqref{eq:Ndef}.
\end{proof}

\begin{prop}\label{prop.properties.dichotomy.three.part.two}
Let $\Lambda$ be a row-finite and cofinal $k$-graph with no sources, and suppose that
$\ell \colon \Lambda^0\to \ZZ$ satisfies $\ell(s(\lambda))=1+\ell(r(\lambda))$ for all
$\lambda\in\Lambda^{e_i}$. Suppose that for each $w \in \Lambda^0$ and each pair of
distinct indices $i, j \in \{1, \dots, k\}$, there are strictly positive rational numbers
$n_{w,i,j}, m_{w,i,j}$ such that
\begin{equation}\label{eqn.indepndent}
|w\Lambda^{e_i}v| = n_{w,i,j}|w\Lambda^{e_j}v|\text{ for all $v \in \Lambda^0$}\ \text{and}\ \ |u\Lambda^{e_i}w| = m_{w,i,j}|u\Lambda^{e_j}w|\text{ for all $u \in \Lambda^0$.}
\end{equation}
Suppose that there exist $w,i,j$ such that $n_{w,i,j} \not= 1$. Then $S(\Lambda)$ is
purely infinite.
\end{prop}
\begin{proof}
Select $w\in \Lambda^0$ and $i\neq j\in \{1,\dots, k\}$ such that $n_{w,i,j}\neq 1$; by
interchanging the roles of $i,j$ if necessary, we can assume that $n_{w,i,j} > 1$, and so
$|w\Lambda^{e_i} v| > |w\Lambda^{e_j} v|>0$ for each $v\in s(w\Lambda^{e_j})$.

Define $H:=\{v\in \Lambda^0: w\Lambda v\neq \emptyset\}$, and let $\Gamma:=H\Lambda$.
Again for each $i \ge 0$ let $H_i = H \cap \ell^{-1}(i)$. Equation~\eqref{eqn.indepndent}
implies that for each $i,j \le k$ and each $\alpha \in \Gamma^{e_i}$ and $\beta \in
s(\alpha)\Gamma^{e_j}$, there exist $\beta' \in \Lambda^{e_j}$ and $\alpha' \in
\Lambda^{e_i}$ such that $r(\beta') = r(\alpha)$, $s(\beta') = s(\alpha) = r(\beta) =
r(\alpha')$, and $s(\alpha') = s(\beta)$; that is, for each $c_ic_j$-coloured path in
$\Gamma$, there is a $c_jc_i$-coloured path passing through the same vertices.
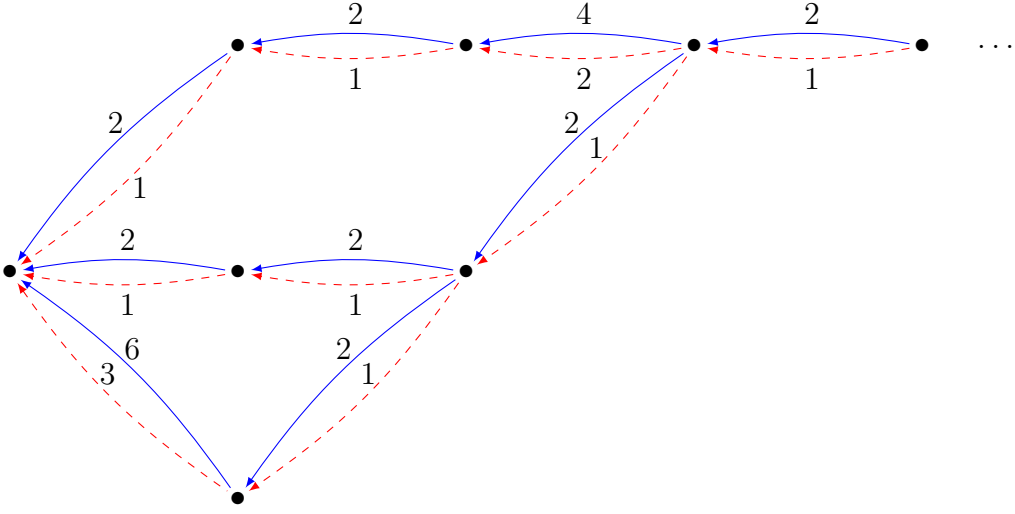
\begin{figure}
\begin{center}
\begin{tikzpicture}
    \node[circle, inner sep=0.5pt] (w) at (0,0) {$\bullet$};
    \node[circle, inner sep=0.5pt] (v0) at (3,3) {$\bullet$};
    \node[circle, inner sep=0.5pt] (v1) at (6,3) {$\bullet$};
    \node[circle, inner sep=0.5pt] (v2) at (9,3) {$\bullet$};
    \node[circle, inner sep=0.5pt] (v3) at (12,3) {$\bullet$};
    \node[circle, inner sep=0.5pt] (w0) at (3,0) {$\bullet$};
    \node[circle, inner sep=0.5pt] (w1) at (6,0) {$\bullet$};
    \node[circle, inner sep=0.5pt] (x0) at (3,-3) {$\bullet$};
    \node[circle, inner sep=0pt] at (13,3) {$\dots$};
    \draw[->, blue, out=215, in=55,  >=latex] (v0) to node[black, above, pos=0.48, inner sep=0.3em] {$2$} (w);
    \draw[->, red, dashed, out=235, in=35,  >=latex] (v0) to node[black, below, pos=0.48, inner sep=0.3em] {$1$} (w);
    \draw[->, blue,  out=170, in=10,  >=latex] (v1) to node[black, above, pos=0.48, inner sep=0.3em] {$2$} (v0);
    \draw[->, blue,  out=170, in=10,  >=latex] (v2) to node[black, above, pos=0.48, inner sep=0.3em] {$4$} (v1);
    \draw[->, blue,  out=170, in=10,  >=latex] (v3) to node[black, above, pos=0.48, inner sep=0.3em] {$2$} (v2);
    \draw[->, red, dashed,  out=190, in=350,  >=latex] (v1) to node[black, below, pos=0.48, inner sep=0.3em] {$1$} (v0);
    \draw[->, red, dashed,  out=190, in=350,  >=latex] (v2) to node[black, below, pos=0.48, inner sep=0.3em] {$2$} (v1);
    \draw[->, red, dashed,  out=190, in=350,  >=latex] (v3) to node[black, below, pos=0.48, inner sep=0.3em] {$1$} (v2);
    \draw[->, blue,  out=170, in=10,  >=latex] (w0) to node[black, above, pos=0.48, inner sep=0.3em] {$2$} (w);
    \draw[->, blue,  out=170, in=10,  >=latex] (w1) to node[black, above, pos=0.48, inner sep=0.3em] {$2$} (w0);
    \draw[->, red, dashed,  out=190, in=350,  >=latex] (w0) to node[black, below, pos=0.48, inner sep=0.3em] {$1$} (w);
    \draw[->, red, dashed,  out=190, in=350,  >=latex] (w1) to node[black, below, pos=0.48, inner sep=0.3em] {$1$} (w0);
    \draw[->, blue,  out=215, in=55,  >=latex] (v2) to node[black, above, pos=0.48, inner sep=0.3em] {$2$} (w1);
    \draw[->, red, dashed,  out=235, in=35,  >=latex] (v2) to node[black, above, pos=0.48, inner sep=0.3em] {$1$} (w1);
    \draw[->, blue,  out=215, in=55,  >=latex] (w1) to node[black, above, pos=0.48, inner sep=0.3em] {$2$} (x0);
    \draw[->, red, dashed,  out=235, in=35,  >=latex] (w1) to node[black, above, pos=0.48, inner sep=0.3em] {$1$} (x0);
    \draw[<-, blue,  out={235+90}, in={35+90},  >=latex] (w) to node[black, above, pos=0.48, inner sep=0.3em] {$6$} (x0);
    \draw[<-, red, dashed,  out={215+90}, in={55+90},  >=latex] (w) to node[black, above, pos=0.48, inner sep=0.3em] {$3$} (x0);
\end{tikzpicture}
\caption{The skeleton of a $2$-graph satisfying the hypotheses of Proposition~\ref{prop.properties.dichotomy.three.part.two}.}
\label{fig1}
\end{center}
\end{figure}

Fix $l \in \NN$, $u\in H_l$ and $v\in H_{l+2}$. Let $Y := \{y \in H_{l+1} : u
\Lambda^{e_i} y \Lambda^{e_j} v \not= \emptyset\}$. For each $y \in Y$, let $a_y :=
|u\Gamma^{e_i}y|$, $b_y := |u\Gamma^{e_j} y|$, $c_y := |y\Gamma^{e_i} v|$ and $d_y := |y
\Gamma^{e_j} v|$. By assumption there exist $N,M,N',M'\in \NN$ such that $a_y =
\frac{N}{M} b_y$ and $c_y = \frac{N'}{M'}d_y$ for each $y \in Y$. By the factorisation
property $A^t_{e_i}A^t_{e_j} = A^t_{e_j}A^t_{e_i}$ and so $\sum_{y \in Y} a_yd_y= \sum_{y \in Y} b_yc_y$.
Hence
\[
MN' = MN'\frac{\sum_y a_yd_y}{\sum_y a_yd_y}
    = NM'\frac{\sum_y b_yc_y}{\sum_y a_yd_y}
    = NM'\frac{\sum_y a_yd_y}{\sum_y a_yd_y}
    = NM',
\]
giving $\frac{N}{M}=\frac{N'}{M'}$.

By choice of $w$, we have $|w\Lambda^{e_i} v| > |w\Lambda^{e_j} v|$ for each $v\in
s(w\Lambda^{e_j})$. So applying the preceding paragraph with $u: = w$, we see that
$|v\Lambda^{e_i} u|> |v\Lambda^{e_j} u|$ for each $v\in s(w\Lambda^{e_j})$ and each $u\in
s(v\Lambda^{e_j})$. Now an induction argument shows that $|u\Lambda^{e_i} v| > |u\Lambda^{e_j} v|$
for any $u\in \Gamma^0$ and any $v\in s(u\Lambda^{e_j})$. That is, $n_{u,i,j} > 1$ for
all $u \in \Gamma^0$.

Now fix $u \in H$, say $\ell(u) = l$, and observe that $A^t_{e_j}\delta_u = \sum_{v \in
H_{l+1}} |u\Lambda^{e_j} v|\delta_{v}$. So
\begin{align*}
A^t_{e_i}\delta_u
    = \sum_{v \in H_{l+1}} |u\Lambda^{e_i} v|\delta_{v}
    = \sum_{v \in H_{l+1}} (n_{u, i, j} - 1) |u\Lambda^{e_j} v|\delta_{v} + A^t_{e_j}\delta_u.
\end{align*}
Since $\Lambda$ has no sources, and since $n_{u,i,j} > 1$, we have $\sum_{v \in H_{l+1}}
(n_{u, i, j} - 1) |u\Lambda^{e_j} v|\delta_{v} \not= 0$. Since $[A^t_{e_i} \delta_u] =
[\delta_u] = [A^t_{e_j} \delta_u]$ in $S(\Gamma)$, it follows that $\classg{\delta_u}$ is
infinite in $S(\Gamma)$. Since $\Gamma$ is cofinal,
Lemma~\ref{lemma.properties.conical}\eqref{lemma.properties.conical.iii} ensures that $\classg{\delta_u}$ is properly infinite. Hence
$S(\Gamma)$ is purely infinite, and so Lemma~\ref{lem.hereditary} shows that $S(\Lambda)$
is purely infinite.
\end{proof}

\begin{example}\label{easy.bridge}
\label{three.colour.bridge} Consider any 2-graph $\Lambda$ with skeleton equal to the $2$-coloured graph
illustrated on Figure \ref{eq:generic2graph}.
\begin{figure}
\begin{center}
\begin{tikzpicture}
    \node[circle, inner sep=0pt] (v1) at (0,0) {$v_1$};
    \node[circle, inner sep=0pt] (v2) at (3,0) {$v_2$};
    \node[circle, inner sep=0pt] (v3) at (6,0) {$v_3$};
    \node[circle, inner sep=0pt] (v4) at (9,0) {$v_4$};
    \node[circle, inner sep=0pt] at (10.5,0) {$\dots$};
    \draw[->, blue,  out=170, in=10,  >=latex] (v2) to node[black, above, pos=0.48, inner sep=0.3em] {$\vdots$} (v1);
    \draw[->, blue,  out=135, in=45,  >=latex] (v2) to (v1);
     \draw[->, blue,  out=170, in=10,  >=latex] (v3) to node[black, above, pos=0.48, inner sep=0.3em] {$\vdots$} (v2);
    \draw[->, blue,  out=135, in=45,  >=latex] (v3) to (v2);
    \draw[->, blue,  out=170, in=10,  >=latex] (v4) to node[black, above, pos=0.48, inner sep=0.3em] {$\vdots$} (v3);
    \draw[->, blue,  out=135, in=45,  >=latex] (v4) to (v3);
    \draw[->, red, dashed,  out=190, in=350,  >=latex] (v2) to (v1);
    \draw[->, red, dashed,  out=225, in=315,  >=latex] (v2) to node[black, above, pos=0.48, inner sep=0.3em] {$\vdots$} (v1);
    \draw[->, red, dashed,  out=190, in=350,  >=latex] (v3) to (v2);
    \draw[->, red, dashed,  out=225, in=315,  >=latex] (v3) to node[black, above, pos=0.48, inner sep=0.3em] {$\vdots$} (v2);
    \draw[->, red, dashed,  out=190, in=350,  >=latex] (v4) to (v3);
    \draw[->, red, dashed,  out=225, in=315,  >=latex] (v4) to node[black, above, pos=0.48, inner sep=0.3em] {$\vdots$} (v3);
\end{tikzpicture}
\caption{}
\label{eq:generic2graph}
\end{center}
\end{figure}
The numbers  $|v_n\Lambda^{e_1} v_{n+1}|$ and $|v_m\Lambda^{e_1} v_{m+1}|$ of blue edges
connecting distinct pairs of consecutive vertices are not assumed to be equal, and
likewise for red edges. It was shown in \cite{MR3507995} that $C^*(\Lambda)$ is stably
finite precisely when $|v_1\Lambda^{e_1}|=|v_1\Lambda^{e_2}|$, and is purely infinite
otherwise. This aligns with Proposition~\ref{prop.properties.dichotomy.three} and
Proposition~\ref{prop.properties.dichotomy.three.part.two}, which give that $S(\Lambda)$
is either stably finite or purely infinite, and with Corollary \ref{stably.fin.purely.inf}
which shows that $C^*(\Lambda)$ is either stably finite or purely infinite if we also know that
$\Lambda$ is aperiodic for $|v_1\Lambda^{e_1}|\neq |v_1\Lambda^{e_2}|$ (and we do know
that by \cite{MR3507995}).
\end{example}

Recall that a nontrivial state on a preordered (abelian) semigroup $S$ with identity is a
function $\beta \colon S\to [0,\infty]$ which respects $+$ and $\leq$, satisfies
$\beta(0)=0$ and takes a value different from $0$ and $\infty$. As in Section~\ref{section.ThmE} a faithful graph trace on
a row-finite and $k$-graph $\Lambda$ with no sources is a function $\tau \colon
\Lambda^0\to \RR^+\setminus\{0\}$ such that $\tau(v)=\sum_{\lambda\in v\Lambda^n}
\tau(s(\lambda))$ for all $v\in \Lambda^0$ and all $n\in \NN^k$.

\begin{prop}\label{prop.properties.dichotomy.four}
Let $\Lambda$ be a row-finite and cofinal $k$-graph with no sources, and coordinate
matrices $A_{e_1}, \dots , A_{e_k}$. Then the following are equivalent:
\begin{enumerate}
\item\label{prop.properties.dichotomy.four.(1)} $S(\Lambda)$ is stably finite;
\item\label{prop.properties.dichotomy.four.(2)} $(n+1)x\not\leq nx$ for any nonzero $x\in S(\Lambda)$ and $n\in \NN$;
\item\label{prop.properties.dichotomy.four.(3)} $S(\Lambda)$ admits a nontrivial state;
\item\label{prop.properties.dichotomy.four.(4)} $\big(\sum_{i=1}^k\image(1-A^t_{e_i}) \big) \cap \NN\Lambda^0= \{0\}$; and
\item\label{prop.properties.dichotomy.four.(5)} $\Lambda$ admits a faithful graph trace.
\end{enumerate}
\end{prop}

\begin{proof}
The implications
\eqref{prop.properties.dichotomy.four.(1)}$\Rightarrow$\eqref{prop.properties.dichotomy.four.(2)}$\Rightarrow$\eqref{prop.properties.dichotomy.four.(3)}
follow from
Proposition~\ref{prop.main.2}\eqref{prop.main.2.(5)}--\eqref{prop.main.2.(7)}. For
\eqref{prop.properties.dichotomy.four.(3)}$\Rightarrow$\eqref{prop.properties.dichotomy.four.(1)},
let $\beta$ be a nontrivial state on $S(\Lambda)$. By rescaling $\beta$ we can assume
there exists $y\in S(\Lambda)$ such that $\beta(y)=1$. Suppose that $S(\Lambda)$ is not
stably infinite and fix a nonzero infinite element $x\in S(\Lambda)$. Since $\Lambda$ is
cofinal, $S(\Lambda)$ is simple. Hence $x\leq n y$ and $y\leq m x$ for some $n,m\in \NN$.
Consequently, $\beta(x)\leq \beta(ny)=n$, so $\beta(x)<\infty$, and $0<\beta(y)\leq
\beta(mx)=m\beta(x)$, so $0<\beta(x)$. Set $\lambda:=\beta(x)\in(0,\infty)$. Since $x$ is
infinite and $S(\Lambda)$ is simple, $x$ is properly infinite, so $2x\leq x$. Hence
$2\lambda=\beta(2x)\leq \beta(x)=\lambda$, which is a contradiction. Hence $S(\Lambda)$
is stably finite. To verify
\eqref{prop.properties.dichotomy.four.(1)}$\Rightarrow$\eqref{prop.properties.dichotomy.four.(4)},
suppose that $(\sum_{i=1}^k\image(1-A^t_{e_i})) \cap \NN\Lambda^0\neq \{0\}$, say $f,
f_1,\dots, f_k\in \NN\Lambda^0$ satisfy $f\neq 0$ and $\sum_{i=1}^k (f_i-A^t_{e_i}f_i) =
f $. Since $f\neq 0$, we have $f_i\neq 0$ for some $i$. This implies that
$\class{\sum_if_i}=\class{\sum_if_i}+\class{f}$, so $S(\Lambda)$ is not stably finite.
The implication
\eqref{prop.properties.dichotomy.four.(4)}$\Rightarrow$\eqref{prop.properties.dichotomy.four.(5)}
is in \cite[Proposition~3.6]{MR3507995}.

Finally, to show that
\eqref{prop.properties.dichotomy.four.(5)}$\Rightarrow$\eqref{prop.properties.dichotomy.four.(3)},
suppose that $\tau$ is a faithful graph trace on $\Lambda$. We claim that there is a
function $\beta\colon S(\Lambda)\to [0,\infty]$ such that
$\beta(\class{\sum_{i=1}^n\delta_{v_i}})=\sum_{i=1}^n \tau(v_i),$ for all finite
sequences $(v_i)_{i=1}^n$ in $\Lambda^0$. Indeed take any $x,y\in \NN\Lambda^0$ such that
$x\sim_\Lambda y$. Then $A^t_px=A^t_qy$ for some $p,q\in \NN^k$. Write
$x=\sum_{i=1}^n\delta_{v_i}$ and $y=\sum_{i=1}^m\delta_{w_i}$ for suitable $v_i,w_i\in
\Lambda^0$ and $n,m\in \NN$. By the definition of the coordinate matrix $A^t_p$ we obtain
$A^t_p\delta_{v_i}=\sum_{\lambda\in v_i\Lambda^p} \delta_{s(\lambda)}$. Applying this to
each $w_i$ we get
$$\sum_{i=1}^n\sum_{\lambda\in v_i\Lambda^p}  \delta_{s(\lambda)}=\sum_{i=1}^m\sum_{\lambda\in w_i\Lambda^q} \delta_{s(\lambda)}.$$
If $\delta_u+\delta_v+\delta_w=\delta_x+\delta_y+\delta_z$ then, by definition,
$\tau(u)+\tau(v)+\tau(w)=\tau(x)+\tau(y)+\tau(z)$. Hence
$$\sum_{i=1}^n\sum_{\lambda\in v_i\Lambda^p} \tau(s(\lambda))=\sum_{i=1}^m\sum_{\lambda\in w_i\Lambda^q} \tau(s(\lambda)).$$
Because $\tau$ is a graph trace we have $\sum_{\lambda\in v_i\Lambda^p}
\tau(s(\lambda))=\tau(v_i)$ and similarly for $w_i$. Thus
$\sum_{i=1}^n\tau(v_i)=\sum_{i=1}^m \tau(w_i)$.
Now suppose that $f\approx_{\Lambda} g$. Select $x_i,y_i$ witnessing this. Then
\[
\sum_{v \in \Lambda^0} f(v)\tau(v)
	= \sum_i \sum_{v \in \Lambda^0} x_i(v)\tau(v)
	= \sum_i \sum_{v \in \Lambda^0} y_i(v)\tau(v)
	= \sum_{v \in \Lambda^0} g(v)\tau(v).
\]
Thus there is a well-defined additive map $\beta$ as claimed. Since $\beta$
respects $+$ and $\leq$ and satisfies $\beta(0) = 0$, it is a state on $S(\Lambda)$.
Since $\tau$ is faithful $\beta$ is nontrivial.
\end{proof}

\begin{remark}
Proposition~\ref{prop.properties.dichotomy.four} can also be recovered from our
previous $C^*$-algebraic results Proposition~\ref{prop.main.2} and
Theorem~\ref{thm.stably.finite}, but again we have elected to give a direct
algebraic proof instead. The algebraic properties \eqref{prop.properties.dichotomy.four.(1)}--\eqref{prop.properties.dichotomy.four.(5)} of
Proposition~\ref{prop.properties.dichotomy.four} are equivalent to the $C^*$-algebraic
property of $C^*(\Lambda)$ being stably finite.
\end{remark}

\begin{remark}
Let $\Lambda$ be a row-finite and cofinal $1$-graph with no sources. Then $S(\Lambda)$ is
automatically either stably finite or purely infinite. Indeed, if $S(\Lambda)$ is not
stably finite then $\textrm{im} (1-A^t_{e_1}) \cap \NN\Lambda^0\neq \{0\}$ by
Proposition~\ref{prop.properties.dichotomy.four}. Consequently, using \cite[Lemma
4.2]{MR3507995} and then Proposition \ref{prop.properties.dichotomy.merged} it follows
that $\Lambda$ contains a cycle and hence $S(\Lambda)$ is purely infinite.
\end{remark}

Not much is known about the order structure of $S(\Lambda)$ in general. In particular, it
seems unclear whether $S(\Lambda)$ is almost unperforated when $S(\Lambda)$ is stably
finite. So we finish by recording some partial results about when $S(\Lambda)$ is almost
unperforated.

Let $S$ be an (abelian) semigroup with identity. We say that a subset $X\subseteq S$ is
\emph{cancellative} if for each $a,b,c\in X$ the equality $a + c = b + c$ implies $a =
b$. Let $S^*$ denote the set of nonzero elements of $S$.  A nonzero element $x$ in $S$ is
an \emph{atom} (or \emph{irreducible}) if, whenever $a,b\in S$ satisfy $x = a + b$, we
have either $a = 0$ or $b = 0$. Recall that $S$ is almost unperforated if whenever $x, y
\in S$ and $n, m \in \NN$ are such that $nx \leq  my$ and $n > m$, one has $x\leq y$.

\begin{lemma}\label{lem.S.almost.unper.original}
Let $\Lambda$ be a row-finite and cofinal $k$-graph with no sources. Then $S(\Lambda)$ is
almost unperforated if and only if it satisfies precisely one of the following:
\begin{enumerate}
\item\label{lem.S.almost.unper.original.i} $S(\Lambda)$ contains an atom and $S(\Lambda)\cong\NN=\{0, 1, \dots\}$;
\item\label{lem.S.almost.unper.original.ii} $S(\Lambda)$ is atomless and contains an infinite element, $S(\Lambda)^*$
    is cancellative, and $x \leq y$ for all nonzero $x, y$ in $S(\Lambda)$; or
\item\label{lem.S.almost.unper.original.iii} $S(\Lambda)$ is atomless and contains no infinite element, and is cancellative
    and almost unperforated.
\end{enumerate}
\end{lemma}
\begin{proof}
It is evident that each of the properties \eqref{lem.S.almost.unper.original.i}--\eqref{lem.S.almost.unper.original.iii} implies that $S(\Lambda)$ is almost
unperforated: the semigroup $\NN$ is almost unperforated and if $x \leq y$ for all
nonzero $x, y$ in $S(\Lambda)$, then $S(\Lambda)$ is unperforated (because $S(\Lambda)$
is conical by Lemma~\ref{lemma.properties.conical}\eqref{lemma.properties.conical.i}, so $nx\leq 0$ in $S(\Lambda)$
implies $x=0$).

Conversely suppose that $S(\Lambda)$ is almost unperforated. We show $S(\Lambda)$
possesses one of the properties~\eqref{lem.S.almost.unper.original.i}, \eqref{lem.S.almost.unper.original.ii}~or~\eqref{lem.S.almost.unper.original.iii}. Clearly either $S(\Lambda)$ contains an
atom or not. Assume $S(\Lambda)$ contains an atom. Since $S(\Lambda)$ is a conical simple
refinement (i.e., satisfies the Riesz refinement property) semigroup with identity,
\cite[Corollary~2.7]{MR2806681} ensures $S(\Lambda)\cong\NN$, giving \eqref{lem.S.almost.unper.original.i}. Now assume that
$S(\Lambda)$ is atomless. Using Lemma~\ref{lemma.properties.conical}, $S(\Lambda)$ is a
conical simple refinement semigroup with identity. So the implication (i)$\Rightarrow$(v) of
\cite[Proposition~3.8]{MR2806681} implies that $S(\Lambda)$ satisfies weak comparability
in the sense of \cite{MR1301484}, which allows us to apply \cite[Proposition~1.4(b) and Corollary~1.8]{MR1301484} below. Clearly either $S(\Lambda)$ contains an infinite
element or not. If $S(\Lambda)$ contains an infinite element $u$, then $a + u = u$ for
some $a\in S(\Lambda)^*$. By \cite[Proposition~1.4(b)]{MR1301484}, $S(\Lambda)^*$ is
cancellative and $x \leq y$ for any nonzero $x, y$ in $S(\Lambda)$, giving~\eqref{lem.S.almost.unper.original.ii}. So it
remains to show that if $S(\Lambda)$ does not contain any infinite element then it is
cancellative. We certainly have that $u + a = u$ implies $a = 0$ for all $a,u\in
S(\Lambda)$, so \cite[Corollary 1.8]{MR1301484} shows that $S(\Lambda)$ is cancellative.
\end{proof}

\begin{remark}
It is worth noting that in Lemma~\ref{lem.S.almost.unper.original}\eqref{lem.S.almost.unper.original.ii}, we cannot expect
$S(\Lambda)$ itself to be cancellative: Since $\Lambda$ is cofinal, $S(\Lambda)$ is
simple. So if it contains a nonzero infinite element $x$, this $x$ must be properly
infinite, and so $2x \leq x$, say $x = 2x + z = x + (x + z)$. Cancellativity would force
$x + z = 0$ and then $x = 0$ because $S(\Lambda)$ is conical, contradicting that $x$ was
a nonzero infinite element. Since $0 \not\in S(\Lambda)^*$, this obstruction disappears
if we only ask that $S(\Lambda)^*$ is cancellative.
\end{remark}

\begin{remark}\label{rem.almost.cancel}
It follows from Lemma~\ref{lem.S.almost.unper.original} that for any row-finite and
cofinal $k$-graph $\Lambda$ with no sources, if $S(\Lambda)$ is almost
unperforated then $S(\Lambda)^*$ is cancellative.
\end{remark}

\begin{prop}\label{prop.properties.dichotomy.five}
Let $\Lambda$ be a row-finite and cofinal $k$-graph with no sources. Then the following are equivalent:\begin{enumerate}
\item\label{prop.properties.dichotomy.five.(1)} $S(\Lambda)$ is purely infinite;
\item\label{prop.properties.dichotomy.five.(2)}  $x\leq y$ for all nonzero $x,y\in S(\Lambda)$;
\item\label{prop.properties.dichotomy.five.(3)}  $S(\Lambda)$ contains an infinite element and $S(\Lambda)$ is almost
    unperforated; and
\item\label{prop.properties.dichotomy.five.(4)}  $S(\Lambda)$ contains an infinite element and $S(\Lambda)^*$ is
    cancellative.
\end{enumerate}
\end{prop}
\begin{proof}
For \eqref{prop.properties.dichotomy.five.(1)}$\Rightarrow$\eqref{prop.properties.dichotomy.five.(2)} fix any nonzero $x,y\in S(\Lambda)$. Using Lemma
\ref{lemma.properties.conical}, $S(\Lambda)$ is simple, so we can select $n\geq 1$ such
that $x\leq ny$. Since $S(\Lambda)$ is purely infinite we know $ny \leq y$. Hence $x \leq
ny \leq y$.

\eqref{prop.properties.dichotomy.five.(2)}$\Rightarrow$\eqref{prop.properties.dichotomy.five.(3)}. Suppose that $x \leq y$ for all nonzero $x, y$ in $S(\Lambda)$.
Lemma~\ref{lemma.properties.conical}\eqref{lemma.properties.conical.i} shows that $S(\Lambda)$ is conical, and so $nx
\le 0$ implies $x = 0$. So if $nx \le my$ with $n > m$, then either $x = 0$, in which
case $x \le y$ trivially, or else $nx \not\le 0$ so both $x$ and $y$ are nonzero, and
then $x \le y$ by hypothesis.

For \eqref{prop.properties.dichotomy.five.(3)}$\Rightarrow$\eqref{prop.properties.dichotomy.five.(4)} see Remark~\ref{rem.almost.cancel}.
Finally for \eqref{prop.properties.dichotomy.five.(4)}$\Rightarrow$\eqref{prop.properties.dichotomy.five.(1)}, assume that~\eqref{prop.properties.dichotomy.five.(1)} fails. Reusing the proof of
Lemma~\ref{lem.A.conical}\eqref{lem.A.conical.iii} gives a nontrivial state on $S(\Lambda)$: for this type of
proof cancellation in $S(\Lambda)^*$ suffices, see \cite[p.~95]{Sie}. Using
Proposition~\ref{prop.properties.dichotomy.four} it then follows that $S(\Lambda)$ is
stably finite. Hence by definition of stably finite we see that~\eqref{prop.properties.dichotomy.five.(4)} fails.
\end{proof}

\begin{remark}
We do not know whether $S(\Lambda)^*$ is cancellative whenever $S(\Lambda)$  contains an
infinite element. There are a number of properties, defined in
\cite{MR2806681, MR2874927}, each equivalent to cancellation in $S(\Lambda)^*$ when
$S(\Lambda)$ contains an infinite element and $\Lambda$ is cofinal. Without giving any
explicit proof let us mention that some of these are the corona factorisation property,
the strong corona factorisation property, property (QQ), $n$-comparison for some $n$,
$n$-comparison for any $n$, weak comparability, strict unperforation, almost
unperforation, or the property that $(n+1)x\leq nx$ implies $2x\leq x$ for any $x\in
S(\Lambda), n\in \NN$. For all we know they might all be automatic when $S(\Lambda)$
contains an infinite element and $\Lambda$ is cofinal. If this happens, then any simple
$k$-graph $C^*$-algebra is either stably finite or purely infinite.
\end{remark}

\end{document}